\documentclass[absolute]{mymathart}
\usepackage[utf8]{inputenc}
\usepackage{amsmath}
\usepackage{amsfonts}
\usepackage{amssymb}
\usepackage{amsthm}
\usepackage{graphicx}
\usepackage{caption}
\usepackage[page, toc]{appendix}
\usepackage{transparent}

\usepackage{pdflscape}

\usepackage{hyperref}

\usepackage{marginnote}
\usepackage{xcolor}

\DeclareMathOperator{\dist}{dist}

\numberwithin{equation}{section}

\newcommand*{\defeq}{\mathrel{\vcenter{\baselineskip0.5ex \lineskiplimit0pt
                     \hbox{\scriptsize.}\hbox{\scriptsize.}}}%
                     =}

\newcommand{\HH}{\mathbb{H}}
\newcommand{\re}{\operatorname{Re}}
\newcommand{\im}{\operatorname{Im}}

\usepackage{mathrsfs}
\usepackage{xtab}
\usepackage{cite}
\usepackage{tikz}
\usepackage{epigraph}
\usepackage{hyperref}
\usepackage{listings}
\usepackage{subfigure}
\usepackage{caption}
\usepackage{mathtools}

\usepackage{enumerate}
\usepackage{calc}
\usepackage{enumitem}

\newtheorem{prop}{Proposition}[section]
\newtheorem{theorem}[prop]{Theorem}
\newtheorem{lemma}[prop]{Lemma}

\newtheorem{cor}[prop]{Corollary}

\newtheorem{defn}[prop]{Definition}

\newtheorem{standingassumption}[prop]{Standing Assumption}

\newcommand{\abs}[1]{\lvert #1 \rvert}
\newcommand{\C}{\ensuremath{\mathbb{C}}}
\newcommand{\Chat}{\ensuremath{\widehat{\mathbb{C}}}}
\theoremstyle{definition}

\newcommand{\classB}{\mathcal{B}}
\newcommand{\Blog}{\classB_{\log}}
\newcommand{\Blogp}{\Blog^{\operatorname{p}}}
\newcommand{\eps}{\varepsilon}

 \newtheoremstyle{claimstyle}%
   {}%             space above
   {}%             space below
   {\normalfont}%     body font
   {}%                indent
   {\itshape}%        header font
   {.}%               punctuation
   { }%               space after head
   {\thmnote{#3}}%    typeset note only.

\newenvironment{subproof}{\begin{proof}}{%
               \end{proof}}

\theoremstyle{claimstyle}
\newtheorem*{varclaim}{}

\newenvironment{claim}[1][Claim]{\begin{varclaim}[#1]}{\end{varclaim}}

\newcommand{\DD}{\mathbb{D}}
\newcommand{\Z}{\mathbb{Z}}

\title[Slow-growing counterexamples]{Slow-growing counterexamples to the strong Eremenko conjecture}
\author{Andrew P. Brown}
\address{Lincoln College, University of Oxford}
\email{linc5765@ox.ac.uk}
\begin{document}

\begin{abstract}
Let $f\colon\C\to\C$ be a 
transcendental entire function. 
%, and consider the \emph{escaping set}
%\[ I(f) \defeq \{z\in\C\colon f^n(z)\to\infty\}. \]
In 1989, Eremenko asked the following question concerning the set $I(f)$ of points that tend
to infinity under iteration: 
 can every point of $I(f)$ be joined to $\infty$ by a curve in $I(f)$? 
 This is known as the \emph{strong Eremenko conjecture} and was disproved in 2011 by Rottenfu{\ss}er, R\"uckert, Rempe and Schleicher by construction of a counterexample.
% In their example, the Julia set $J(f)$ contains $I(f)$, and every connected component $C$ of $J(f)$ is a closed, unbounded set containing no curve to infinity. 
The function has relatively small infinite order: it can
  be chosen such that $\log \log \,\lvert f(z)\rvert = (\log \lvert z\rvert)^{1+o(1)}$ as $f(z)\to \infty$. Moreover, $f$ belongs to the \emph{Eremenko--Lyubich class $\classB$}. When a function belongs to this class, we can study the function via a \textit{logarithmic change of coordinates}. In this frame of coordinates, we are able to study the function via the \textit{tracts} that arise which are Jordan domains with unbounded real part. The key feature of the tracts in the counterexample of Rottenfu{\ss }er et al is that of large \textit{wiggling}  sections. In this article we adapt the tracts used by Benitez and Rempe in order to deduce the existence of counterexample functions $f \in \classB$ satisfying certain growth properties. We consider how slow such an $f$ may grow. 
        
        Suppose that $\Theta\colon [t_0,\infty)\to [0,\infty)$ is a function such that $\Theta(t) \to 0$ and 
      \[ (\log t)^{\Theta(\log t)}\Theta(t) \to \infty \quad\text{ as $t\to \infty$} \] 
        along with a certain regularity assumption. Then there exists a counterexample $f\in\classB$
           as above such that 
           \[ \log \log \abs{f(z)} = O\bigl( (\log \abs{z})^{1 + \Theta(\log\abs{z})}\bigr) \quad\text{as $f(z) \to\infty$}. \]
    The hypotheses are satisfied, in particular, for $\Theta(t) = 1/(\log \log t)^{\alpha}$, for any $\alpha>0$. 
\end{abstract}

\maketitle

\section{Introduction}\label{sec:intro}
Let $f\colon \C\to\C$ be a transcendental entire function, that is, an entire function that is not a polynomial. The \emph{escaping set} of $f$ is defined to be 
 \[ I(f) \defeq \{z\in\C\colon f^n(z)\to\infty \text{ as $n\to\infty$}\}, \]
   where $f^n$ denotes the $n$-th iterate of $f$. An important fact is that for every transcendental entire function $f$, $I(f)$ is non-empty~\cite[Theorem 1]{erem}.

In 1926, Fatou~\cite[p.~369]{fat} had already noticed the existence of curves contained within the escaping sets of certain transcendental entire functions and asked whether this is true more generally. In 1989, Eremenko~\cite{erem} was the first to study the set $I(f)$ in general. In particular, he made Fatou's question more precise, stating ``It is plausible that the set $I(f)$ always has the following property: every point $z \in I(f)$ can be joined with $\infty$ by a curve in $I(f)$.''~\cite[p.~344]{erem}

This is known as the \emph{strong Eremenko conjecture} and was answered in the negative by Rottenfu\ss er, R\"uckert, Rempe and Schleicher by means of a counterexample; see Theorem~1.1 of~\cite{rrrs}.
%
%It should also be mentioned that in~\cite{erem}, Eremenko also said ``It is plausible that the set $I(f)$ has not bounded connected components.'' Which came to be known as the \emph{Eremenko conjecture}. This conjecture similarly inspired much work in the field up until it was disproved by Mart\`{i}-Pete, Rempe, and Waterman in~\cite{mav}. We will not be studying this conjecture in this thesis.

An entire function $f$ belongs to the 
  \emph{Eremenko--Lyubich class} $\classB$ if its set of critical and asymptotic values is bounded (see Section~\ref{sec:classB}). 
  Moreover, $f$ has \emph{finite order} if 
    \begin{equation}\label{finiteorder} \log \log \lvert f(z)\rvert  = O( \log \abs{z}) \quad\text{as $\lvert z \rvert \to\infty \text{ with } \abs{f(z)} > e$}, \end{equation}
      and \emph{infinite order} otherwise. By~\cite[Theorem~1.2]{rrrs}, the strong Eremenko conjecture holds for 
      any $f\in\classB$ of finite order. 

The counterexample $f$ in~\cite[Theorem~1.1]{rrrs} belongs to the class $\classB$, so its order must be infinite; 
  it can be checked that 
      \begin{equation}\label{eqn:powergrowth}
          \log \log \, \abs{f(z)} = O\bigl( (\log \abs{z})^{M}\bigr)  \quad\text{as $\lvert z \rvert \to\infty$ with  $\abs{f(z)} > e$};\end{equation}
   for some $ M >1$~\cite[Proposition~8.1]{rrrs}.
%    In~\cite{rrrs}, the authors also discuss, without giving all details of the analysis, 
%   certain modifications of the construction. Using these, $f$ can be chosen to satisfy~\eqref{eqn:powergrowth} for $M$ arbitrarily close to 1~\cite[Proposition~8.2]{rrrs}, and even 
    In~\cite{rrrs}, the authors also discuss, without giving all details of the analysis, 
   certain modifications of the construction. Using these, $f$ can be chosen so that~\eqref{eqn:powergrowth} holds for 
     \emph{every} $M>1$; in other words, 
      \begin{equation}\label{eqn:littleoq} \log \log \abs{f(z)} = O\bigl( (\log \abs{z})^{1+ o(1)}\bigr)  \quad\text{as $\lvert z \rvert  \to\infty$ with  $\abs{f(z)} > e$};\end{equation}
   see~\cite[Proposition~8.3]{rrrs} and~\cite[Theorem~1.10]{lasse13}.
   
   What is not discussed in~\cite{rrrs} and~\cite{lasse13} is the exact behaviour of the $o(1)$ term in~\eqref{eqn:littleoq}. Since $x^{1/\log x} = e$,  it can be checked that \[ \left( \log \lvert z \rvert \right)^{1 + \frac{1}{\log \log \lvert z \rvert }} = e \log \lvert z \rvert.\] Therefore, if a function $f \in \classB$ satisfies \[ \log \log \lvert f(z) \rvert = O \left( \left( \log \lvert z \rvert \right)^{1 + \frac{1}{\log \log \lvert z \rvert }}\right)   \quad\text{as $\lvert z \rvert \to\infty$ with  $\abs{f(z)} > e$},\] then  $f$ would be of finite order and hence the strong Eremenko conjecture would hold. We are faced with the following question: In order for a function $f \in \classB$ to be a counterexample to the strong Eremenko conjecture satisfying the relation~\eqref{eqn:littleoq}, what conditions may (or must) the $o(1)$ term satisfy? A shorter question that we can ask of such an $f$ is: How slow can you grow? This article provides an answer.

  We will now try to make this more precise by introducing some notation. Suppose that $\Theta$ is a positive decreasing function (defined for all sufficiently large
  positive real numbers) such that $\Theta(t)\to 0$ as $t\to \infty$. 
     When does there exist a counterexample $f\in\classB$ to the strong
   Eremenko conjecture such that 
     \begin{equation}\label{eqn:Thetagrowth}
        \log \log \abs{f(z)} = O\bigl( (\log \abs{z})^{1 + \Theta(\log\abs{z})}\bigr) \quad\text{as $\lvert z \rvert  \to\infty$ with $\abs{f(z)} > e$?}
     \end{equation}
  As $f$ must have infinite order, we may suppose that
    \begin{equation}\label{eqn:Thetalowerbound}
       \Theta(t)\log t \text{ is strictly increasing and }\Theta(t)\log t \to \infty \text{ as } t \to \infty.
    \end{equation}
%  Let us also require the following regularity condition: 
%   \begin{equation}\label{eqn:Thetaregularity}
%      \Theta(t^2)/\Theta(t) \to 1 \quad\text{as $t \to \infty$}. 
%   \end{equation}

\begin{theorem}\label{paracxple}
 Let $\Theta$ be a positive decreasing function of one real variable
    such that $\Theta(t)\to 0$ as $t\to \infty$, and such that~\eqref{eqn:Thetalowerbound} holds.
      Suppose that, additionally, the function $\Theta$ satisfies 
    \begin{equation}\label{eqn:deryck} (\log t)^{\Theta(\log t)}\Theta(t) \to \infty \quad\text{as $t\to \infty$}. \end{equation}
  Then there exists $f \in \classB$ satisfying~\eqref{eqn:Thetagrowth} such that 
   $I(f)$ contains no curve to infinity. 
\end{theorem}

It is plausible that the condition~\eqref{eqn:deryck} is essentially optimal, in the following sense:
  If $f\in\classB$ and for that 
  \begin{equation}\label{eqn:deryck2}
   (\log t)^{\Theta(\log t)}\Theta(t) \to 0
    \quad\text{as $t\to \infty$}, \end{equation}
    then the strong Eremenko conjecture holds for $f$.

It is easy to check that, for any $\alpha>0$, the function
    \[ \Theta(t) \defeq \frac{1}{(\log \log t)^{\alpha}} \]
  satisfies the hypotheses of Theorem~\ref{paracxple}. We perform this check in Section~\ref{loglogcheck}. 
  
% Let us remark on the condition~\eqref{eqn:Thetaregularity}. It is a regularity condition, but also implies
%   that $\Theta$ tends to zero more slowly than $1/(\log t)^{\alpha}$ for any $\alpha>0$. 
%   It is easy to check that the latter functions do not satisfy~\eqref{eqn:deryck} (in fact,
%   they satisfy~\eqref{eqn:deryck2}). Thus, requiring the regularity condition~\eqref{eqn:Thetaregularity} in
%   the presence of~\eqref{eqn:deryck} does not impose additional constraints on how fast $\Theta$ may tend to $0$. 

\subsection{Structure of the paper and ideas behind the construction} \hfill \\ The overview provided here is intended to be give the reader an intuitive and visual understanding of the work carried out in the paper before the full details are presented in the subsequent sections.

%A very broad view of the work can be described in the following way. At no point do we deal directly with any class $\classB$ functions. Instead, we construct a conformal isomorphism $\widehat{F}$ before we are able to deduce the existence of a class $\classB$ function $f$ for which it has a logarithmic transform $F$ with the same domain of definition as $\widehat{F}$ on which $F$ and $\widehat{F}$ are \emph{approximately} equal. We can express the relation by the following extreme abuse of notation $\widehat{F} \approx F = \log f$.

At no point do we deal directly with any class $\classB$ functions. Instead, we construct a conformal isomorphism $F$ before we are able to deduce the existence of a class $\classB$ function $f$ for which it has a logarithmic transform $\widehat{F}$ with the same domain of definition as $F$ on which $F$ and $\widehat{F}$ are \emph{approximately} equal. We can express this with the following abuse of notation $F \approx \widehat{F} = \log f$.

It will do the reader no harm to be parachuted into the exact situation in which almost all of the work is undertaken takes place as the domains which we spend most of our time studying are very straightforward to consider. The precise details of the setup are covered in sections~\ref{sec:background} and~\ref{sec:classB}. 

We define \emph{tracts} in Section~\ref{sec:classB} but we can informally discuss them and the differences between those used in~\cite{rrrs} and the tracts seen in~\cite{taniathesis} and~\cite{tania}. Although the latter two sources do not deal with the strong Eremenko conjecture directly; they address questions and constructions regarding the topology of Julia sets. We will be adapting the tracts in~\cite{tania}, as well as several other results, in order to achieve our goal by finally replicating the approach used in~\cite{rrrs} by using new tracts, new estimates, and new approximation methods. Tracts naturally arise for functions that have a bounded set of singular values which allows us to work in \emph{logarithmic coordinates}. We spend our time working in logarithmic coordinates, constructing the tracts to determine behaviours of the corresponding isomorphism that will map the tract into the right half-plane $\HH$ (see Section~\ref{sec:classB}) before using approximation methods to deduce the existence of an entire function satisfying the conditions we desire.

A tract, $T$, is a simply connected Jordan domain of height, for any fixed real part, less than $2\pi$, with real part unbounded from above, that is a covering of the right half-plane $\HH \defeq \lbrace z \in \C \colon \re z > 0 \rbrace$ via a conformal isomorphism $F$. If we ensure that $\overline{T} \subset \HH$ then we can consider iteration and preimages of $T$ under $F$. Figure~\ref{fig:tracttoplane} illustrates this where a copy of $T$ with dashed lines has been added to the half-plane for emphasis. The tracts that we study, following the examples of~\cite{rrrs} and ~\cite{tania}, have ``wiggling'' sections where the distance between the endpoints increases drastically at each step. For tracts that have sufficiently good enough 	``wiggling'', we are able to deduce certain behaviours of the corresponding conformal isomorphism. The key feature we want is to have the \emph{vertical} geodesics (see Chapter~\ref{sec:tracts}) $C_0$ and $\dot{C}_0$ in $T$ map to the semi-circular geodesics $F(C_0)$ and $F(\dot{C}_0)$ of $\HH$ that conveniently lie ``across'' the next wiggle section with real parts smaller than $C_1$ and $\dot{C}_1$ in the manner of Figure~\ref{fig:tracttoplane}. We then also want $F(C_1)$ and $F(\dot{C}_1)$ to map similarly over the following wiggle section, having real parts strictly smaller than $C_2$ and $\dot{C}_2$, and so on. Rather than drawing multiple copies of the tract, we draw one copy of the tract as in Figure~\ref{rrrstract} to display this. These tracts are similar to the ones used~\cite{rrrs}. The specific details of how this furnishes us with a counterexample are discussed in Section~\ref{sec:cxples}.

%\begin{figure}[htbp]
%\centering
%\def\svgwidth{1\textwidth}
%\input{tracttoplane.pdf_tex}
%\caption{A tract mapping onto $\HH$.}
%\label{fig:tracttoplane}
%\end{figure}

\begin{figure}[hpb]
\centering
\def\svgwidth{1\textwidth}
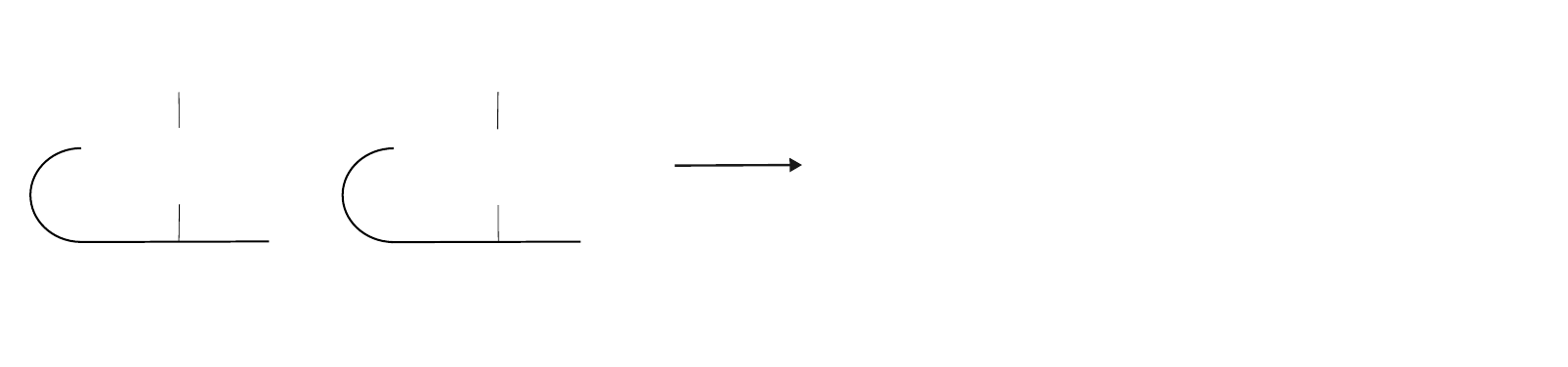
\caption{A tract mapping onto $\HH$.}
\label{fig:tracttoplane}
\end{figure}

\begin{figure}[htbp]
\centering
\def\svgwidth{1\textwidth}
%% Creator: Inkscape 1.3 (0e150ed6c4, 2023-07-21), www.inkscape.org
%% Accompanies image file 'rrrstract.pdf' (pdf, eps, ps)
\begingroup%
  \makeatletter%
  \providecommand\color[2][]{%
    \errmessage{(Inkscape) Color is used for the text in Inkscape, but the package 'color.sty' is not loaded}%
    \renewcommand\color[2][]{}%
  }%
  \providecommand\transparent[1]{%
    \errmessage{(Inkscape) Transparency is used (non-zero) for the text in Inkscape, but the package 'transparent.sty' is not loaded}%
    \renewcommand\transparent[1]{}%
  }%
  \providecommand\rotatebox[2]{#2}%
  \newcommand*\fsize{\dimexpr\f@size pt\relax}%
  \newcommand*\lineheight[1]{\fontsize{\fsize}{#1\fsize}\selectfont}%
  \ifx\svgwidth\undefined%
    \setlength{\unitlength}{656.70456293bp}%
    \ifx\svgscale\undefined%
      \relax%
    \else%
      \setlength{\unitlength}{\unitlength * \real{\svgscale}}%
    \fi%
  \else%
    \setlength{\unitlength}{\svgwidth}%
  \fi%
  \global\let\svgwidth\undefined%
  \global\let\svgscale\undefined%
  \makeatother%
  \begin{picture}(1,0.59450209)%
    \lineheight{1}%
    \setlength\tabcolsep{0pt}%
    \put(0,0){\includegraphics[width=\unitlength,page=1]{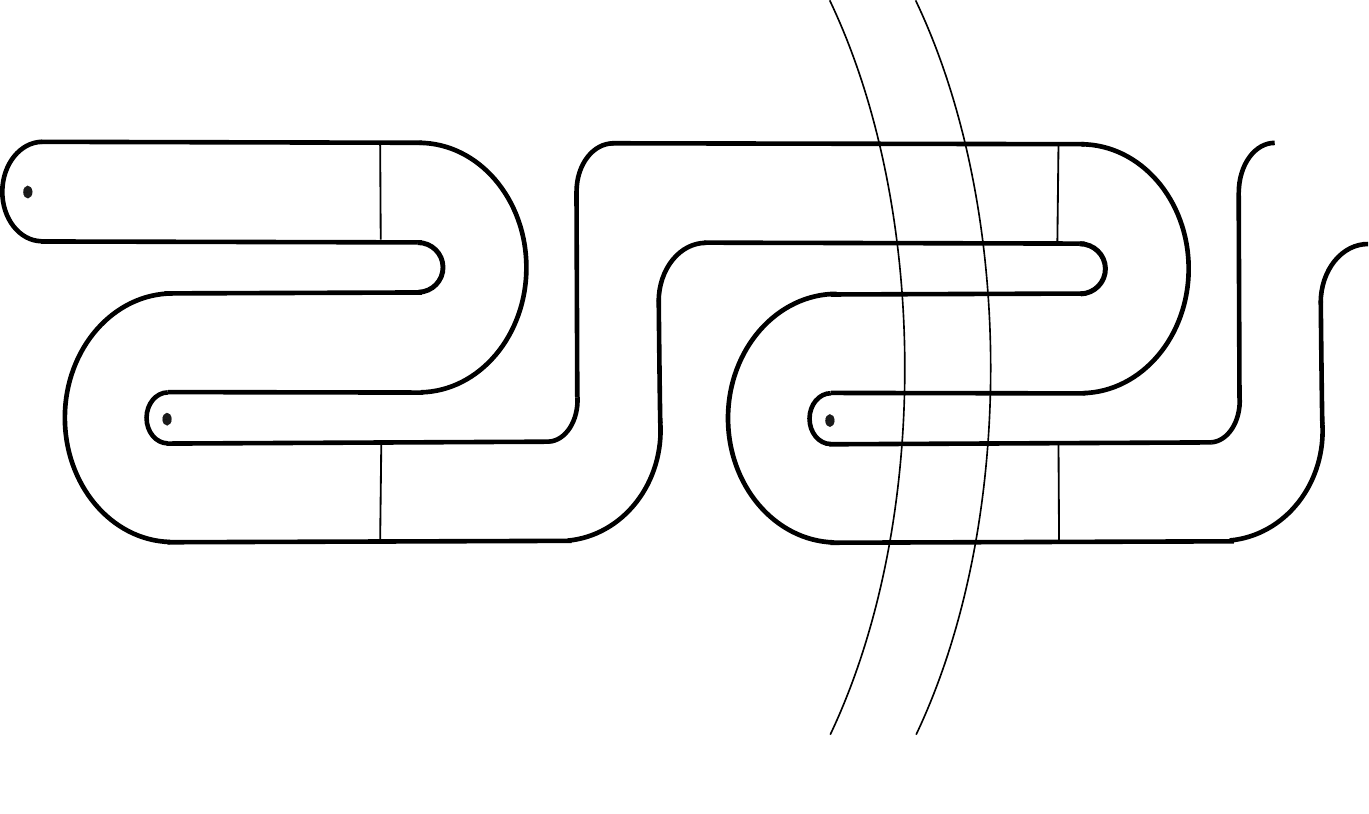}}%
    \put(0.22104373,0.52262278){\color[rgb]{0.10196078,0.10196078,0.10196078}\makebox(0,0)[lt]{\lineheight{1.25}\smash{\begin{tabular}[t]{l}$C_0$\end{tabular}}}}%
    \put(0.22517478,0.15123715){\color[rgb]{0.10196078,0.10196078,0.10196078}\makebox(0,0)[lt]{\lineheight{1.25}\smash{\begin{tabular}[t]{l}$\dot{C}_0$\end{tabular}}}}%
    \put(0.52598265,0.52421115){\color[rgb]{0.10196078,0.10196078,0.10196078}\makebox(0,0)[lt]{\lineheight{1.25}\smash{\begin{tabular}[t]{l}$F(C_0)$\end{tabular}}}}%
    \put(0.62985435,0.01048318){\color[rgb]{0.10196078,0.10196078,0.10196078}\makebox(0,0)[lt]{\lineheight{1.25}\smash{\begin{tabular}[t]{l}$F(\dot{C}_0)$\end{tabular}}}}%
    \put(0.75923279,0.51278726){\color[rgb]{0.10196078,0.10196078,0.10196078}\transparent{0.99607801}\makebox(0,0)[lt]{\lineheight{1.25}\smash{\begin{tabular}[t]{l}$C_1$\end{tabular}}}}%
    \put(0.77048491,0.14484432){\color[rgb]{0.10196078,0.10196078,0.10196078}\transparent{0.99607801}\makebox(0,0)[lt]{\lineheight{1.25}\smash{\begin{tabular}[t]{l}$\dot{C}_1$\end{tabular}}}}%
  \end{picture}%
\endgroup%

\caption{A tract as seen in~\cite{rrrs}.}
\label{rrrstract}
\end{figure}

Although the boundary of the tracts seen in Figure~\ref{fig:tracttoplane} and Figure~\ref{rrrstract} can be written down explicitly, a simplification made in~\cite{tania} and~\cite{taniathesis} makes the exact description of the tract decidedly less cumbersome and provides us with the ``rectangular'' tracts seen in Figure~\ref{tractnogate}.

\begin{figure}[htbp]
\centering
\def\svgwidth{1\textwidth}
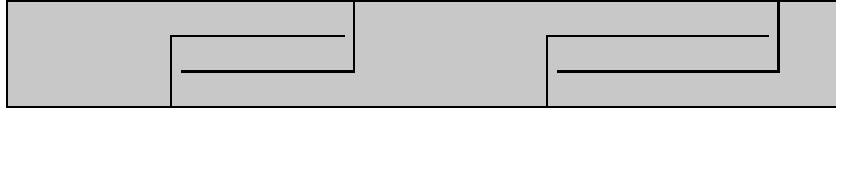
\caption{A tract as seen in~\cite{tania}.}
\label{tractnogate}
\end{figure}

 This allows us to be able to completely define our tracts with only two sequences of real numbers, $(r_j)_{j=0}^\infty$ and $(R_j)_{j=0}^\infty$, which denote the endpoints of the wiggle sections.
% We no longer have a Jordan domain but this problem is solved by considering the preimage of the right half-plane with real part strictly greater than $1$. This gives us a tract similar to those in Figure~\ref{fig:tracttoplane} and Figure~\ref{rrrstract} once again, allowing us to invoke theorems that require us to have Jordan domains.

Before making any attempts to determine orders of growth, we spend Section~\ref{sec:cxples} showing that the illustrated mapping behaviours, which can be summarised in a chain of inequalities regarding the real parts of vertical geodesics, their image under $F$ as well as the end points of the subsequent wiggle section and the corresponding pair of vertical geodesics in a manner represented by Figure~\ref{tractmap}.

It transpires that the relationships between the endpoints of the wiggling sections are instrumental in determining the order of growth of the function. This analysis is carried out in the same manner as~\cite[Proposition 8.1]{rrrs} in Section~\ref{sec:growthgeom}. If we restrict ourselves to only using tracts like the kind in Figure~\ref{tractnogate} we find that we cannot guarantee both the function derived having the properties of being a counterexample to the strong Eremenko conjecture as well as having the low order of growth desired. The growth estimate available from~\cite[Proposition 7.2]{tania} is unfortunately not enough to remedy this either.

The key to solving our problem is to include small openings (``gates'') in our tracts at the first turning point of the wiggle section as seen in Figure~\ref{tract} and to find a corresponding growth estimate, Theorem~\ref{growth}, that accounts for these gates. The result is effectively a refinement of~\cite[Proposition 7.2]{tania} as the proofs are very similar. This means our tracts depend on three sequences of real numbers; $(r_j)_{j=0}^\infty$ and $(R_j)_{j=0}^\infty$, which determine where the endpoints of the wiggles are placed, and then $(\eps_j)_{j=0}^\infty$ which determines the size of the gate openings at each wiggle.

\begin{figure}[htbp]
\centering
\def\svgwidth{1\textwidth}
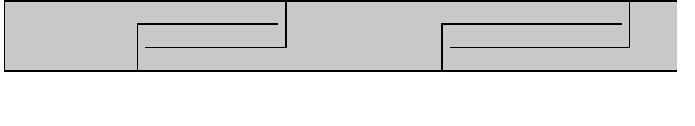
\caption{The tract $T$ defined in Definition~\ref{defn:ourtracts} with `gates'.}
\label{tract}
\end{figure}

Should there exist any curve of points that joins a finite point to $\infty$ which remains in $T$ for all iterates of $F$, this curve must pass through the small gate openings. Although the calculation is carried out in Section~\ref{sec:growth}, a quick glance at the standard estimate ~\eqref{eqn:standardestimate} tells us that the section of any path passing through these gates must contribute significantly to the overall growth of the function since our Euclidean distance to the boundary of the tract is small as the path passes through the gate. Since our aim is to reduce the order of growth as seen in the counterexample of \cite{rrrs}, this may seem counterintuitive. By having the gate placed at the first turning point of the wiggle section, near $R_j$, the contribution from the subsequent section of curve returning to a place of significantly smaller real part, near $r_j$, before coming back manages to dampen this effect enough so that it doesn't force our resulting order of growth to be too large.

Section~\ref{sec:growth} gives us a valuable tool in Theorem~\ref{growth} that shows us how the geometry of the tract can be linked to the growth of the associated isomorphism, that is, we are able to compare $\log \abs{F(z)}$ with $R_j$ when $r_j < \re z < R_j$, i.e. $z$ is in the $j$-th wiggle. Our aim is to prescribe an order of growth on the function $F$ through the geometry of the tract. In Section~\ref{sec:growthgeom} we use this to decide on which relations between $(r_j)_{j=0}^\infty$ and $(R_j)_{j=0}^\infty$ to impose so as to achieve the desired order of growth. In this stage we find that if we make the following relations: \[ \log r_{n+1} \defeq \varphi(R_n) -1 \text{ and } \psi(r_n) \defeq R_n, \] for functions $\varphi$ and $\psi$ guaranteeing slow infinite order of growth; we find that we have to compare them against each other and by allowing $\varphi$ to grow faster than $\psi$ (details given in Section~\ref{sec:est}), we achieve the growth relation \[ \log \re F(z) = O( \varphi(\re z)) \text{ as } \lvert z \rvert \to \infty.\] This allows us to define $\Phi(t)$ as a function in one real variable that tends to zero so that $\varphi(t) = t^{1 + \Phi(t)}$. The particulars about how and why we choose $\Phi$ are discussed in detail in Section~\ref{sec:growthgeom}. From that point onward, $\varphi$ remains fixed with certain assumptions for the rest of the paper.

\begin{figure}[htbp]
\centering
\def\svgwidth{1\textwidth} 
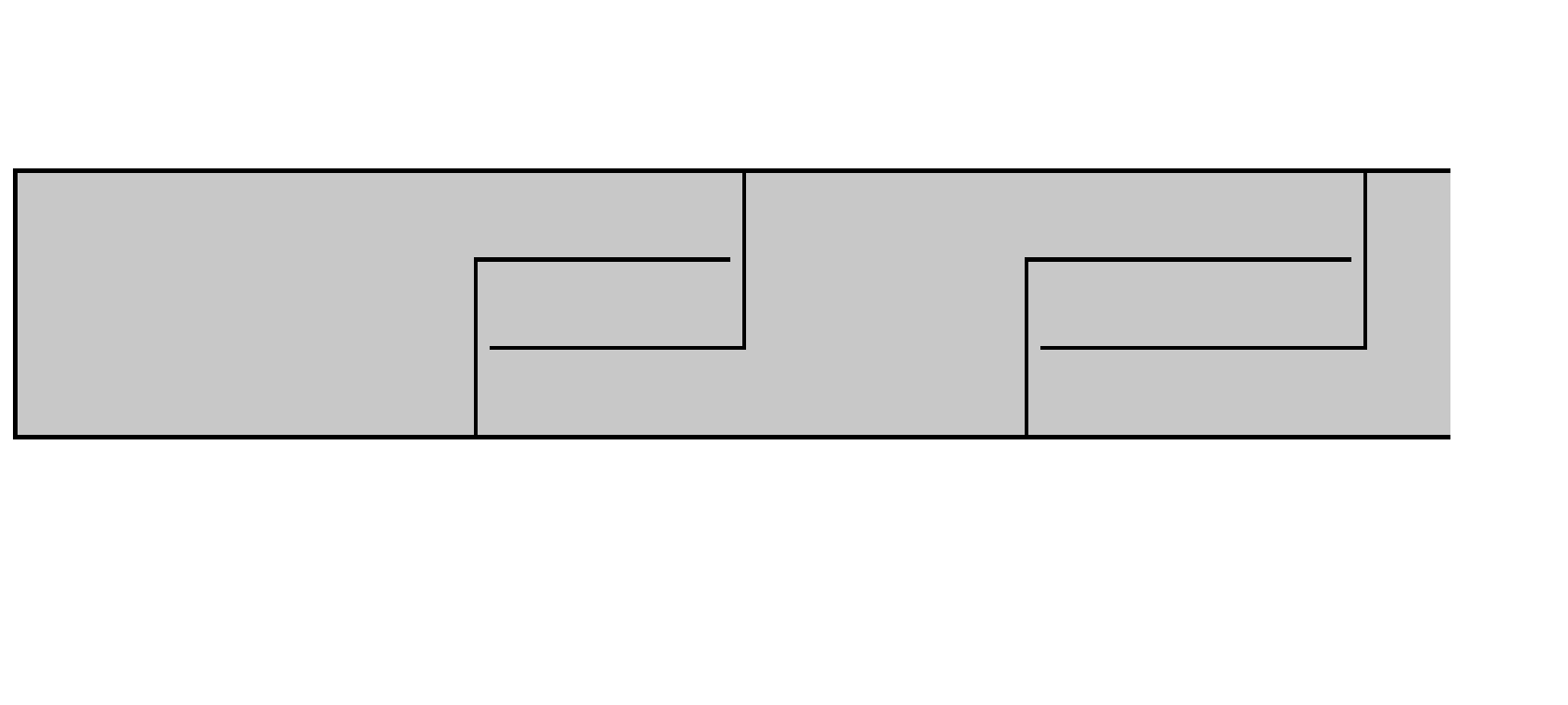
\caption{The desired mapping behaviour of counterexample tracts.}
\label{tractmap}
\end{figure}

The next stage is making sure we can maintain the same kind of mapping behaviour required for counterexamples while also using the gates to guarantee the desired order of growth, now determined (mostly) by $\varphi$. If the gate openings were left too large then it becomes apparent from the comments made earlier that we ``undershoot'' in the sense of Figure~\ref{undershoottract}. Equally, there is also the opportunity to ``overshoot'' in the sense of Figure~\ref{overshoottract} should we make the gate openings too small.
\begin{figure}[htbp]
\centering
\def\svgwidth{1\textwidth}
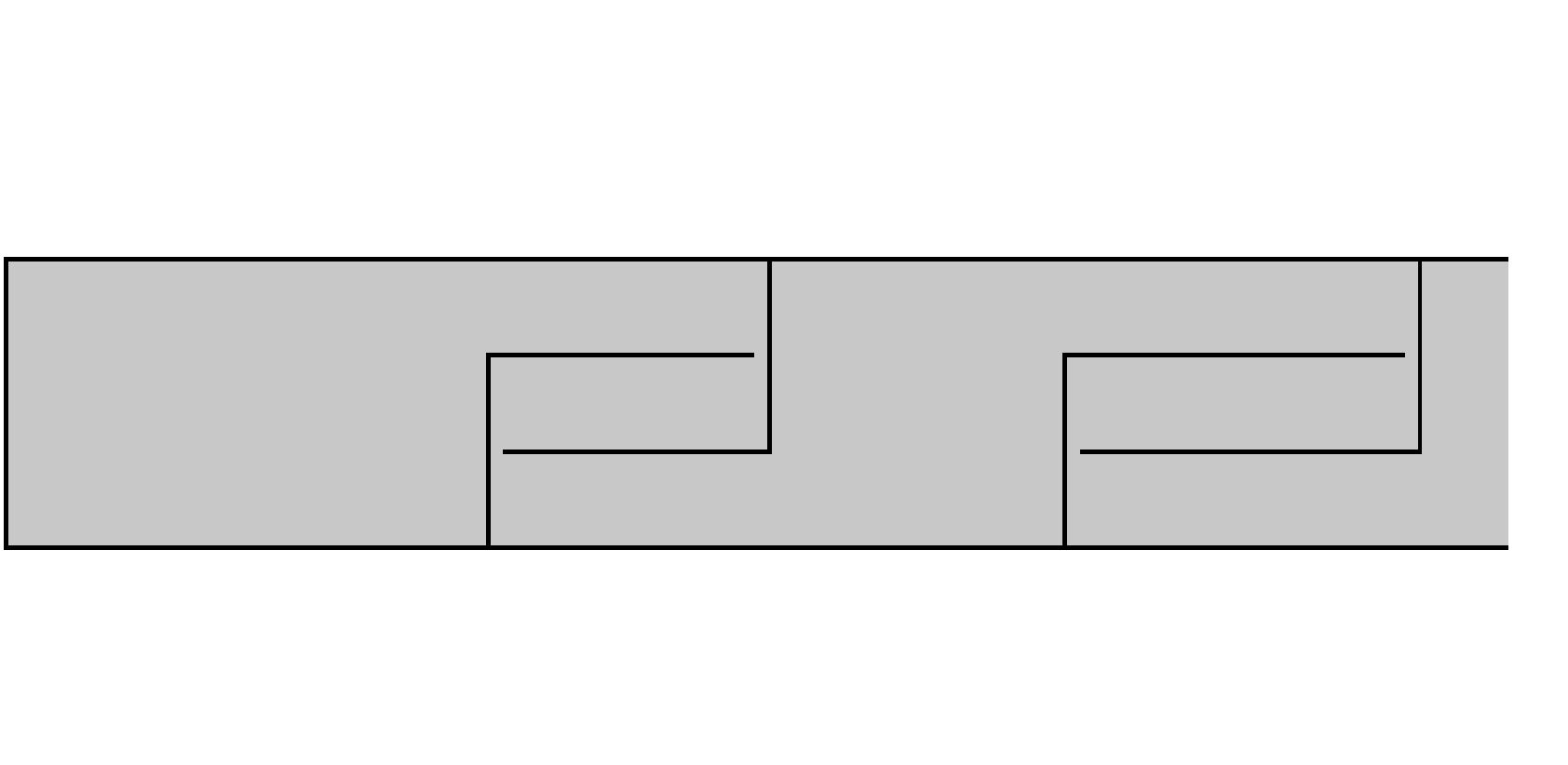
\caption{Undershooting with gates that are too large.}
\label{undershoottract}
\end{figure}

\begin{figure}[htbp]
\centering
\def\svgwidth{1\textwidth}
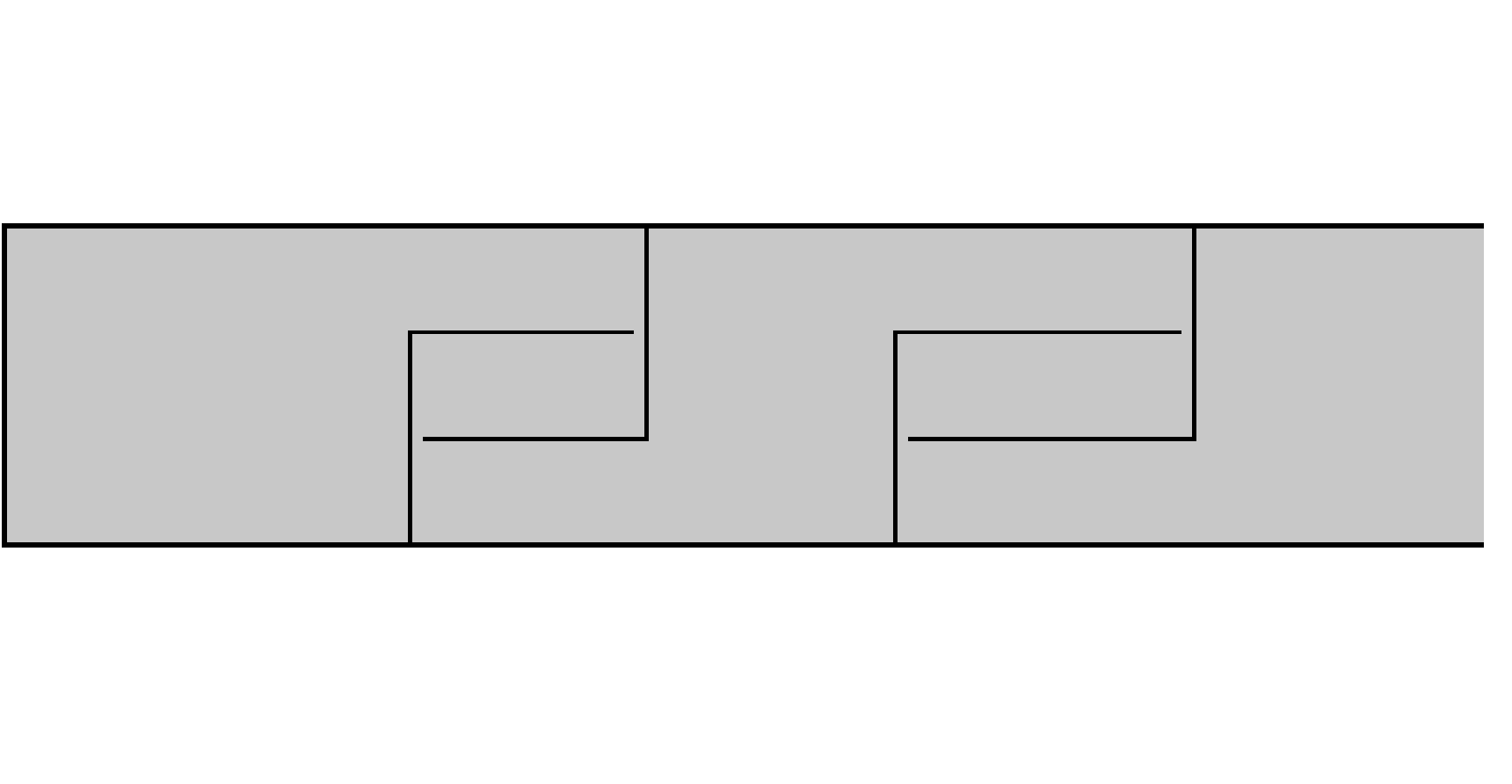
\caption{Overshooting with gates that are too small.}
\label{overshoottract}
\end{figure}

We are then required to satisfy an intermediate value problem for countably many variables simultaneously, that is, ensuring we can find suitable sizes for the countably many gate openings simultaneously in order to achieve the mapping behaviour desired in Figure~\ref{tractmap} across \emph{all} wiggle sections at once. We use a corollary to the Poincar\'{e}--Miranda Theorem (Corollary~\ref{surj} in this paper) and some results from Carath\'{e}odory kernel convergence (Section~\ref{sec:kernel}) to achieve this in Section~\ref{sec:gateselection}.

In Section~\ref{sec:phivspsi} we discuss the way in which we compare $\varphi$ and $\psi$ and what we have to consider in order to ensure our model function, built in accordance with the tracts we construct, leads to a counterexample function satisfying the order of growth we desire.

We prove Theorem~\ref{paracxple} in Section~\ref{secn:theoremproof}. We show that, for a suitable choice of $\Theta$ satisfying certain conditions laid out, we are then able to select a corresponding $\Phi$ such that we can deduce the existence of a function $f \in \classB$ that is a counterexample to the strong Eremenko conjecture and satisfies a very low order of infinite growth. Specifically, \[ \log \log \abs{f(z)} = O\left( \log\abs{z})^{1 + \Theta(\log \abs{z})}\right).\] In our case we follow Bishop's method of constructing functions in class $\classB$ via quasiconformal folding~\cite{fold} as demonstrated in~\cite{bishb}.

Finally, we verify that $\Theta(t)  = 1/(\log\log(t))^\alpha$ is a suitable choice for all $\alpha > 0$.

In Eremenko's paper, he also made the following related statement: ``It is plausible that the set $I(f)$ has no bounded connected components.''~\cite[p.~343]{erem}. This came to be known as the \emph{Eremenko conjecture} which similarly acted as a catalyst for much work and was disproven by Mart\`i-Pete, Rempe, and Waterman~\cite{mav} but we will not discuss this conjecture.

\subsection*{Acknowledgments:} Most of all, I would like to thank Lasse Rempe for giving me this problem to solve which formed the basis of my thesis. Many thanks to both him and David Mart\'{i}-Pete for all of their patience and guidance throughout my studies as wonderful supervisors. I would also like to thank Daniel Meyer and Phil Rippon for taking the great care in reading a very early version of this paper and acting as my examiners. Their thoughtful comments have only made it considerably better and hopefully clearer. I would also like to thank Gwyneth Stallard and Adam Epstein for their support, comments and thoughtful questions throughout the time I spent working on this paper and delivering talks based on it. I would also like to thank my colleagues, especially Louis-Pierre Arguin, Jack Kelly, and Joe Field, as well as my students at Lincoln College, Oxford, where I was teaching when I initially submitted my thesis.

\section{Background}\label{sec:background}

\subsection{Prerequisites from complex analysis and holomorphic dynamics}\label{secn:holdyn}
The reader is referred to \cite{mil06}, \cite{beardon}, \cite{cg}, \cite{steinmetz} and \cite[Section 3]{devaney} for a general introduction to dynamics in one complex variable, and to \cite{waltermero}, \cite{bishnotes} and \cite{dierksurvey} for background on transcendental dynamics. For a survey on the escaping set see~\cite{lassewaltersurvey}.

\begin{defn} Let $U$ be a simply connected domain in $\Chat$. We say that a sequence of functions $f_n \colon U \to \C$ converges \emph{locally uniformly} to $f \colon U \to \C$ on $U$ if, for each $z_0 \in U$ there exists a neighbourhood $U_0 \subset U$ of $z_0$ such that $f_n \to f$ uniformly on $U_0$.
\end{defn}
An easy exercise is the following:
\begin{theorem} $f_n \to f$ locally uniformly on $U$ if and only if $f_n \to f$ uniformly on every compact subset of $U$.
\end{theorem}

\begin{defn}
Let $U \subset \C$ be a domain and let $\mathcal{A}$ be an indexing set. Let $\mathcal{F} = \lbrace f_a \colon a \in \mathcal{A} \rbrace $ where, for each $a \in \mathcal{A}$, $f_a \colon U \to \C$ is holomorphic. We say $\mathcal{F}$ is a \emph{normal family} if every sequence of functions from $\mathcal{F}$ contains a subsequence that converges locally uniformly to a holomorphic limit function $f \colon U \to \C$.
\end{defn}

Let $f\colon \C \to \C$ be a non-constant and non-linear holomorphic map.
\begin{defn}
We define the \emph{Fatou set} of $f$ to be \begin{align*} F(f) \defeq \lbrace z \in \C \colon &\lbrace f^n \colon n \in \mathbb{N}\rbrace \text{ is defined and } \\ &\text{ normal in some neighbourhood of } z \rbrace \end{align*}
\end{defn} 

\begin{defn}
We define the \emph{Julia set} of $f$ to be $J(f) \defeq \C \setminus F(f)$.
\end{defn}

$F(f)$, $J(f)$, and (in the transcendental case) $I(f)$ are forward invariant ($f(A) \subset A$). The reverse inclusion may fail for instances where $f$ has excluded values. For example, $F(\exp(z)/10)$ contains a neighbourhood of $0$ but $0$ is an omitted value for this function. Informally, the Fatou set is the set of stability and the Julia set is the set of instability or the place where `chaos' occurs.

\subsection{Class $\classB$, $\Blog$ and order of growth}\label{sec:classB}
See also \cite{el92}, \cite{rrrs} and \cite{six}. 
Given a transcendental entire function $f \colon \C \to \C$ we define the following.

\begin{defn}
We define the set of \emph{critical values} of $f$ to be the set $CV(f) \defeq \lbrace f(z) \colon z\in\C\text{ and } f'(z) = 0 \rbrace \subseteq \C$.
\end{defn} 

\begin{defn}
We say that $a \in \C$ is an \emph{asymptotic value} of $f$ if there exists a curve $\gamma \colon [0, \infty) \to \C$ with $\lim_{t \to \infty} \abs{\gamma(t)} = \infty$ such that $a = \lim_{t \to \infty} f(\gamma(t))$ and  define the set of \emph{asymptotic values} of $f$ to be $AV(f) \defeq \lbrace a \in \C \colon a \text{ is an asymptotic value of } f \rbrace$.
\end{defn}

\begin{defn}
We define the (finite) set of \emph{singular values} of $f$ to be \[ S(f) \defeq \overline{CV(f) \cup AV(f)}.\]
\end{defn}

\begin{defn}
 The  \emph{Eremenko--Lyubich class $\classB$} is the class of transcendental entire functions $f$ for which $S(f)$ is bounded. 
 \end{defn}
 
It is worth taking a moment to appreciate what this allows. If we are given a function $f \in \classB$ then, as long as we take a large enough domain in $\C$ that contains $S(f)$, we know we can take inverse branches of $f$ on neighbourhoods of every point on the complement. As discussed in the introduction, almost all of the work in this paper is dedicated to the manipulation of tracts.
 
% 
% \begin{defn}[Speiser class]
% The \emph{Speiser} class $\mathcal{S}$ is the class of transcendental entire functions $f$ for which $S(f)$ is finite.
% \end{defn}
% 
% It is immediate to see that all functions in the Speiser class are immediately in class $\classB$ but the converse is not true. 
% 
% \begin{expl} $\sin z/z \in \classB \setminus \mathcal{S}$.
% \end{expl} 

 \noindent The following is well-known but we give a proof for completeness. 

\begin{lemma}\label{lem:tracts}  Let $f \in \classB$ and let $D \subset \C$ be a bounded Jordan domain that contains $S(f)$. Define $W \defeq \C \setminus \overline{D}$ and let $V$ be a connected component of $\mathcal{V} \defeq f^{-1}(\C \setminus \overline{D})$. Then $V$ is an unbounded Jordan domain and $f \colon V \to W$ is a universal covering.
\end{lemma}
\begin{proof}
We know that $S(f) \subset D$, $f \colon V \to W$ is a covering map. $W$ is conformally equivalent to the punctured disc, $\mathbb{D}^*$, and so, by the classification of the covering spaces of the punctured unit disc~\cite[Theorem~5.10]{forster}, $V$ must be simply connected and conformally equivalent to a half-plane as $f$ is transcendental. If we now consider a slightly smaller domain $D'$ such that $S(f) \subset D'$ and $\overline{D'} \subset D$ we can apply the previous observations to $W' \defeq \C \setminus \overline{D'}$ to see that $f \colon \partial V \to \partial D$ is a universal covering and so the result follows.
\end{proof}

\begin{figure}[htbp]
\centering
\def\svgwidth{1\textwidth}
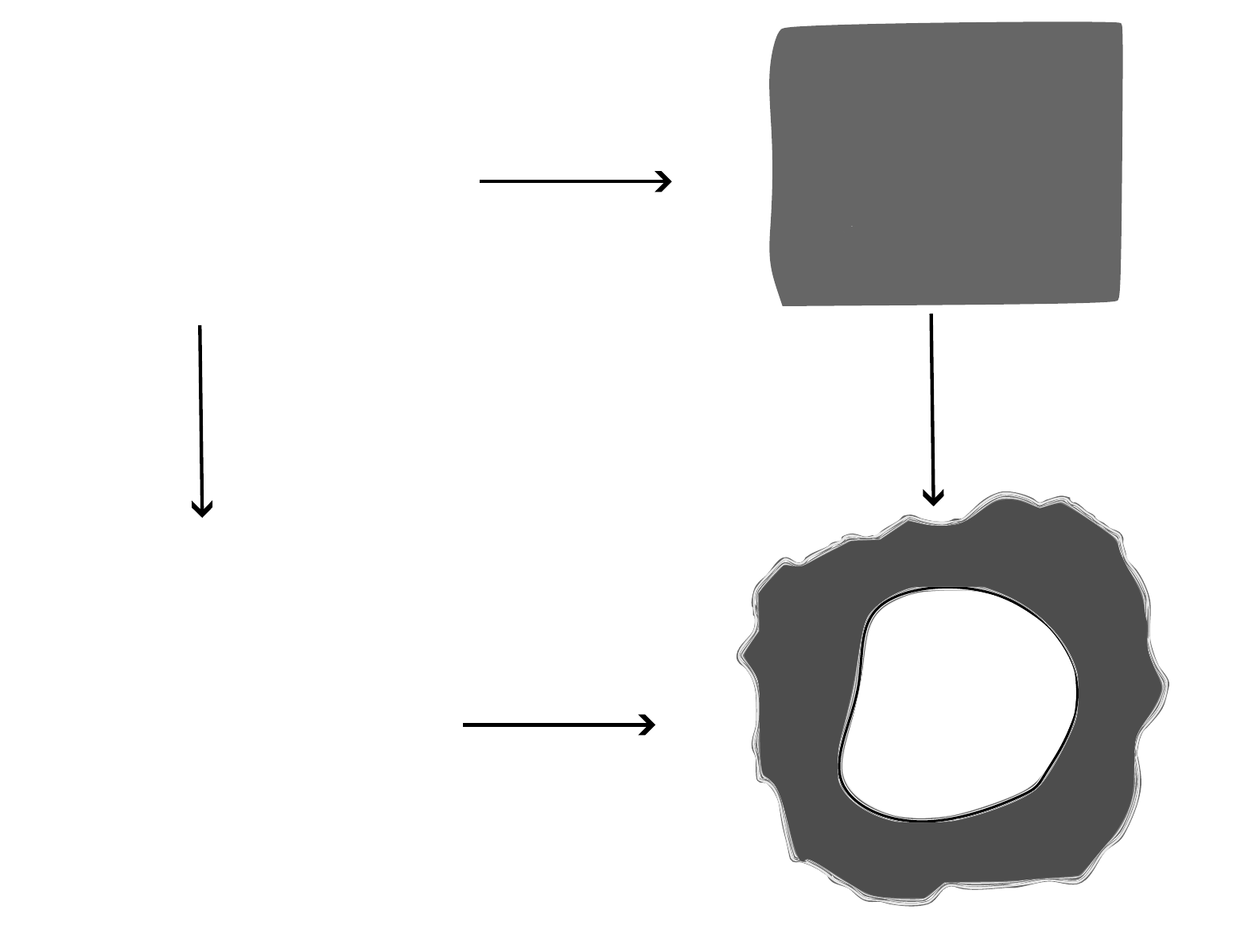
\caption{A logarithmic change of coordinates.}
\label{blogcoords}
\end{figure}

\begin{defn}
The sets $V$ from Lemma~\ref{lem:tracts} are called \emph{tracts} of $f$.
\end{defn}
A consequence of Lemma~\ref{lem:tracts} is that we can understand $f$ by applying a change of coordinates, illustrated in Figure~\ref{blogcoords}. Indeed, suppose 
 that $D$ is chosen as in Lemma~\ref{lem:tracts}, and that additionally $\lbrace 0, f(0)\rbrace \subset D$ then $0 \notin V$ for all tracts $V$ which  allows for preimages under the exponential map. 
Let $\mathcal{U} \defeq \exp^{-1}(\mathcal{V})$ and $H \defeq \exp^{-1}(W)$. 
For any component $U$ of $\mathcal{U}$, both $\exp\colon H\to W$ and $f\circ \exp \colon U \to W$ are universal covering maps. So there is a
 conformal isomorphism $F \colon U\to H$ such that $\exp(F(z)) = f(\exp(z))$. Choosing such a map for every component $U$ of $\mathcal{U}$, we obtain
  a function $F\colon \mathcal{U} \to H$, called a \emph{logarithmic transform} of $f$. The components $U \in \mathcal{U}$ are referred to as the \emph{tracts} of $F$ and we refer to this as taking a \emph{logarithmic change of coordinates}..

From the construction, the following conditions hold:

\begin{defn}[$\Blog$ conditions]\label{blogconds} 
\leavevmode
\begin{enumerate}[label = (\alph*)]
	\item $H$ is a $2\pi i$-periodic Jordan domain that contains a right half-plane;
	\item Every component $U$ of $\mathcal{U}$ is an unbounded Jordan domain with real parts bounded below, but unbounded from above;
	\item The components of $\mathcal{U}$ have disjoint closures and accumulate only at infinity; that is, if $(z_n)_{n=0}^\infty \subset \mathcal{U}$ is a sequence of points all belonging to different tracts, then $z_n \to \infty$;
	\item For every component $U$ of $\mathcal{U}$, $F \colon U \to H$ is a conformal isomorphism between Jordan domains. In particular, $F$ extends continuously to the closure $\overline{\mathcal{U}}$ of $\mathcal{U}$ in $\C$;
	\item For every component $U \in \mathcal{U}$, $\exp \vert_{U}$ is injective;
	\item $\mathcal{U}$ is invariant under translation by $2\pi i$.
\end{enumerate}
\end{defn}

\begin{defn}
We denote by $\Blog$ the class of all $F \colon U \to H$ such that $H$, $\mathcal{U}$, and $F$ satisfy properties (a)~-~(f) in Definition~\ref{blogconds} regardless of whether they arise from a class $\classB$ function or not. If $F$ is also $2\pi i$-periodic then we say that $F$ belongs to the class $\Blogp$.
\end{defn}
 Note that, in the above construction of the logarithmic
transform $F$ of $f\in\classB$, we can always ensure that $F\in \Blogp$. 

By~\cite[Lemma 2.1]{el92}, for $F \in \Blog$ there is a $\rho_0 > 0$ such that \[ \abs{F'(z)} \geq 2 \] whenever $\re F(z) \geq \rho_0$. 
 For $F \in \Blog$, we say that $F$ is of \emph{disjoint type} if $\overline{\mathcal{U}} \subset H$. If our functions are of disjoint type, iteration is defined for all backward images and we consider the following sets. Given such a disjoint-type function $F \in \Blog$, we define 
\begin{align*}
J(F) &\defeq \lbrace z \in \overline{\mathcal{U}} \colon F^n(z) \text{ is defined and in } \; \overline{\mathcal{U}} \text{ for all } n \geq 0 \rbrace  \quad \text{ and }\\
I(F) &\defeq \lbrace z \in J(F) \colon \re F^n(z) \to \infty \text{ as } n \to \infty \rbrace.
\end{align*}
These definitions are the natural ones to take when we consider the instance when $F$ is a logarithmic transform for some $f \in \classB$. Recall the fact that $J(f)$ and $I(f)$ are forward invariant sets and $I(f)$ is non-empty. In order for $z \in I(f)$ to escape and for any $R>0$ there must be some iterate under $f$ such that $\lvert f^k(z)\rvert \geq R = \max \lbrace \lvert \zeta \rvert \colon \zeta \in \overline{D} \rbrace$, where $D$ is the Jordan domain of Lemma~\ref{lem:tracts}. This then requires $ \log \re f^k(z)  > \log R$. This means that the points that map onto $I(f)$ under the exponential map must all inhabit a right half-plane. Since $J(f) \subset I(f)$, the points mapping onto $J(f)$ under the exponential map must also inhabit the same half-plane.

%
%If $f \in \classB$ and $F$ is a logarithmic transform of $f$ then $\exp(J(F)) = J(f)$ and $\exp(I(F)) = I(F)$. PROOF OF THESE?
%In particular, if $I(F)$ does not contain a curve to infinity, then $I(f)$ does not either. 
The following result comes from~\cite[Corollary, p.344]{erem}.
\begin{cor}\label{eremcor}
If $f \in \classB$, then $J(f) = \overline{I(f)}$.
\end{cor}
The same holds for disjoint-type $F\in\Blog$ since such $F$ is 
uniformly expanding on $J(F)$; see~\cite[Lemma~2.1]{rrrs}. Which allows us to state the following.

\begin{cor}\label{loggood}
If $F \in \Blog$ is of disjoint type then $\exp(J(F)) = J(f)$.
\end{cor}
In fact, much more is true, see \cite[Theorem~2.3]{arclike}.

We have already seen in \eqref{finiteorder} what it means for an entire function to have finite order. We now define the corresponding notion for logarithmic transforms.

\begin{defn}\label{Ffiniteorder}
We say that $F \in \Blog$ has \emph{finite order} if \[ \log \re F(z) = O(\re z) \text{ as } \re z \to \infty.\] 
\end{defn}
We say that $F$ has \emph{infinite order} otherwise.

%
%\begin{cor}\label{nologescape}
%If $f \in \classB$ has a logarithmic transform $F$ 
%\end{cor}

\subsection{Approximation via quasiconformal folding}\label{sec:bishopmodel}
It is quite common in holomorphic dynamics for constructions to rely on some form of approximation result from function theory such as the use of Cauchy integrals in~\cite{rrrs}. Bishop's technique of quasiconformal folding~\cite{fold} provided a powerful new technique which has seen the construction of a number of interesting examples and counterexamples since. In Section~\ref{secn:theoremproof} we will use the following results of~\cite{bishb} to deduce the existence of a counterexample function $f \in \classB$ satisfying our desired growth condition.

\begin{defn}\label{eqn:modeldom}
Suppose $\Omega = \bigcup_{j=0}^\infty \Omega_j \subset \C$ is a disjoint union of unbounded simply connected domains satisfying the following conditions.

\begin{enumerate}
\item Sequences of components of $\Omega$ accumulate only at infinity.
\item The set $\partial\Omega_j$ is connected for each $j$ (as a subset of $\C$).
\end{enumerate}

Such an $\Omega$ is called a \emph{model domain}. If $\overline{\Omega} \cap \overline{\mathbb{D}} = \emptyset$ then we say that the model domain is of \emph{disjoint type}.
\end{defn}

\begin{defn}\label{defn:modelfun}
Given a model domain, suppose that $\sigma \colon \Omega \to \HH$ is holomorphic and that the following conditions hold.

\begin{enumerate}
\item The restriction of $\sigma$ to each $\Omega_j$ is a conformal map $\sigma_j \colon \Omega_j \to \HH$.
\item If $(z_n )_{n=0}^\infty$ is a sequence in $\Omega$ and $\sigma(z_n) \to \infty$, then $z_n \to \infty$.
\end{enumerate}

Given such a $\sigma \colon \Omega \to \HH$ we call $g(z) \defeq \exp (\sigma(z))$ a \emph{model function}. 
\end{defn}

\begin{defn}\label{defn:model}
A choice of both a model domain $\Omega$ and a model function $g$ on $\Omega$ will be called a \emph{model}. If $\Omega$ is of disjoint type then we call the overall model $(\Omega, g)$ a disjoint-type model with disjoint-type function $G$. 
\end{defn}
%Given a model $(\Omega, G)$ and $\rho>0$, we let 
%
%\[ \Omega(\rho) \defeq \lbrace z \in \Omega  \colon \abs{G(z)} > e^\rho \rbrace  = \sigma^{-1} ( \lbrace x + iy \colon x > \rho \rbrace)\] and \[\Omega(\delta, \rho)  \defeq \lbrace z \in \Omega  \colon e^\delta <  \abs{G(z)} < e^\rho \rbrace  = \sigma^{-1} ( \lbrace x + iy \colon \delta < x  < \rho \rbrace). \]
\begin{defn}
Given a model $(\Omega, g)$, define
 $J(g)\defeq \lbrace z \in \Omega \colon g^n(z) \in \Omega \text{ for all } n \geq0 \rbrace$ and $I(g) \defeq \lbrace z \in J(G) \colon \re g^n(z) \to \infty \text{ as } n \to \infty \rbrace.$
 \end{defn}
The following result follows from~\cite[Theorem 1.1]{bishb} and~\cite[Chapter II, §4.2]{lehto}, as pointed out by Rempe~\cite[Page 205]{bishb}. 

\begin{theorem}\label{Iconj} If $(\Omega,g)$ is any disjoint-type model, then there is a disjoint-type entire function $f \in \classB$ and a homeomorphism $p \colon \C \to \C$ so that  \[ f \circ p = p \circ g, \] on an open set $U$ that contains $\overline{I(g)}$. We find $p(I(g)) = I(f)$ and that $f$ is bounded on the complement of $p(U)$. 

Furthermore, both $p$ and $p^{-1}$ are H\"older continuous at $\infty$ with some H\"older exponent $K$. 
\end{theorem}

In fact, the map $p$ is \emph{quasiconformal}, but the only property of quasiconformal maps that we will require is their
 H\"older continuity at infinity; see~\cite[Chapter II, §4.2]{lehto}. Also note that we have used Corollary~\ref{eremcor}.

\subsection{The hyperbolic metric}
Throughout the paper we make use of what is known as the \emph{standard estimate}\cite[Cor. A.8]{mil06} of the hyperbolic metric of a simply connected domain $V \subset \C$. The hyperbolic density of $V$, $\lambda_V$, satisfies
\begin{equation}\label{eqn:standardestimate} \frac{1}{2\dist(z, \partial V)} \leq \lambda_V(z) \leq \frac{2}{\dist(z, \partial V)}.\end{equation}

\noindent Another fact that will be freely used throughout is the following. If $A>B>0$ then the hyperbolic distance between $A$ and $B$ in $\HH$ is given by the following
\begin{equation}\label{Hlength}
\text{dist}_{\HH}(A,B)=\log A - \log B.
\end{equation}
The reader is directed to~\cite{anderson} as an introductory text on hyperbolic geometry.

\subsection{A geometric result}\label{sec:geodesic}
This section reproduces results of Appendix A in \cite{rrrs}.

\begin{lemma}{(Geometry of geodesics)}\label{geodes}
Consider the rectangle \[Q = \lbrace z \in \C \colon \abs{\re z} < 4, \abs{\im z} < 1 \rbrace\]
and let $Y \subset \widehat{\C}$ be a simply connected Jordan domain with $Q \subsetneq Y$ such that $\partial Q \cap \partial Y$ consists exactly of the two horizontal boundary sides of $Q$. Let $P$, $R$, $P'$, $R' \in \partial Y$ be four distinct boundary points in this cyclic order, subject to the condition that $P$ and $P'$ are in the boundary of different components of $Y \setminus Q$ and so that the quadrilateral $Y$ with the marked points $P$, $R$, $P'$, $R'$ has modulus $1$. 

Let $\gamma$ be the hyperbolic geodesic in $Y$ connecting $R$ with $R'$. If $0 \in \gamma$, then the two endpoints of $\gamma$ are on the horizontal boundaries of $Q$, one endpoint each on the upper and lower boundary.

\end{lemma}

This following corollary is provided in its stated form so that it is possible to find a maximum diameter of a geodesic in the tracts we use throughout the paper.
\begin{cor}\label{nu0is60} For a marked quadrilateral $Y \subset \widehat{\C}$ of modulus $1$ and a rectangle $Q \subsetneq Y$ where the vertical sides and the horizontal sides are in a ratio of $1 \colon 4$, suppose  $\partial Q \cap \partial Y$ consists exactly of the two horizontal boundary sides of $Q$. The hyperbolic geodesic $\gamma$ that passes through the midpoint of $Q$ has endpoints on the horizontal sides of $Q$ and is completely contained within $Q$.
\end{cor}

\subsection{Carath\'{e}odory kernel convergence}\label{sec:kernel}
Here we recall some topological notions that are used when addressing the issue of selecting suitably sized gates simultaneously across the entire tract in Section~\ref{sec:gateselection}. Much of the exposition follows that of \cite[5.1]{mcmullen}, \cite[119-123]{carath} and~\cite[Section~1.4]{pomm}.

%\noindent The following definition is adapted from \cite[§3]{mil06}.
%\begin{defn} Given a domain $\Omega \subset \C$ let $\text{Map}(\Omega, \C)$ denote the set of continuous maps $f \colon \Omega \to \C$. Given a compact $K \subset \Omega$ and $\eps > 0$ we define $N_{K,\eps} (f) \defeq \lbrace g \in \text{Map}(\Omega, \C) \colon \lvert f(z) - g(z) \rvert < \eps \text{ for all } z \in K \rbrace.$ We say that $U \subset \text{Map}(\Omega, \C)$ is open if, and only if, for every $f \in U$ there exists compact $K \subset \Omega$ and $\eps > 0$ such that $N_{K,\eps}(f) \subset U$. This characterisation of open sets is how we define the \emph{topology of locally uniform convergence} on $Map(\Omega, \C).$
%\end{defn}

\begin{defn} A \emph{disc} is any simpy-connected region in $\C$, possibly $\C$ itself. $\mathcal{D}$ is the set of \emph{pointed discs}, (U,u), where $U$ is a disc and $u \in U$. Let $\mathcal{E}$ be the subspace of pointed discs $(U,u)$ where $U \neq \C$.
\end{defn}
 
\begin{defn}\label{carattop} The \emph{Carath\'{e}odory topology} on $\mathcal{D}$ is defined in the following way. The sequence of pointed discs $(U_n, u_n) \to (U,u)$ if, and only if, \begin{enumerate}\item $u_n \to u$ in the usual sense. \item For any compact $K \subset U,$ there exists some $m \geq 0$ such that $K \subset U_n$ for all $n\geq m.$ \item For any open connected $V$ containing $u$, if $V \subset U_n$ for infinitely many $n$ then $V \subset U$. \end{enumerate} Equivalently, convergence means $u_n \to u$ and for any subsequence such that $\Chat \setminus U_n \to K$ in the Haudorff topology on compact subsets of the sphere, $U$ is equal to the component of $\Chat \setminus K$ that contains $u$. 
\end{defn}
We state the following as per~\cite{pomm}, originally given in~\cite{carath}.
\begin{theorem}[The Carath\`{e}odory kernel theorem]\label{thm:caratkern}
	Let $\phi_n$ map $\DD$ conformally onto $U_n$ with $\phi_n(0)=0$ and $\phi_n'(0) > 0$. If $U = \lbrace  u \rbrace$ let $\phi(z)\equiv u$, otherwise let $\phi$ map $\DD$ conformally onto $U$ with $\phi(0) = u$ and $\phi'(0) > 0$. Then, as $n \to \infty$, \[ \phi _n \to \phi \text{ locally uniformly in } \DD \text{ if and only if } U_n \to U \text{ in the sense of Definition~\ref{carattop}}.\]
	
\end{theorem}

%We recall the following, taken from~\cite[II.9]{falconer}
%\begin{defn} If $\mathcal{C}(\Chat) \defeq \lbrace V \subset \Chat \colon V \text{ is compact} \rbrace$ and if $V_\delta \defeq \lbrace z \in \Chat \colon \lvert z - v \rvert \leq \delta \text{ for some } v \in V \rbrace$ then $\text{d}_H (A,B) \defeq \inf \lbrace \delta\geq 0 \colon A \subset B_\delta \text{ and } B \subset A_\delta \rbrace$ is a metric on $\mathcal{C}(\Chat).$
%\end{defn}
%
%With this we can define the following
%
%\begin{defn}\label{caratmetric} Let $(U, u)$ and $(V, v) \in \mathcal{E}$. Define $\text{D}((U,u),(V,v)) \defeq \max \lbrace \text{d}_H(\Chat\setminus U, \Chat \setminus V), \lvert u - v \rvert \rbrace.$ This defines a metric on $\mathcal{E}$ with which we retrieve the Carath\'eodory topology.
%\end{defn}

\subsection{The Poincar\'{e}--Miranda Theorem}\label{sec:pmthm}
This section reproduces results from \cite{mawhin} and \cite{kulpa} where some interesting historical context and background to the theorem is provided. The Poincar\'e--Miranda theorem, conjectured by Poincar\'e, was shown to be equivalent to the Brouwer Fixed Point Theorem by Miranda.

The theorem itself can be thought of as a generalisation of the intermediate value theorem and Corollary~\ref{surj} will be used in Section~\ref{sec:gateselection}. We introduce the following definition with corresponding notation to refer to the $n$-dimensional cube.

\begin{defn}\label{ncube}
For $n \geq 1$ let $\Lambda^n \defeq [-1, 1]^n$. For $1\leq k \leq n$ define \begin{align*}
\Lambda_k^+ &\defeq \lbrace x \in \Lambda^n \colon x_k = 1 \rbrace \\
\Lambda_k^- &\defeq \lbrace x \in \Lambda^n \colon x_k= -1 \rbrace
\end{align*}
which we call the $k$-th pair of opposite faces.
\end{defn}
\begin{theorem}[The Poincar\'e--Miranda Theorem]
Let $p \colon \Lambda^n \to \mathbb{R}^n$, $p = (p_1, p_2,...,p_n)$ be a continuous map such that for each $1 \leq k \leq n$, $p_k(\Lambda_k^-) \subset (-\infty, 0]$ and $p_k(\Lambda_k^+) \subset [0, \infty)$. Then there exists a point $c \in \Lambda^n$ such that $p(c) = 0.$
\end{theorem}

%\textbf{Note:} The assumption that for each $j \leq n$, $p_j(\Lambda_j^-) \subset (-\infty, 0]$ and $p_j(\Lambda_j^+) \subset [0, \infty)$ can be replaced by what is called the \emph{Bolzano condition} which requires \[ p_j(a)\cdot p_j(b) \leq 0 \] for each $j \leq n$ and $a \in \Lambda_j^-$, $b \in \Lambda_j^+$.
\begin{theorem}[Coincidence Theorem]
If maps $p, q \colon \Lambda^n \to \Lambda^n$ are continuous and if $p(\Lambda_k^-) \subset \Lambda_k^-$ and $p(\Lambda_k^+) \subset \Lambda_k^+$ for each $1\leq k \leq n$ then there exists a point $c \in \Lambda^n$ such that $q(c) = p(c)$.
\end{theorem}
\begin{proof} Let $r(x) \defeq p(x) - q(x)$, this map satisfies the conditions of the Poincar\'e--Miranda theorem so there exists a point $c \in \Lambda^n$ such that $r(c)=0$, that is, $p(c) = q(c)$.
\end{proof}

\begin{cor}\label{surj} Let $p \colon \Lambda^n \to \Lambda^n$  be a continuous function. If $p(\Lambda_k^-) \subset \Lambda_k^-$ and $p(\Lambda_k^+) \subset \Lambda_k^+$ for each $ 1 \leq k \leq n$ then $p$ is surjective.
\end{cor}
\begin{proof} Given $a \in \Lambda^n$, let $q_a(x) \defeq a$, a constant function. We can apply the Coincidence Theorem to deduce the existence of $c_a \in \Lambda^n$ such that $p(c_a) = q_a(c_a) = a$. This can be achieved for all $a \in  \Lambda^n$ thus giving surjectivity.
\end{proof}

\section{Tracts}\label{sec:tracts}
Our goal is to construct a function $\tilde{F} \in \Blogp$ such that $J(\tilde{F})$ does not contain a curve to infinity and that the following order of growth is satisfied: $\log (\re \tilde{F} (z)) = O (\log ((\re z)^{1 + o(1)}))$ in a way such that, by use of the approximation results in Section~\ref{sec:bishopmodel}, we can find a function $f \in \classB$ which is approximated by our constructed $\tilde{F}$.

In order to construct the suggested $\tilde{F}$, we will define a certain unbounded Jordan domain $T$, disjoint from its $2\pi i\Z$-translates, and a conformal isomorphism
  $F\colon T\to \HH$, and then extend this map periodically to a function $\tilde{F}$ in $\Blogp$. 

 The simplest way to obtain a domain that is disjoint from its $2\pi i\Z$ translates is to require it to be contained in a (right) half-strip of height $2\pi$. It will be slightly 
  more convenient to work with simply connected domains $T$ that are not necessarily Jordan domains to begin with before making a suitable restriction in Definition~\ref{defn:partplane} and Proposition~\ref{prop:Blogp}. More precisely, all domains we consider will
  belong to the following general class. Here we follow~\cite[Definition 3.1]{tania}.
  
\begin{defn} Let $\mathcal{T}$ be the collection of simply connected domains $T \subset \lbrace z \in \C \colon \re z > 4, \abs{\im z} < \pi \rbrace$ such that
\begin{itemize}
\item $5 \in T$,
\item $\partial T \cup \{\infty\}$ is locally connected,
\item there is exactly one access to $\infty$ in $T$. That is, any two curves connecting the same finite endpoint to infinity in $T$ are homotopic.
\end{itemize}
\end{defn}
\begin{prop}\label{prop:F_from_T}
For any $T \in \mathcal{T}$, there exists a unique conformal isomorphism $F \colon T \to \HH$ such that $F(5)= 5$ and  \[ \lim_{z \to \infty, z\in T} F(z)= \infty.\]
\end{prop}
\begin{proof}There exists a conformal isomorphism $F \colon T \to \HH$ such that $F(5) = 5$ which is unique up to post-composition by a M\"{o}bius transformation that fixes $5$. By the Carath\'eodory--Torhorst theorem \cite[Theorem 2.1]{pomm}, $F^{-1} \colon \HH \to T$ extends continuously to $\overline{\HH} \cup \lbrace \infty \rbrace$. 
The points $\zeta \in \partial \HH \cup \lbrace \infty \rbrace$ with $F^{-1}(\zeta) = \infty$ are in one-to-one correspondence with the accesses to
infinity in $T$ (see~\cite[Section 3]{fatousassociates}). So by assumption on $T$, there is exactly one such point $\zeta$. By post-composing $F$ with a M\"{o}bius transformation we may assume that $\zeta = \infty$. It follows that $\zeta$ is the only accumulation
point of $F(z)$ as $z\to\infty$. Since $\overline{\HH} \cup \lbrace \infty \rbrace$ is compact, $F(z)\to\infty$ as $z\to\infty$ in $T$, as claimed. 

If $\tilde{F}$ is another conformal isomorphism satisfying the conclusion of the proposition, then $M\defeq \tilde{F} \circ F^{-1}\colon \HH\to\HH$ is a conformal
 isomorphism, and hence a M\"obius transformation. We have $M(5)=5$ and $M(\infty)=\infty$. So $M$ is the identity, and $F$ is indeed unique. 
\end{proof}
\begin{defn} Let $\mathcal{H}$ denote the collection of all conformal isomorphisms as in Proposition~\ref{prop:F_from_T}; i.e.
 \begin{align*}
    \mathcal{H} \defeq \lbrace F \colon T \to \HH \text{ conformal isomorphism}\colon &T \in \mathcal{T}, F(5)=5, \\ &\text{ and } 
      \lim_{\lvert z \rvert \to \infty} F(z) = \infty \rbrace . \end{align*}
\end{defn}
Proposition~\ref{prop:F_from_T} thus states that the function that maps an $F\in\mathcal{H}$ to its domain $T \in \mathcal{T}$ is a bijection. 
 Given $F\in\mathcal{H}$, we can obtain a disjoint-type function in $\Blogp$ by making the following restriction. First we give notation to the following right half-plane that is a subset of $\HH$.
 \begin{defn}\label{defn:partplane} Let $H_1 \defeq \lbrace z \in \C \colon \re(z)>1 \rbrace$.
 \end{defn}
 By restricting $F^{-1}$ to $H_1$ we are able to fulfil the conditions of $\Blog$ since the preimage of $H_1$ under $F$ will be a Jordan domain that is a subset of $T$, this allows us to satisfy condition (c) by avoiding any intersection between the $2\pi i $ translates of $\overline{F^{-1}(H_1)}$.
\begin{prop}\label{prop:Blogp}
 Let $F\in\mathcal{H}$, say $F\colon T\to\HH$ and let  $\tilde{T}\defeq F^{-1}(H_1)$. Set 
  $\tilde{\mathcal{U}} \defeq \tilde{T} + 2\pi i \Z$ and consider the function 
    \[ \tilde{F}\colon \tilde{\mathcal{U}} \to H_1; \quad z\mapsto F(z - 2\pi i m), \] 
    where $m$ is such that $z\in \tilde{T} + 2\pi i m$. We define $\tilde{T}_m \defeq \tilde{T} + 2\pi i m$.
    
   Then $\tilde{F}\in \Blogp$ is of disjoint type.
\end{prop}
\begin{proof} Since $F$ is continuous at $\infty$, the domain $\tilde{T}$ is an unbounded Jordan domain. The
  conditions from the definition of $\Blog$ are easy to check and $\tilde{F}$ is periodic by definition, 
  so $\tilde{F}\in \Blogp$ as claimed. The closure of $T$ is contained in $H_1$ by definition of $\mathcal{T}$ and
  therefore $\tilde{F}$ is of disjoint type. 
\end{proof}
 We will mostly study $F \colon T \to \HH$ for the majority of the paper and will refer to the corresponding $\tilde{F} \in \Blogp$ when appropriate.
 
 The geodesics of $\HH$ are either horizontal half-lines extending from the imaginary axis or are semicircular arcs whose endpoints meet the imaginary axis at right-angles~\cite{anderson}. A number of results in this paper will involve studying the preimages of these semicircular arcs, particularly those centred at the origin. Since the semicircular geodesics split $\HH$ into two regions, we can similarly split $T$ into two regions by taking a preimage under $F$. The preimages of these semicircular geodesics that are geodesics of $T$ will be considerable importance in the following work so, following~\cite{tania}, we call these \emph{vertical} geodesics and give the formal definition now.
       
\begin{defn}\label{vertgeo} Given $T \in \mathcal{T}$ with corresponding conformal isomorphism $F \in \mathcal{H}$ and $\rho > 0$ we define  $\Gamma_\rho \defeq \lbrace z \in T \colon \abs{F(z)} = \rho \rbrace.$ We refer to $\Gamma_\rho$ as a \emph{vertical geodesic} of $T$.
\end{defn}

Following~\cite[Definition~3.1]{tania}, we further define the following subclass of $\mathcal{H}$ where all vertical geodesics have their diameter bounded above by a universal constant.
\begin{defn} For $\nu > 0$, \[  \mathcal{H}_\nu \defeq \lbrace F \in \mathcal{H} \colon \text{ for all } \rho>0, \text{ diam}(\Gamma_\rho) < \nu \rbrace \] where the diameter is understood to be taken in the Euclidean sense.
\end{defn}
We wish to show that there is some $\nu_0>0$ such that all of the conformal isomorphisms we will be considering are elements of $\mathcal{H}_{\nu_0}$. This means that all of the vertical geodesics of every tract $T$ that we consider in this paper will all be bounded above by a single constant $\nu_0$. We will achieve this in Proposition~\ref{nu0}.
%Finally define the following subset of $T$ for a given $F$
%
%\begin{defn}
%For $F \in \mathcal{H}$ let $X(F) \defeq \lbrace z \in T \colon F^n(z) \in T \text{ for all } n\geq 0 \rbrace$.
%\end{defn}

\subsection{Wiggles and gates} As mentioned in the introduction, all the tracts we consider in our construction are of a specific form, consisting of a sequence 
of ``wiggles'' that contain small openings which we call ``epsilon gates'' (or just ``gates''). More precisely, each such tract is defined by a collection of sequences as follows.

BRICKING UP THE TRACT here, explain where 30 and 60 come from below.
\begin{defn}\label{defn:datasets}
Let $(r_j)_{j=0}^\infty$, $(R_j)_{j=0}^\infty$, $(\varepsilon_j)_{j=0}^\infty$ and $(\tau_j)_{j=0}^\infty$ be sequences of positive real numbers such that, for all $j \geq 0$:
\begin{itemize}
\item $r_0> 6$, $R_j > r_j + 30,$ and $r_{j+1} > R_j + 60$.
\item $0 < \eps_j \leq 1$,
\item $ r_j < \tau_j < R_j - 1 - 3\pi$.
\end{itemize}
We call a collection $\xi \defeq \left( r_j, R_j, \varepsilon_j, \tau_j \right)_{j=0}^\infty$ of such information a \emph{tract datum}, and denote the 
collection of all possible tract data by $\Xi_0$.
\end{defn}

It is elementary to check that 
  \begin{equation}\label{eqn:rjbounds} R_j > 36+ 90j \quad\text{and}\quad r_j> 6 + 90j \end{equation}
   for all $j\geq 0$, whenever 
$(r_j)_{j=0}^\infty$ and $(R_j)_{j=0}^\infty$ are as in Definition~\eqref{defn:datasets}.

We will now describe the tract given by a tract datum (see Figure~\ref{tract}).

\begin{defn}\label{defn:ourtracts} Given $\xi \in \Xi_0$, let 
\begin{align*} L =  L^{\xi} \defeq \bigcup_{j=0}^{\infty} \Bigl[ &
   \lbrace r_j + ti \colon t \in [-\pi, \pi /3] \rbrace \cup 
  \lbrace t + \pi i/3 \colon t \in [r_j, R_j - 1] \rbrace  \\ 
    &\cup
 \lbrace t - \pi i /3 \colon t \in [r_j +1 , R_j] \rbrace  
\\&\cup \lbrace R_j + ti \colon t \in [-\pi /3, \pi ] \rbrace  \\
   &\cup \lbrace \tau_j + ti \colon t \in [\pi/3, \pi(2  - \varepsilon_j)/3] \rbrace  
    \\&\cup\lbrace \tau_j + ti \colon t \in [\pi(2 + \varepsilon_j)/3, \pi] \rbrace
  \Bigr]. \end{align*}
Let $T = T^{\xi} \defeq \lbrace z \in \C \colon \re z > 4, \lvert\im z\rvert  < \pi \rbrace \setminus L^{\xi}.$ 
\end{defn}

For a given $\xi \in \Xi_0$, there is a corresponding tract, $T^{\xi} \in \mathcal{T}$, and conformal isomorphism, $F=F^{\xi}\in \mathcal{H}$, such that $F^{\xi} \colon T^{\xi} \to \HH$.

\begin{figure}[htbp]
\centering
\def\svgwidth{1\textwidth}
%% Creator: Inkscape 1.3 (0e150ed6c4, 2023-07-21), www.inkscape.org
%% Accompanies image file 'wigglegates.pdf' (pdf, eps, ps)
\begingroup%
  \makeatletter%
  \providecommand\color[2][]{%
    \errmessage{(Inkscape) Color is used for the text in Inkscape, but the package 'color.sty' is not loaded}%
    \renewcommand\color[2][]{}%
  }%
  \providecommand\transparent[1]{%
    \errmessage{(Inkscape) Transparency is used (non-zero) for the text in Inkscape, but the package 'transparent.sty' is not loaded}%
    \renewcommand\transparent[1]{}%
  }%
  \providecommand\rotatebox[2]{#2}%
  \newcommand*\fsize{\dimexpr\f@size pt\relax}%
  \newcommand*\lineheight[1]{\fontsize{\fsize}{#1\fsize}\selectfont}%
  \ifx\svgwidth\undefined%
    \setlength{\unitlength}{841.88976378bp}%
    \ifx\svgscale\undefined%
      \relax%
    \else%
      \setlength{\unitlength}{\unitlength * \real{\svgscale}}%
    \fi%
  \else%
    \setlength{\unitlength}{\svgwidth}%
  \fi%
  \global\let\svgwidth\undefined%
  \global\let\svgscale\undefined%
  \makeatother%
  \begin{picture}(1,0.70707071)%
    \lineheight{1}%
    \setlength\tabcolsep{0pt}%
    \put(0.71600362,0.12575261){\color[rgb]{0,0,0}\transparent{0.98227799}\makebox(0,0)[lt]{\lineheight{1.25}\smash{\begin{tabular}[t]{l}\Large $R_j$\end{tabular}}}}%
    \put(0,0){\includegraphics[width=\unitlength,page=1]{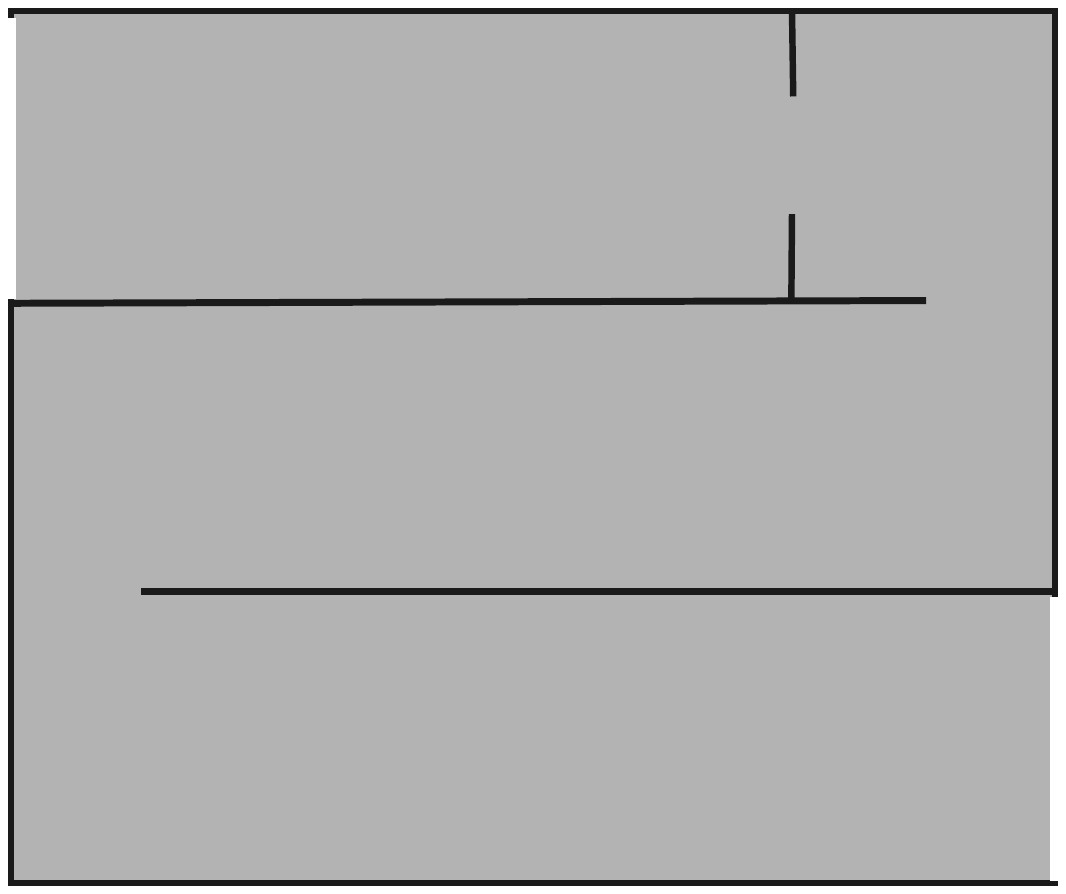}}%
    \put(0.45364231,0.56295775){\color[rgb]{0,0,0}\transparent{0.98227799}\makebox(0,0)[lt]{\lineheight{1.25}\smash{\begin{tabular}[t]{l}\Large $\frac{2\pi \varepsilon_j}{3}\Bigl\{$\end{tabular}}}}%
    \put(0.5666917,0.12581381){\color[rgb]{0,0,0}\transparent{0.98227799}\makebox(0,0)[lt]{\lineheight{1.25}\smash{\begin{tabular}[t]{l}\Large $\tau_j$\end{tabular}}}}%
    \put(0.13006234,0.12685067){\color[rgb]{0,0,0}\makebox(0,0)[lt]{\lineheight{1.25}\smash{\begin{tabular}[t]{l}\Large $r_j$\end{tabular}}}}%
    \put(0,0){\includegraphics[width=\unitlength,page=2]{wigglegates.pdf}}%
  \end{picture}%
\endgroup%

\caption{Close-up of a wiggle section with a gate.}
\label{wigglegates}
\end{figure}

\begin{defn}
We define  $\mathcal{T}^{\Xi_0} \defeq \lbrace T^{\xi} \colon \xi \in \Xi_0 \rbrace$
and \[\mathcal{H}^{\Xi_0} \defeq \lbrace F^{\xi} \colon T^{\xi} \to \HH \colon T^{\xi} \in \mathcal{T}^{\Xi_0}  \rbrace \subset \mathcal{H}.\]
\end{defn}

\begin{prop}\label{nu0} There exists $\nu_0 > 0$ such that $\mathcal{H}^{\Xi_0} \subset \mathcal{H}_{\nu_0}.$
\end{prop}
\begin{proof} Although the tracts in this paper are different to those in~\cite{tania}, the proof of~ \cite[Proposition 7.1]{tania} does not change as the only difference is the presence of our gates. The argument proceeds by considering harmonic measure of $\partial \HH$ at $5$ and the proof carries out identically. \cite{gmharmonic} is a recommended text for the reader unfamiliar with harmonic measure. The other results invoked are~\cite[Corollary 4.18, Corollary 4.21]{pomm} and~\cite[Proposition 7.4]{arclike}.
\end{proof}

For the remainder of the paper, we fix the value $\nu_0$ from Proposition~\ref{nu0}. Using Corollary~\ref{nu0is60}, it is possible to show that we can take
$\nu_0=60$ but we will not be required to use this explicit value.

Using Proposition~\ref{nu0}, we can now fix the position $\tau_j$ of the gates relative to $R_j$ in terms of $\nu_0$ which we couldn't do before deducing the existence of this universal constant.

\begin{defn}\label{defn:rRspacing}
 Given sequences $(r_j)_{j=0}^{\infty}$, $(R_j)_{j=0}^{\infty}$ and $(\eps_j)_{j=0}^{\infty}$ that satisfy the 
  conditions in Definition~\ref{defn:datasets} as well as the following conditions (in light of Proposition~\ref{nu0}) \[R_j > r_j + \max \lbrace 5 + 3\nu_0, 30 \rbrace , \text{ and } r_{j+1} > R_j + \max\lbrace 2\nu_0, 60 \rbrace .\] We define 
  $\tau_j\defeq R_j - 2 - 3\nu_0$ for $j\geq 0$. This defines a tract datum $\xi\in \Xi_0$.
  
   We define $\Xi$ to be the set of all tract data arising in this manner, and all tract data
    considered in the remainder of the paper will be in $\Xi$. Slightly abusing notation, we will also write 
    $\xi = (r_j,R_j,\eps_j)_{j=0}^{\infty}$ for members of $\Xi$.
    
We denote the corresponding subclasses of tracts and isomorphisms by $\mathcal{T}^{\Xi}\subset \mathcal{T}^{\Xi_0}$ and $\mathcal{H}^{\Xi}\subset\mathcal{H}^{\Xi_0}$. 
\end{defn}

What is achieved by this restriction is that any vertical geodesic that contains a point in the lower two-thirds of the ``wiggle'' between real parts $r_j$ and $R_j$ will not enter the gate, since its diameter is at most $\nu_0$. 

\section{Conditions for counterexamples}\label{sec:cxples}

%%\begin{defn} Given $T \in \mathcal{T}$ with corresponding conformal isomorphism $F \in \mathcal{H}$ we define $X(F) \defeq \lbrace z \in T \colon F^n(z) \text{ is defined and in } T \text{ for all } n \geq 0 \rbrace.$
%%\end{defn}
If we impose a small set of conditions on the spacings between $r_j$, $R_j$, certain vertical geodesics in $T$, and their images in $\HH$, we are able to arrive at a tract datum that gives a tract $T$ and conformal isomorphism $F$ such that $J(\tilde{F})$ does not contain any
curves to $\infty$, where $\tilde{F}\in\Blogp$ is the function obtained from $F$ as in Proposition~\ref{prop:Blogp}.
In Section~\ref{sec:modelbuild} we will show how such an $F$ can be used to determine the existence of an $f \in \classB$ that is a counterexample to the strong Eremenko conjecture.

\begin{defn}\label{defn:W}
Given $T \in \mathcal{T}^{\Xi}$ let us define \[ W_j \defeq \Bigl\{ z \in T \colon r_j < \re z < R_j \text{ and } -\pi < \im z < \pi/3 \Bigr\}. \] This corresponds to the ``bottom two-thirds'' of a wiggle.

Let us further define the corresponding subsets of $W_j$:
\begin{align}
W^+_j  &\defeq \Bigl\{ z \in W_j \colon  -\pi/3 < \im z <\pi/3 \Bigr\}, \label{W+} \\
W^-_j &\defeq \Bigl\{ z \in W_j \colon  -\pi < \im z < -\pi/3 \Bigr\}. \label{W-}
\end{align}
See Figure~\ref{wj} for an illustration.
\end{defn}
\begin{figure}[htbp]
\centering
\def\svgwidth{1\textwidth}
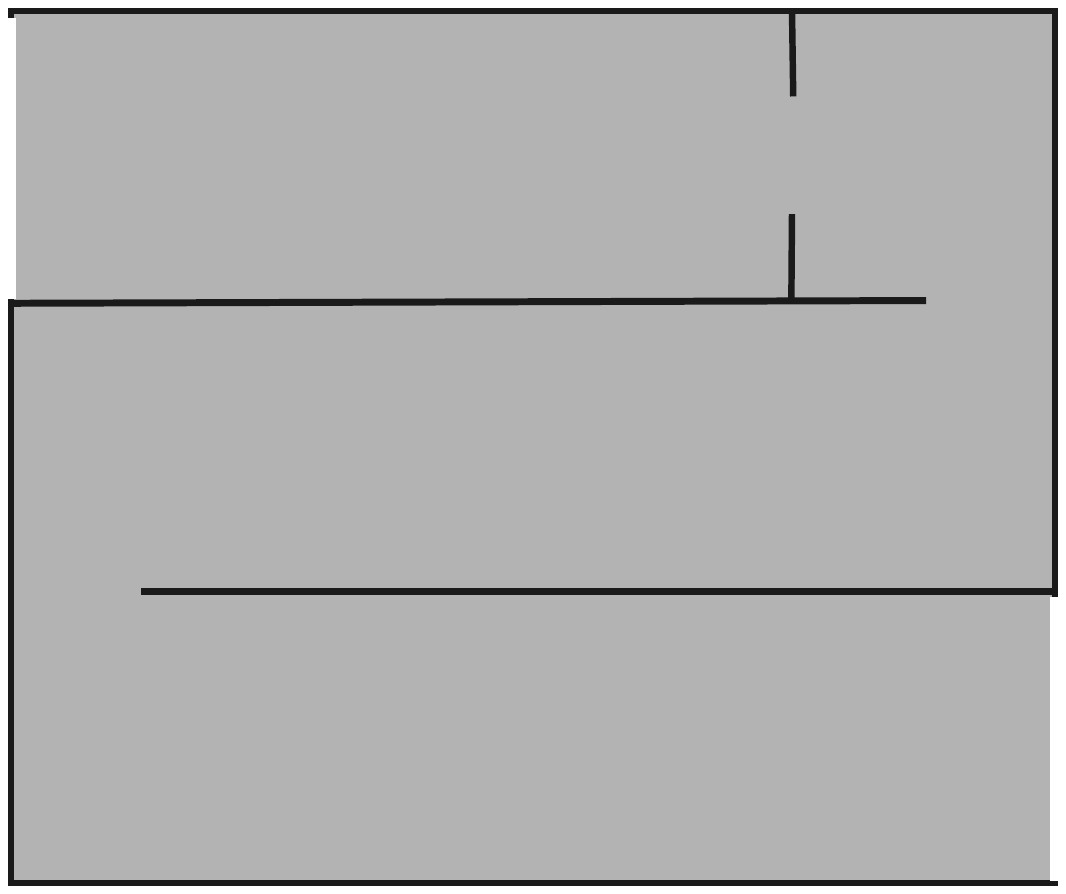
\caption{$W_j$, $W_j^+$, and $W_j^-$.}
\label{wj}
\end{figure}
Following the example of the conditions laid out in \cite[Section~6]{rrrs} that led to the existence of counterexamples, we list a similar corresponding set of conditions that will ensure any function satisfying them must be a counterexample to the strong Eremenko conjecture, which we will justify in Theorem~\ref{nocurve}. The following definition is illustrated by Figure~\ref{tractmap}.
\begin{defn}\label{cxpleconds} Let $\xi\in\Xi$ and let $T=T^{\xi}$ and $F=F^{\xi}$ be the corresponding elements of
$\mathcal{T}^{\Xi}$ and $\mathcal{H}^{\Xi}$. We call $\xi$ a \emph{counterexample tract datum} if there exist sequences
  $(\rho_j)_{j=1}^{\infty}$ and $(\dot{\rho}_j)_{j=1}^{\infty}$ of positive numbers such that the following conditions hold. 
%
%Suppose, for a set of data $\xi \in \Xi$ with corresponding $T \in \mathcal{T}^{\Xi}$ and $F \in \mathcal{H}^\Xi$, that for each $j \geq 0$ there are vertical geodesics $C_j$ and $\dot{C}_j$ and numbers $\rho_{j+1}$, $\dot{\rho}_{j+1} > 0$ where $C_j \defeq \Gamma_{\rho_{j+1}}$ and $\dot{C}_j \defeq \Gamma_{\dot{\rho}_{j+1}}$ such that the following conditions hold for some constant $\kappa >0$:	
		\begin{enumerate}[label=(\alph*)] 
			\item  $r_j + R_{j-1} + 1 + 2\pi \leq \rho_j < \frac{\dot{\rho}_j}{2} < \dot{\rho}_j < R_j - 2 - 4\nu_0$ for all $j\geq 1$;
			\item For all $j\geq 0$, the vertical geodesics $C_{j} \defeq \Gamma_{\rho_{j+1}}$ and $\dot{C}_j\defeq \Gamma_{\dot{\rho}_{j+1}}$ are contained in
			     $W_j^+$ and $W_j^-$, respectively, and both 
			     have real parts strictly between $R_j - 2 - 3 \nu_0 $ and $R_j - 2 - \nu_0 $. 
		\end{enumerate} 
We denote the set of all counterexample tract data by $\Xi_\mathcal{C}\subset \Xi$ and similarly define classes of corresponding tracts, $\mathcal{T}^{\Xi_\mathcal{C}}$, and conformal isomorphisms, $\mathcal{H}^{\Xi_\mathcal{C}}$.
\end{defn}
In Section~\ref{sec:gateselection} we will prove that such a tract data does exist.

\begin{figure}[htbp]
\centering
\def\svgwidth{1\textwidth}
%% Creator: Inkscape 1.3 (0e150ed6c4, 2023-07-21), www.inkscape.org
%% Accompanies image file 'preimages.pdf' (pdf, eps, ps)
\begingroup%
  \makeatletter%
  \providecommand\color[2][]{%
    \errmessage{(Inkscape) Color is used for the text in Inkscape, but the package 'color.sty' is not loaded}%
    \renewcommand\color[2][]{}%
  }%
  \providecommand\transparent[1]{%
    \errmessage{(Inkscape) Transparency is used (non-zero) for the text in Inkscape, but the package 'transparent.sty' is not loaded}%
    \renewcommand\transparent[1]{}%
  }%
  \providecommand\rotatebox[2]{#2}%
  \newcommand*\fsize{\dimexpr\f@size pt\relax}%
  \newcommand*\lineheight[1]{\fontsize{\fsize}{#1\fsize}\selectfont}%
  \ifx\svgwidth\undefined%
    \setlength{\unitlength}{601.54046631bp}%
    \ifx\svgscale\undefined%
      \relax%
    \else%
      \setlength{\unitlength}{\unitlength * \real{\svgscale}}%
    \fi%
  \else%
    \setlength{\unitlength}{\svgwidth}%
  \fi%
  \global\let\svgwidth\undefined%
  \global\let\svgscale\undefined%
  \makeatother%
  \begin{picture}(1,0.40531504)%
    \lineheight{1}%
    \setlength\tabcolsep{0pt}%
    \put(0,0){\includegraphics[width=\unitlength,page=1]{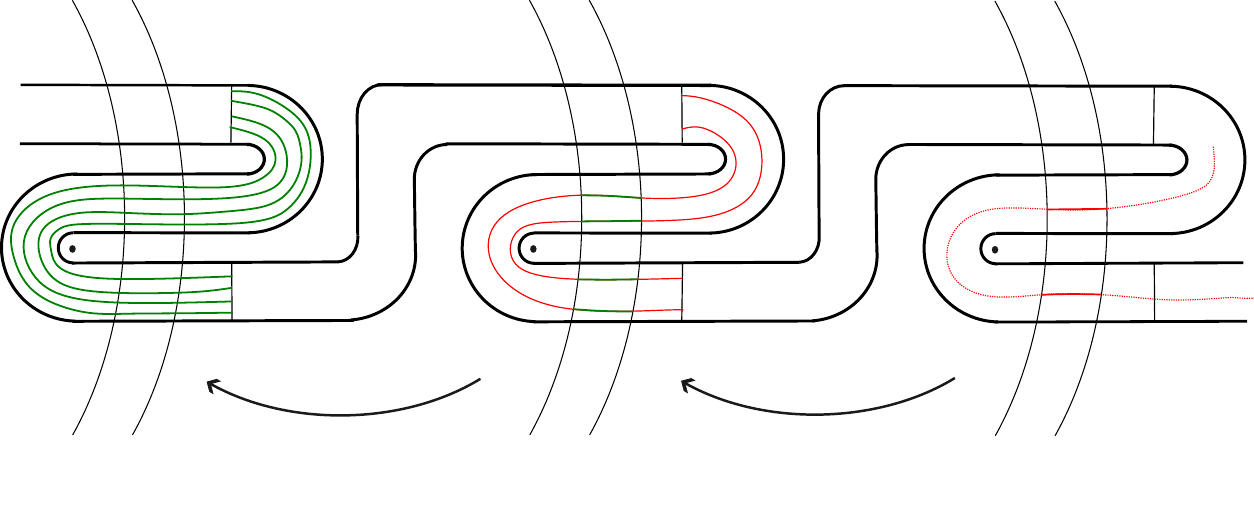}}%
    \put(0.23178528,0.00692788){\color[rgb]{0.10196078,0.10196078,0.10196078}\makebox(0,0)[lt]{\lineheight{1.25}\smash{\begin{tabular}[t]{l}$F^{-1}$\end{tabular}}}}%
    \put(0.60957415,0.00647758){\color[rgb]{0.10196078,0.10196078,0.10196078}\makebox(0,0)[lt]{\lineheight{1.25}\smash{\begin{tabular}[t]{l}$F^{-1}$\end{tabular}}}}%
  \end{picture}%
\endgroup%

\caption{The doubling of arcs in the proof of Theorem~\ref{nocurve}.}
\label{preimages}
\end{figure}
The following theorem is analogous to~\cite[Theorem 6.1]{rrrs} and proceeds in a very similar manner but it is provided to verify the result in the context of the tracts defined in Section~\ref{sec:tracts}. Figures~\ref{tractmap} and~\ref{preimages} illustrate the behaviour of the function as discussed in the proof.
\begin{theorem}\label{nocurve}
 Let $F\in\mathcal{H}^{\Xi_{\mathcal{C}}}$, and let $\tilde{F}\in\Blogp$ be defined as in Proposition~\ref{prop:Blogp}. 
  Then $J(\tilde{F})$ contains no curve to $\infty$.
\end{theorem}
\begin{proof} Let $\xi = (r_j,R_j,\eps_j)_{j=0}^{\infty} \in \Xi_{\mathcal{C}}$ be the tract datum corresponding to $F$.
   Since $5 \in J(\tilde{F})$, $J(\tilde{F})$ is non-empty and so let $w_0$ be a point of $J(\tilde{F})$. We will show that there is no curve connecting 
 $w_0$ to $\infty$ in $J(\tilde{F})$. Using the notation from Proposition~\ref{prop:Blogp}, we have
 $\tilde{F}^n(w_0)\in \tilde{\mathcal{U}}$ for all $n\geq 0$, i.e.\ 
 $\tilde{F}^n(w_0) \in \tilde{T}_{s_n} =  \tilde{T} + 2\pi i s_n$ for some $s_n\in\Z$.

We begin with a simple observation about the geometry of the tracts. 

\begin{claim}[Claim 1]
 Let $j \geq 0$ and let $ \gamma \colon [0,1] \to T$ be a curve such that $ \gamma(0) \in W^+_j $ and $ \gamma(1) \in W^-_j $. Then there exists a $ t^* \in (0,1) $ such that $ \gamma(t^*) \in (r_j - \pi i/3 , r_j +1 - \pi i/3).$
\end{claim}
\begin{subproof} 
 The segment $I = (r_j - \pi i/3 , r_j +1 - \pi i/3)$ disconnects $T$, and separates $W^+_j$ from $W^-_j$ in $T$.  
   Indeed, the sets $\{z\in T\colon \re z \geq R_j\}\cup W_j^-$ and $\{z\in T\colon \re z < R_j\}\setminus (W_j^-\cup I)\supset W_j^+$ are
   disjoint open subsets of $T$, and their union is $T\setminus I$. The claim follows. 
\end{subproof}

Let $(\rho_j)_{j=1}^{\infty}$, ($\dot{\rho}_j)_{j=1}^{\infty}$, $(C_j)_{j=0}^{\infty}$ and $(\dot{C}_j)_{j=0}^{\infty}$ 
be sequences as in Definition~\ref{cxpleconds}. Also 
  define $C^m_j \defeq C_j + 2\pi im$ and $\dot{C}^m_j \defeq \dot{C}_j + 2\pi im$.

Recall that we have chosen an initial point $w_0\in J(\tilde{F})$ and that we wish to show that $w_0$ cannot be connected to
    infinity by a curve in $J(\tilde{F})$. Inductively define $w_{j+1} \defeq \tilde{F}(w_j)$ and fix $m\geq 0$ such that $\dot{C}_m^{s_0}$ separates $w_0$ from $\infty$ in $\tilde{T}_{s_0}$.
  
\begin{claim}[Claim 2]
  For all $j \geq 0$, $\dot{C}^{s_j}_{m+j}$ separates $w_j$ from $\infty$ in
   $\tilde{T}_{s_j}$ and $\abs{w_{j+1}} < \dot{\rho}_{m+j+1}$.
\end{claim}
\begin{subproof} First observe that the second part of the claim follows from the first. Indeed, suppose that 
  $\dot{C}^{s_j}_k$ separates $w_j$ from $\infty$ in $\tilde{T}_{s_j}$, for some $k\geq 0$. $\dot{C}_k$ then separates $w_j-2\pi i s_j$ from $\infty$ in $\tilde{T}$. Hence $F(\dot{C}_k)=\{z\in \HH\colon \lvert z\rvert = \dot{\rho}_{k+1}\}$ separates $F(w_j-2\pi i s_j) = \tilde{F}(w_j)=w_{j+1}$ from $\infty$ in $\HH$. In other words,
  $\abs{w_{j+1}} < \dot{\rho}_{k+1}$. 
  
   We now prove the first part of the claim by induction; it holds for $j=0$ by choice of $m$ and the preceding observation. 
   Now suppose that the claim holds for $j$; i.e. $\dot{C}^{s_j}_{m+j}$ separates $w_j$ from $\infty$ in $\tilde{T}_{s_j}$ and $\lvert w_{j+1} \rvert < \dot{\rho}_{m+j+1}$ which also gives $\re w_{j+1} < \dot{\rho}_{m+j+1}$. By the construction of the tract, the real parts of the points in the unbounded component of $\tilde{T}\setminus \dot{C}_{m+j+1}$ have real parts strictly larger than $\dot{\rho}_{m+j+1}$. This means $w_{j+1} - 2\pi i s_{j+1}$ is separated from $\infty$ by $\dot{C}_{m+j+1}$, therefore $w_{j+1}$ is separated from $\infty$ by $\dot{C}^{s_{j+1}}_{m+j+1}$ as claimed. 
\end{subproof}

%\begin{claim}[Claim 3]
% Let $j\geq 0$ and $k\geq 0$. If $\alpha$ is an arc in $T+ 2\pi i s_{j+1}$ connecting $C^{s_{j+1}}_{m+j+k}$ and
% $C^{s_{j+1}}_{m+j+k}$, then $F^{-1}(\alpha)$ contains two disjoint arcs connected
% $C_{m+j+k-1}$ and $\dot{C}_{m+j+k-1}$. 
%\end{claim}

\begin{claim}[Claim 3]
 Let $n\geq m$. If $\alpha$ is an arc connecting $C_{n+1}$ and $\dot{C}_{n+1}$ in $T$,
  and $1\leq j\leq n-m$, then $F^{-1}(\alpha + 2\pi i s_j)$ contains two disjoint arcs connecting
  $C_n$ and $\dot{C}_n$. 
\end{claim}
\begin{subproof}
 By Claim 1, there is a point $z\in \alpha$ that lies on the segment $(r_{n+1} - \pi i/3, r_{n+1} + 1 - \pi i/3)$. 
  We have 
     \[ \lvert z\rvert \leq \re z + \lvert \im z\rvert < r_{n+1} + 1 + 2\pi \lvert s_j\rvert + \pi/3. \]
   By Claim 2,  
      \[ \dot{\rho}_{n} \geq \dot{\rho}_{m+j} > \lvert w_{j}\rvert \geq \lvert \im w_{j}\rvert > 2\pi \lvert s_{j} \rvert - \pi . \]
   So 
     \[ \lvert z\rvert < r_{n+1} + 1 + 4\pi/3 + \dot{\rho}_n \leq r_{n+1} + 1 + 2\pi + R_n \leq \rho_{n+1}, \]
   by the definition of a counterexample tract datum. 
   
  On the other hand, all points of $C^{s_j}_{n+1}$ and $\dot{C}^{s_j}_{n+1}$ have real part, and hence modulus, greater than
   $\dot{\rho}_{n+1}$. It follows that the piece of $\alpha + 2\pi is_j$ connecting 
   $C^{s_j}_{n+1}$ to $z$ contains an arc connecting $F(C_n)$ and $F(\dot{C}_n)$, and likewise
   for the piece connecting $z$ to $\dot{C}^{s_j}_{n+1}$. The claim follows.
\end{subproof}

Now suppose for a contradiction that there is a curve $\gamma \subset J(\tilde{F})$ 
  that connects $w_0$ to $\infty$. Let $j\geq 0$. Then by Claim~2, 
   the curve $\tilde{F}^j(\gamma)$ contains an arc connecting $C^{s_j}_{m+j+1}$ and $\dot{C}^{s_j}_{m+j+1}$. 
   If $j\geq 1$, then by Claim~3 it follows that 
   $\tilde{F}^{j-1}(\gamma)$ contains two disjoint arcs connecting $C^{s_{j-1}}_{m+j}$ and $\dot{C}^{s_{j-1}}_{m+j}$. 
   By inductive use of Claim~3, we conclude that $\gamma$ contains $2^j$ disjoint arcs connecting
   $C^{s_0}_{m+1}$ and $\dot{C}^{s_0}_{m+1}$. This is illustrated by Figure~\ref{preimages}.

 Let
   \[ r \colon [0,\infty)\to \gamma \]
   be a parameterisation of $\gamma$ such that $r(0) = w_0$ and $r(t)\to\infty  $ as $t\to\infty$. 
   We may choose the parameterisation such 
    that $r(t) \notin C^{s_0}_{m+1}\cup \dot{C}^{s_0}_{m+1}$ for all $t>1$. 
    
   By uniform continuity, there is an $\eps>0$ such that $\lvert a-b\rvert \geq \eps$ whenever 
      $\lvert r(a) - r(b)\rvert \geq \dist(C^{s_0}_{m+1},\dot{C}^{s_0}_{m+1})$ (Euclidean distance). This means that
      $r([0,1])$, and hence $\gamma$, contains at most $\lfloor 1/\eps\rfloor$ pairwise disjoint 
      pieces connecting $C^{s_0}_{m+1}$ and $\dot{C}^{s_0}_{m+1}$. Since $j$ was arbitrary in the above, 
      this is a contradiction, and the proof of the theorem is complete. 
\end{proof}

\section{Growth of functions in $\mathcal{H}^\Xi$}\label{sec:growth}
Recall from the introduction that we wish to study the growth of $\Blog$ functions which requires studying $\log \re F$. In this section we will prove a result that gives us an estimate on the growth of $\log\lvert F (z) \rvert$ which serves us well. We estimate the size of $\log \lvert F(z) \rvert$ for an $F \in \mathcal{H}^{\Xi}$, particularly within the wiggling sections, $W_j$,  of the corresponding tract $T$ which were introduced in Definition~\ref{defn:W} and depicted in Figure~\ref{wj}. It is seen that the growth in these sections is determined (up to a constant) by the values of $R_j$ and $\varepsilon_j$. This shows that the geometry of the tract $T$ is directly associated with the growth of the corresponding conformal isomorphism $F$.
\begin{figure}[htbp]
\centering
\label{tractwithpath}
\def\svgwidth{1\textwidth}
%% Creator: Inkscape 1.3 (0e150ed6c4, 2023-07-21), www.inkscape.org
%% Accompanies image file 'tractwithpath.pdf' (pdf, eps, ps)
\begingroup%
  \makeatletter%
  \providecommand\color[2][]{%
    \errmessage{(Inkscape) Color is used for the text in Inkscape, but the package 'color.sty' is not loaded}%
    \renewcommand\color[2][]{}%
  }%
  \providecommand\transparent[1]{%
    \errmessage{(Inkscape) Transparency is used (non-zero) for the text in Inkscape, but the package 'transparent.sty' is not loaded}%
    \renewcommand\transparent[1]{}%
  }%
  \providecommand\rotatebox[2]{#2}%
  \newcommand*\fsize{\dimexpr\f@size pt\relax}%
  \newcommand*\lineheight[1]{\fontsize{\fsize}{#1\fsize}\selectfont}%
  \ifx\svgwidth\undefined%
    \setlength{\unitlength}{323.1857444bp}%
    \ifx\svgscale\undefined%
      \relax%
    \else%
      \setlength{\unitlength}{\unitlength * \real{\svgscale}}%
    \fi%
  \else%
    \setlength{\unitlength}{\svgwidth}%
  \fi%
  \global\let\svgwidth\undefined%
  \global\let\svgscale\undefined%
  \makeatother%
  \begin{picture}(1,0.14262786)%
    \lineheight{1}%
    \setlength\tabcolsep{0pt}%
    \put(0,0){\includegraphics[width=\unitlength,page=1]{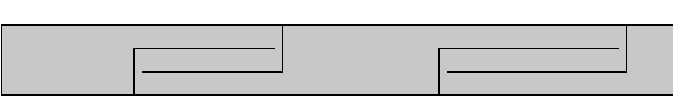}}%
    \put(-0.00281017,0.11405493){\color[rgb]{0,0,0}\makebox(0,0)[lt]{\lineheight{1.25}\smash{\begin{tabular}[t]{l}$T$\end{tabular}}}}%
    \put(0,0){\includegraphics[width=\unitlength,page=2]{tractwithpath.pdf}}%
    \put(0.00632637,0.07234148){\color[rgb]{0,0,0}\makebox(0,0)[lt]{\lineheight{1.25}\smash{\begin{tabular}[t]{l}$5$\end{tabular}}}}%
    \put(0,0){\includegraphics[width=\unitlength,page=3]{tractwithpath.pdf}}%
    \put(0.07798562,0.06297573){\color[rgb]{0,0,0}\makebox(0,0)[lt]{\lineheight{1.25}\smash{\begin{tabular}[t]{l}$\alpha$\end{tabular}}}}%
  \end{picture}%
\endgroup%

\caption{The path $\alpha$ as seen in the tract $T$.}
\end{figure}
\begin{theorem}\label{growth}
There exists an absolute constant $C>1$ such that, for any $\xi \in \Xi$ with corresponding tract $T$ and conformal isomorphism $F \colon T \to \HH$, the following inequalities hold for all $j \geq 0$. 

For $z \in W_j$:

\[ \frac{1}{C} \left( R_j + \sum_{k=0}^{j} \log \left( \frac{1}{\varepsilon_k} \right) \right) \leq \log \abs{F(z)} \leq  C \left( R_j + \sum_{k=0}^{j} \log \left( \frac{1}{\varepsilon_k} \right) \right). \]

For $z \in U_{j+1} \defeq \lbrace\zeta \in T \colon R_j < \re \zeta < \tau_{j+1} -\nu_0 \rbrace\setminus W_{j+1}$: \[ \frac{1}{C} \left( \re z + \sum_{k=0}^{j} \log \left( \frac{1}{\varepsilon_k} \right) \right) \leq \log \abs{F(z)} \leq  C \left( \re z + \sum_{k=0}^{j} \log \left( \frac{1}{\varepsilon_k} \right) \right). \]
\end{theorem}

\begin{proof}
  Consider the arc that connects $5$ to $\infty$ within $T$ defined in the following way:
    \begin{align*}
        \alpha \defeq [5, r_0 - \frac{1}{2}] &\cup \bigcup_{k \geq 0}  \bigl( r_k - \frac{1}{2}, i[0,2\pi/3] \bigr) \cup \bigl( [r_k - \frac{1}{2}, R_k - \frac{1}{2}] + 2\pi i/3 \bigr) \\ &\cup \bigl( R_k - \frac{1}{2} + i[0,2\pi/3] \bigr) \cup [r_k + \frac{1}{2}, R_k - \frac{1}{2}] \\ &\cup \bigl( r_k + \frac{1}{2} + i[-2\pi/3, 0] \bigr) \cup ([r_k + \frac{1}{2}, R_k + \frac{1}{2}] - 2\pi i/3 ) \\& \cup \bigl( R_k + \frac{1}{2} + i[-2\pi/3, 0]\bigr) \cup [R_k + \frac{1}{2}, r_{k+1} - \frac{1}{2}] 
    \end{align*}

We will split $\alpha$ into a part $\alpha^0$ that consists of the pieces that pass through the gates, and which must hence (if $\eps_j$ is small) pass close to the boundary of
$T$, and a complementary part $\alpha^1$, which stays away from $\partial T$ by a definite amount. More precisely, write $\alpha = \alpha^0 \cup \alpha^1$ as follows:
\begin{align}
    \alpha^0 &\defeq \bigcup_{k\geq0}\lbrace z \in \alpha \colon \tau_k - 1 \leq \re z \leq \tau_k +1, \im z = 2\pi/3 \rbrace \label{alpha0}\\
    \alpha^1 &\defeq \alpha \setminus \alpha^0 \label{alpha1}
\end{align}

For $z \in \alpha$, let us denote the part of $\alpha$ that connects $5$ to $z$ by $\alpha_z$. 
Similarly write $ \alpha_z = \alpha_z^0 \cup \alpha_z^1$ where $\alpha_z^0 = \alpha_z\cap  \alpha^0$ and $\alpha_z^1 = \alpha_z \cap \alpha^1$. 
Now let $z\in \alpha\cap W_j$ for some $j\geq 1$. We will 
 estimate the hyperbolic length of $\alpha_z$ in $T$, $\ell_T(\alpha_z)$, by estimating those of $\alpha^0_z$ and $\alpha^1_z$ separately.

Since $\alpha_z^1$ stays away from the boundary and is the section of $\alpha_z$ that stays away from the gates by construction in~\eqref{alpha1}, by the standard estimate~\eqref{eqn:standardestimate} its hyperbolic length is comparable to its Euclidean length, $\ell_E(\alpha_z^1)$,
  which in turn is comparable to $R_j$ by the definition of $\alpha$. More precisely: 
  
\begin{claim}[Claim 1]
 $\frac{1}{4\pi}R_j \leq \ell_T(\alpha_z^1) \leq 16R_j$.
\end{claim}
\begin{subproof}
By definition of $\alpha$,   
\begin{align*}
    \ell_E(\alpha_z^1)    &\leq R_j - 5 + \frac{8\pi}{3}j + 2\pi + 2\sum_{k=0}^j \bigl( R_k - r_k - 3\bigr)\\
					&\leq R_j - 5 + 9(j+1) + 2 \sum_{k=0}^j (R_k - r_k)  \\
					&\leq R_j - 5 + 9(j+1) + 2 R_j 
					\leq \frac{31R_j}{10} + \frac{2}{5}
					\leq 4R_j.
\end{align*}
Similarly, 
\[ \ell_E(\alpha_z^1)    \geq R_j - \frac{1}{2} - 5 + \frac{8\pi}{3}j + \pi + 2\sum_{k=0}^{j-1} (R_k - r_k -3)\\
					\geq R_j - 5. \]

Overall, since $R_j > 36$, we conclude that
 \[ \frac{R_j}{2} \leq R_j - 5 \leq \ell_E(\alpha_z^1) \leq 4R_j. \]

Recall that the distance to the boundary
  of any point $\zeta\in \alpha_z^1$ is at least $1/2$ and at most $\pi$. By the standard estimate~\eqref{eqn:standardestimate}, 
    we have $1/2\pi \leq \lambda_T(\zeta)\leq 4$.  The claim follows. 
\end{subproof}

Now we turn to estimating the hyperbolic length of 
$\alpha_z^0$, which is the section of $\alpha_z$ that passes through the epsilon gates. We expect the hyperbolic length of this section to be considerably larger for small values of $\eps_j$.
\begin{claim}[Claim 2]
	\begin{align*}
		\frac{1}{2} \sum_{k=0}^j \log \left( \frac{1}{\varepsilon_k} \right) &\leq \ell_T(\alpha_z^0 ) \leq 4\left( \frac{R_j}{25} + \sum_{k=0}^j \log \left( \frac{1}{\varepsilon_k}\right) \right).
	\end{align*}
\end{claim}
\begin{subproof}
	
Let $\hat{\alpha}_k$ be the piece of
  $\alpha_z^0$ that belongs to the $k$-th wiggle, that is, the piece that passes through the $k$-th epsilon gate; i.e.,  $\hat{\alpha}_k = [\tau_k-1,\tau_k+1] + 2\pi i/3$. If we define 
  $\mu_k(t) = \max(\lvert t-\tau_k \rvert, \pi\varepsilon_k/3)$, then for $z\in \hat{\alpha}_k$, 
    \begin{align}
       \label{eqn:mukestimate}
        \mu_k(\re z) &= \max(\lvert \re z - \tau_k\rvert , \pi \varepsilon_k/3) \leq \dist(z,\partial T) \\ \notag &\leq \lvert \re z - \tau_k \rvert + \pi \varepsilon_k/3 \leq 2 \mu_k(\re z). \end{align}
  We have 
   \begin{align*} \int_{\hat{\alpha}_k} \frac{\lvert d z\rvert}{\mu_k(z)} &= -2\int_{1}^{\pi\eps_k/3} \frac{dt}{t} + \int_{-\pi\eps_k/3}^{\pi\eps_k/3} \frac{3dt}{\pi\eps_k} \\&=
        2\log\left( \frac{3}{\pi\eps_k}\right) + 2 \leq  2\log \left(\frac{1}{\eps_k}\right) + 2. \end{align*}
  By the standard estimate~\eqref{eqn:standardestimate} and by~\eqref{eqn:mukestimate}, \[1/(4\mu_k(\re z)) \leq \lambda_T(z) \leq 2/\mu_k(\re z)\] for
  $z\in \alpha^0_k$. Hence 
    \[    \frac{1}{2}\log\left( \frac{1}{\eps_k}\right) \leq \ell_T( \hat{\alpha}_k) \leq 4(\log\left( \frac{1}{\eps_k}\right) +1).  \]
We can then write \begin{align*}
\frac{1}{2} \sum_{k=0}^j \log \left(  \frac{1}{\varepsilon_k} \right) &\leq \ell_T(\alpha_z^0 ) \leq 4\left( j+1 + \sum_{k=0}^j \log \left( \frac{1}{\varepsilon_k}\right) \right) \\
\frac{1}{2} \sum_{k=0}^j \log \left( \frac{1}{\varepsilon_k} \right) &\leq \ell_T(\alpha_z^0 ) \leq 4\left( \frac{R_j}{25} + \sum_{k=0}^j \log \left( \frac{1}{\varepsilon_k}\right) \right).
\end{align*}

\end{subproof}
Combining these estimates with those of Claim 1, we find \begin{align*}
\frac{1}{4\pi}R_j + \frac{1}{2} \sum_{k=0}^j \log \left( \frac{1}{\varepsilon_k} \right) &\leq \ell_T(\alpha_z ) \leq 16R_j +  4\left( \frac{R_j}{25} + \sum_{k=0}^j \log\left(  \frac{1}{\varepsilon_k}\right) \right).
\end{align*}
For a general $z \in W_j$, let us consider the vertical geodesic $ \Gamma_{\abs{F(z)}}$ and let $\tilde{z} \defeq F^{-1}(\abs{F(z)})$. The Euclidean length of any path of any curve connecting $5$ to $z$ must be at least $R_j - 6$. Using this and considering the minimal contribution from passing through the gates, as seen in the discussion above, we can say \begin{align*} \dist_T(5, \tilde{z}) &\geq \frac{1}{4\pi}(R_j - 6 -  (j+1)) + \frac{1}{2} \sum_{k=0}^j \log \left( \frac{1}{\varepsilon_k} \right) \\&\geq \frac{1}{25}\left( R_j +  \sum_{k=0}^j \log \left( \frac{1}{\varepsilon_k} \right) \right). \end{align*}
We need to provide an upper-bound on $\dist_T(5,\tilde{z})$. From Lemma~\ref{geodes} and Corollary~\ref{nu0is60} we know that \[ \lbrace z \in T \colon R_j + 1/2 \leq \re z \leq R_j + 1/2 + 8\pi \rbrace \] contains the vertical geodesic that passes through the midpoint \\$w \defeq R_j + 1/2 + 4\pi \in \alpha$ and also separates $z$ and $\tilde{z}$ from $\infty.$ We estimate the hyperbolic length of $\alpha_w \setminus \lbrace z \in T \colon \re z \leq R_j \rbrace.$ \[ \dist_T(R_j - \pi i/3, w) \leq 4 \left( \frac{1}{2} + \frac{2\pi}{3} + 4\pi \right) < 65.\] Thus\begin{align*} \dist_T(5, \tilde{z}) \leq \dist_T(5,z) &\leq \dist_T(5, w) \\&\leq 16R_j +  4\left( \frac{R_j}{25} + \sum_{k=0}^j \log \left( \frac{1}{\varepsilon_k}\right) \right) + 65.\end{align*}
Since $F$ is a conformal isomorphism, we use~\eqref{Hlength} to write
\[ \dist_T(5,\Gamma_{\abs{F(z)}}) = \dist_T(5,\tilde{z}) = \dist_\HH(5, \abs{F(z)}) = \log \abs{F(z)} - \log 5.\]
We can then say
\begin{align*}
\frac{1}{25}\left(R_j +  \sum_{k=0}^j \log \left( \frac{1}{\varepsilon_k} \right) \right) + &\log 5  \\ &\leq \log\abs{F(z)} \\&\leq 25 \left( R_j + \sum_{k=0}^j \log \left( \frac{1}{\varepsilon_k} \right)\right) + \log 5 + 65 \end{align*}
and thus, for suitably large enough choice of $r_0$,
\[ \frac{1}{30}\left(R_j +  \sum_{j=0}^j \log \left( \frac{1}{\varepsilon_k} \right) \right) \leq \log\abs{F(z)} \leq 30 \left( R_j + \sum_{k=0}^j \log \left( \frac{1}{\varepsilon_k} \right)\right).\]
The second inequality follows by use of the first inequality and by making similar estimates for the lengths of any path leading (almost) up to the next epsilon gate.
\end{proof}
We will refer to the constant $C$ rather than take any particular value (e.g. $30$ in the final steps of the proof) so that it is clear when this growth estimate is being used.

The parts of the tracts which we haven't produced estimate for are those \emph{near} the gates. Since we have been able to restrict our `difficult' regions to ones of uniform size, these do not pose an issue when it comes to making order of growth estimates given the vanishingly small proportions compared with $R_j$ as $j$ increases particularly as the order of growth is always compared to the real part of $z$.

Now that we are furnished with these growth estimates, it is possible for us to study the growth of $\log \re \tilde{F} (z)$ where $\tilde{F}$ is the disjoint type $\Blog^p$ function derived from $F$ as per Proposition~\ref{prop:Blogp}.

\section{Relating growth and geometry}\label{sec:growthgeom}

This section details how the geometry of a tract $T$ can influence the order of growth of the corresponding conformal isomorphism $F$ and which conditions we should impose on $(r_j)_{j=0}^\infty$ and $(R_j)_{j=0}^\infty$ in order to ensure we satisfy Definition~\ref{cxpleconds} and such that the corresponding $F$ has the desired order of growth.
\subsection{Heuristics for estimating the order of growth while satisfying counterexample conditions}\label{heuristic} \hfill \\ Recall from Theorem~\ref{growth} that 

 \[ \frac{1}{C} \left( R_j + \sum_{k=0}^{j} \log \left( \frac{1}{\varepsilon_k} \right) \right) \leq \log \abs{F(z)} \leq  C \left( R_j + \sum_{k=0}^{j} \log \left( \frac{1}{\varepsilon_k} \right) \right) \] for $z \in W_j$. This suggests that if we can control $(R_j)_{j=0}^\infty$ and $(\eps_j)_{j=0}^\infty$ well enough, we can potentially control $\log \abs{F(z)}$ and hence $\log \re F(z)$. In this section we examine, \emph{very} informally, how $(r_j)_{j=0}^\infty$,  $(R_j)_{j=0}^\infty$, and $(\eps_j)_{j=0}^\infty$ impact $F$. The influence of $(\eps_j)_{j=0}^\infty$ on the growth of $F$ is studied primarily in Section~\ref{sec:gateselection}.
This section produces the rough calculations giving estimates of the growth of $F$ in terms of the endpoints of the wiggles that suggest suitable recurrence relations to use for $(r_j)_{j=0}^\infty$ and $(R_n)_{j=0}^\infty$ that will lead to desirable orders of growth. It is not intended to provide a formal basis but rather to indicate the heuristic reasoning for the choice of recurrence relations seen later in Section~\ref{growthfunctiontotract}. The rest of the section (and paper) is then spent on refining our choices with technical results. The reader is reminded to keep Figure~\ref{tractmap} in mind throughout.

\subsubsection{Recurrence relations}\label{recrels} \hfill \\
In~\cite[Proposition 8.1]{rrrs}, it was shown that by imposing a certain recurrence relation between the endpoints of wiggle sections one could ensure the resulting function was a counterexample to the strong Eremenko conjecture while also providing enough information to calculate the order of growth. The analogue of this in the context of our tracts would be finding recurrence relations $R_j \mapsto r_{j+1}$ and $R_j \mapsto R_{j+1}$, which says that the real part of the end of a wiggle section determines both of the endpoints of the following one. In our attempt to apply this method in a slightly generalised setting, it shows how the recurrence relations that map $r_j \mapsto R_{j}$ and $R_j \mapsto r_{j+1}$ interact with each other, noting the difference from those mentioned above where each endpoint determines the following one. This plays an important role in how we estimate our counterexample function $F$ and in fact raises questions for future work.

We suppose that our sequences $(r_j)_{j=0}^\infty$ and $(R_j)_{j=0}^\infty$ could be related by two  relationships via two real functions $\varphi \colon [0,\infty) \to [0,\infty)$ and $\psi \colon [0,\infty) \to [0,\infty)$ which we write now. An initial idea could be to take relationships of the form \begin{equation}\label{phiandpsi}
 \log r_{j+1} = \varphi(R_j) \text{ and } R_{j} = \psi(r_j). 
\end{equation}

We will require that \[ \psi(t)> t + \max \lbrace 5 + 3\nu_0, 30 \rbrace , \text{ and } \varphi(t)  >\log ( t + \max\lbrace 2\nu_0, 60 \rbrace) \] in order to satisfy Definition~\ref{defn:rRspacing}. Clearly this requires both of $\varphi$ and $\psi$ to tend to infinity as $t \to \infty$ at the very least and we will require both to be nondecreasing. The reason why we require $\varphi$ to exceed $\exp(t)$ becomes clear when we take growth estimates later in Section~\ref{sec:phivspsi} but is very much tied up with the fact that we want to have infinite-order growth. Besides hoping that this furnishes us with a function $F$ satisfying a desirable growth condition, we will still need to ensure $F$ satisfies the counterexample conditions of Definition~\ref{cxpleconds}. 

\subsubsection{Counterexample conditions}
If our tract is constructed by $\varphi$ and $\psi$ above in~\eqref{phiandpsi} \emph{and} if the corresponding $F$  is to satisfy the conditions of Definition~\ref{cxpleconds} we will require the following corresponding conditions for $j\geq0$, once more keeping Figure~\ref{tractmap} in mind:
\begin{itemize}
\item $\abs{F(C_j)}$ should be bigger than $r_{j+1}=\exp(\varphi(R_j))$.
\item $\abs{F(\dot{C}_j)}$ should be smaller than $R_{j+1} = \psi(r_{j+1}) = \psi(\exp(\varphi(R_j))).$
\end{itemize}
In  order to ensure we achieve our counterexample conditions and to exceed finite-order growth, the size of $r_j$ compared to $R_j$ needs to be very small, making the hyperbolic distance between the geodesics $C_j$ and $\dot{C}_j$ comparable to $R_j$. This means we can expect $R_{j+1} \geq \exp(\varphi(R_j) + cR_j)$ for some constant $c>1$. So we will need \[ \psi(r_{j+1}) > \exp( \varphi(R_j) + cR_j) = r_{j+1}\exp(cR_j).\] This should make it apparent that $\varphi$ and $\psi$ are now dependent on each other.

The overall growth of $F$ will be bounded as \begin{equation}\label{esteelorder} \log \re F(z) = O(\psi(\exp(\varphi(\re z)))). \end{equation}

What remains is for us to:
\begin{enumerate}
\item Show that we can choose our gates such that there is a ``first'' vertical geodesic $C_j$ in $W_j^+$ (recall Definition~\ref{defn:W} and~\eqref{W+}) that maps to a semi-circular geodesic in $\HH$ with radius approximately $\exp(\varphi(R_j))$.
\item Show that we can ensure there is a ``second'' vertical geodesic $\dot{C}_j$ in $W_j^-$ that maps to a semi-circlular geodesic in $\HH$ with radius less than $R_{j+1}$ by a definite amount (Theorem~\ref{thm:preshoot}).
\item Verify that, assuming the $\psi$ satisfies the growth condition derived above, this gives a function $F$ satisfying the counterexample conditions in Definition~\ref{cxpleconds}.
\item Given $\Theta$ as in Theorem~\ref{paracxple}, we can provide $\varphi$ and $\psi$ with the required conditions such that $\psi(\exp(\varphi(\re z) = O((\re z)^{1 + \Theta(\re z)})$ meaning that, with~\eqref{esteelorder}, we achieve $\log \re F(z) = O((\re z)^{1 + \Theta(\re z)}).$
\end{enumerate}

If $F$ is to be a counterexample function of infinite order we need to consider what happens at points in $W_j$ near $r_j$ because that is the part of the tract where $\re F(z)$ is largest in terms of $\re z$.

If $z \in W_j$ has real part approximately equal to $r_j$ and we require $\abs{F(z)} < R_{j+1}$ in order to satisfy Definition~\ref{cxpleconds}, we require the following to hold
\[\log \abs{F(z)} < \log R_{j+1} = \log \left( \psi ( \exp (\varphi (\psi (r_j))))\right). \]
If we can also satisfy \[ \log \left( \psi (\exp (\varphi (\psi (r_j))))\right) = O \left( (r_j)^{1 + \Theta(r_j)}\right) \]then it ensures $F$ has the appropriate order of growth and satisfies one of the key counterexample conditions.
Following the labelling in Figure~\ref{tractmap} and Definition~\ref{cxpleconds}, we discuss one further consideration for $F$ with respect to counterexample conditions. We need to ensure that the spacing between $R_{j+1}$ and $r_{j+1}$ is adequate enough so that $\abs{F(C_j)}$ and $\abs{F(\dot{C}_j)}$ can fit between the end sections of the wiggles. This equates to satisfying the following chain of inequality \begin{equation}\label{why9rj}
r_{j+1} < \abs{F(C_j)} < \abs{F(\dot{C}_j)} < R_{j+1}.
\end{equation}

Considering the distance between $C_j$ and $\dot{C}_j$, we provide an upper bound in the following way. Consider the path $\alpha$ from the proof of Theorem~\ref{growth} again, specifically the segment travelling between $C_j$ and $\dot{C}_j$. The distance from the boundary is at least $1/2$ and the Euclidean length can be bounded above by $2(R_j - r_j) + 2\pi/3$. Making use of the standard estimate once more we can write \begin{align*} \dist_T(C_j, \dot{C}_j)&= \dist_\HH(\abs{F(C_j)},\abs{F(\dot{C}_j)})  \\&=\log \abs{F(C_j)} - \log \abs{F(\dot{C}_j)} \\  &< 4 \left(2(R_j - r_j) + \frac{2\pi}{3}\right) < 8R_j. \end{align*}
This means that, if we can guarantee $\log R_{j+1} - \log r_{j+1} > 8R_j$, then we know that there will be adequate spacing. This by itself does not guarantee that we satisfy~\eqref{why9rj}.

We make the following amendment to~\eqref{phiandpsi} and suggest a new relation to use

\begin{equation}\label{phipsitaketwo}
\log r_{j+1} = \varphi(R_j) \text{ and } \log R_{j+1} = \varphi(R_j) + 9R_j.
\end{equation}
This does not mean we are finished with $\psi$. We have simply been able to relate $\varphi$ and $\psi$ with this final consideration for the counterexample conditions of Definition~\ref{cxpleconds}. This will be explored further in the rest of this paper.

\subsubsection{Infinite order}\label{infiniteordergrowthfunction}

One goal of this paper is to ensure $F$ satisfies \[ \log \re F(z) = O\left( \left(\re z \right)^{1 + o(1)}\right) \text{ as } \re z \to \infty.\] By Definition~\ref{Ffiniteorder}, in order for $F$ to be infinite order, this means that \[ \frac{\log \re F(z)}{ \re z} \to \infty \text{ as } \re z \to \infty.\]

One way to achieve this is to introduce $h \colon [0,\infty) \to [0,\infty)$ as a placeholder \[ \log \re F(z) = O\left((\re z)^{1+h(\re z)}\right).\] Where we let $h(t) \to 0$ as $t \to \infty$. In order to maintain an infinite order of growth, we require

\[ \frac{\log \re F(z)}{ \re z} = (\re z)^{h(\re z)} \to \infty \text{ as } \re z \to \infty.\]

By taking a logarithm we find that the following also needs to be satisfied
\begin{equation}h(\re z) \log \re z \to \infty \text{ as } \re z \to \infty.
\end{equation}

This condition will be used later once we define our \emph{growth function} in Section~\ref{growthfunctiontotract}. Note that the suggested definitions of $\varphi$ and $\psi$ are not final with respect to the recurrence relations desired between $(r_j)_{j=0}^\infty$ and $(R_j)_{j=0}^\infty$ but are very close to the ones we use (see Definition~\ref{rR} and Definition~\ref{psi}). The utility of this section was to provide a heuristic reasoning for the approximate form of what they ultimately take.
\subsection{Using a growth function to define a tract}\label{growthfunctiontotract}
Let us now make the following assumption that will hold for the rest of the paper in all occurrences of $\varphi$ and $\Phi$.
\begin{standingassumption}\label{standingassumption}
 Let $\Phi \colon [0, \infty) \to (0, \infty)$ be a non-increasing function such that
\begin{align}\label{Phi1}
  &\lim_{t \to \infty} \Phi(t) = 0, \\
  &\Phi(t)\log t  \text{ is eventually strictly increasing, and }\label{assump2} \\ 
  &\lim_{t \to \infty} \Phi(t)\log t = \infty.
\end{align}
We define
 \begin{equation}\label{phi}
    \varphi(t) \defeq t^{1 + \Phi(t)}. 
 \end{equation}
\end{standingassumption}
Recalling that class $\classB$ functions of finite order automatically satisfy the strong Eremenko conjecture~\cite[Theorem 1.2]{rrrs}, the condition $\lim_{t \to \infty} \Phi(t) \log t = \infty$ will ensure the order of growth of the resulting $F$ remains infinite, as discussed in Section~\ref{infiniteordergrowthfunction}. This assumption will be updated but only once we have been able to prove some useful results. 
\begin{prop}\label{phiprops}
The following properties hold: 
\begin{itemize}
	\item $\varphi(t)$ is eventually strictly increasing.
	\item There is an $A \in (1,\infty)$ such that $\varphi(t) \leq t^A$ for all $t \geq 1$.
	\item $\lim_{t \to \infty}\varphi(t)/t = \infty.$
\end{itemize}
\end{prop}
\begin{proof}
We can write \[ \varphi(t) = \exp ( (1 + \Phi(t))\log t).\] Since $\Phi(t) \log t$ is eventually strictly increasing by assumption, this gives us the first property.

Since $\Phi$ is non-increasing, it is apparent that $1 + \Phi(0) \geq 1 + \Phi(t)$ for all $t \geq 0$. Now define $A \defeq 1 + \Phi(0) \in (1,\infty).$ We can then see \[ \varphi(t) = t^{1 + \Phi(t)} \leq t^{A} \] for all $t\geq 1$ which gives us the second property.

For the third property note that we can write \\ $\varphi(t)/t = t^{\Phi(t)} = \exp(\Phi(t) \log t).$ By considering limits we see \[ \lim_{t \to \infty} \frac{\varphi(t)}{t} = \lim_{t \to \infty} \exp(\Phi(t) \log t) = \infty\] since we assumed $\lim_{t \to \infty}\Phi(t) \log t = \infty.$ 
\end{proof}

Since $\varphi$ is eventually strictly increasing, that means its inverse eventually exists, which we will now find an estimate for.

\begin{prop}\label{prop:phiinv} 
	Let $t_0 > 0$ such that $\varphi(t)=w$ is strictly increasing on $[t_0,\infty)$ and let $M>1$. There exists $w_1> w_0=\varphi(t_0)$ such that \[ w^{1 - \Phi(w)} < \varphi^{-1}(w) < w^{1 - \Phi(w)/M} \] for all $w > w_1$.
\end{prop}
\begin{proof}
	Let $t > t_0$ and since $\varphi(t) = w$ we can write \[ t = w^{\frac{1}{1 + \Phi(t)}} = w ^{1 - \frac{\Phi(t)}{1 + \Phi(t)}}.\]
	
	We know $\Phi(\varphi(t)) = \Phi(w)$ and $w > t$ so $\Phi(t) \geq \Phi(w)$. Using~\ref{assump2} there also exists some $t^* > t_0$ such that we can make the following deduction for values of $t > t^*$,
	\begin{align*}
		\Phi(\varphi(t))\log \varphi(t) &> \Phi(t) \log t, \\
		\Phi(w)(1 + \Phi(t))\log t &> \Phi(t) \log t,  \\
		\Phi(w) &> \frac{\Phi(t)}{1 + \Phi(t)}, \\
		1 - \frac{\Phi(t)}{1 + \Phi(t)} &> 1 - \Phi(w).
	\end{align*}
This allows us to write \[ t = w ^{1 - \frac{\Phi(t)}{1 + \Phi(t)}} > w^{1 - \Phi(w)}\] which is equivalent to $\varphi^{-1}(w) > w^{1 - \Phi(w)}$ and this inequality will hold for $w>\varphi(t^*).$

To achieve the other side of the inequality let $M>1$ and let $t^{**} > t_0$ be such that $1 + \Phi(t) < M$ for all $ t > t^{**}$. We can then write the following for all $t > t^{**}$,

\begin{align*}
	\frac{\Phi(w)}{M} &< \frac{\Phi(w)}{1 + \Phi(t)} \leq \frac{\Phi(t)}{1 + \Phi(t)}, \\
	1 - \frac{\Phi(t)}{1 + \Phi(t)} &\leq 1 - \frac{\Phi(w)}{1 + \Phi(t)} < 1 - \frac{\Phi(w)}{M}.
\end{align*}
We can finally say $\varphi^{-1}(w) < w^{1 - \frac{\Phi(w)}{M}}$ and by taking $w_1 \defeq \varphi(\max\lbrace t^*, t^{**}\rbrace$) the proof is complete.
\end{proof}

\begin{defn}\label{protopsi}
	Let $t_0 >0$ be such that $\varphi$ is strictly increasing on $[t_0, \infty)$ and let $c \geq 1$. Define $\widehat{\Phi}_c(t) \colon [\exp(t_0), \infty) \to (0,\infty)$ to be \[ \widehat{\Phi}_c(t) \defeq \frac{c\varphi^{-1}(\log t)}{\log t}.\] We also define $\widehat{\varphi}_c(t) \defeq t^{1 + \widehat{\Phi}_c(t)}$ on $[t_0, \infty)$.
\end{defn}

We will eventually use $\widehat{\Phi}_c$ to define the $\Psi$ mentioned in Section~\ref{heuristic}.

\begin{prop}
	For all $c > 0$, $\widehat{\Phi}_c(\exp(t_0) + t)$ satisfies the conditions of Standing Assumption~\ref{standingassumption}.
\end{prop}
\begin{proof}
	First let $M>1$ and let $\tau = \exp(t_0) + t$. By use of Proposition~\ref{prop:phiinv} the following holds for large enough $\tau$, \[\widehat{\Phi}_c(\tau) = \frac{c\varphi^{-1}(\log \tau)}{\log \tau} < \frac{c(\log \tau)^{1- \Phi(\log \tau)/M}}{\log \tau}<c\left( \frac{1}{(\log \tau)^{\Phi(\log \tau)}}\right)^{1/M}.\]
	
	Which allows us to write \[ \log \widehat{\Phi}_c(\tau) < \log c - \frac{1}{M}\Phi(\log \tau)\log \log \tau. \]
From our assumptions, this decreases strictly to $-\infty$ meaning $\widehat{\Phi}_c$ satisfies the first condition of Standing Assumption~\ref{standingassumption}.

	 Note $\widehat{\Phi}_c(\tau)\log \tau = c\varphi^{-1}(\log \tau)$ eventually increases strictly to $\infty$ as $\varphi^{-1}$ is the inverse of a function that increases strictly to $\infty$, giving the second and third conditions.
	
\end{proof}

This means that there exists a large enough $t$ such that $\widehat{\varphi}^{-1}$ exists and we can further define a corresponding \[\mathring{\Phi}_{d, c}(t) \defeq \frac{d\widehat{\varphi}^{-1}_c(\log t)}{\log t}\] for $d, c >0$.

\begin{prop}\label{Philimits}
	The following limits hold \[ \lim_{t \to \infty} \frac{\Phi(\varphi(t))}{\Phi(t)} = \lim_{t \to \infty} \frac{\Phi(\widehat{\varphi}_c(t))}{\Phi(t)} = \lim_{t \to \infty} \frac{\Phi(ct)}{\Phi(t)} = 1\] where $c >0$.
\end{prop}
\begin{proof}
	If we take $t$ to be large enough, and by use of~\ref{assump2}, we may make the following deduction 
\begin{align*}
	\Phi(t)\log t &< \Phi(\varphi(t))\log \varphi(t), \\
	\Phi(t)\log t &< \Phi(\varphi(t))\left(1 + \Phi(t)\right)\log t, \\
	\frac{1}{1 + \Phi(t)} &< \frac{\Phi(\varphi(t))}{\Phi(t)} \leq 1.
\end{align*}
The final inequality comes from the non-increasing nature of $\Phi$. The first limit follows. The other two limits follow in a very similar manner by use of~\ref{assump2}.
\end{proof}

We now prove an interesting result that shows if we iterate the process of how we defined $\widehat{\Phi}_c$ in terms of $\Phi$, we get alternating behaviour in certain limits.
\begin{prop}
	If \[ \lim_{t \to \infty}\frac{\widehat{\Phi}_{1}(t)}{\Phi(t)}=0\] then, for all $c_1, c_2, c_3>0$, \[ \lim_{t \to \infty}\frac{\widehat{\Phi}_{c_1}(t)}{\Phi(t)}=0 \text{ and }\lim_{t \to \infty}\frac{\mathring{\Phi}_{c_1, c_2}(t)}{\widehat{\Phi}_{c_3}(t)}=\infty. \]
\end{prop}
\begin{proof}
	If \[ \lim_{t \to \infty}\frac{\widehat{\Phi}_{1}(t)}{\Phi(t)}=0\] then, by taking large enough $t$ as per Proposition~\ref{prop:phiinv} we can write,
	
	\begin{align*}
		\frac{\widehat{\Phi}_{1}(t)}{\Phi(t)}&= \frac{\varphi^{-1}(\log t)}{\Phi(t)\log t} > \frac{\left(\log t\right)^{1 - \Phi(\log t)}}{\Phi(t)\log t} = \frac{1}{\Phi(t) \left(\log t\right)^{(\Phi(\log t))}} > 0.
	\end{align*}
	The first part is immediate from this.
	
	We now consider \[ \frac{\mathring{\Phi}_{c_1, c_2}(t)}{\widehat{\Phi}_{c_3}(t)} = \frac{c_1 \widehat{\varphi}^{-1}_{c_2}(\log t)}{\log t} \frac{\log t}{c_3\varphi(\log t)}.\] We now let $M>1$ and use Proposition~\ref{prop:phiinv} to write \[ \frac{c_1}{c_3}\frac{\widehat{\varphi}^{-1}_{c_2}(\log t)}{\varphi^{-1}(\log t)} > \frac{c_1}{c_3}\frac{\left(\log t\right)^{1 - \widehat{\Phi}_{c_2}(\log t)}}{\left( \log t\right)^{1 - \Phi(\log t)/M}} = \frac{c_1}{c_3}\left(\log t\right)^{\Phi(\log t)/M - \widehat{\Phi}_{c_2}(\log t )}. \]
	
	By taking a logarithm we find \[ \log \frac{\mathring{\Phi}_{c_1, c_2}(t)}{\widehat{\Phi}_{c_3}(t)} = \log ( c_1/c_3) + \left(\frac{1}{M} - \frac{\widehat{\Phi}_{c_2}(\log t)}{\Phi(\log t)}\right)\Phi(\log t) \log \log t. \]  By the hypotheses of the proposition and the Standing Assumption~\ref{standingassumption}, this final expression tends to $\infty$ as $t \to \infty$ which completes the proof.
\end{proof}
This result will be of much more significance in Section~\ref{sec:phivspsi}.

The following corresponding result is also true and the proof follows in the exact same manner.

\begin{prop}
	If \[ \lim_{t \to \infty}\frac{\widehat{\Phi}_{1}(t)}{\Phi(t)}=\infty\] then, for all $c_1, c_2, c_3>0$, \[ \lim_{t \to \infty}\frac{\widehat{\Phi}_{c_1}(t)}{\Phi(t)}=\infty\text{ and }\lim_{t \to \infty}\frac{\mathring{\Phi}_{c_1, c_2}(t)}{\widehat{\Phi}_{c_3}(t)}=0. \]
\end{prop}

We prove the following result which is used when we estimate the order of growth of our model function in Section~\ref{sec:phivspsi}.
\begin{lemma}\label{phishift} For every $\alpha>0$ and $M>1$ there exists a $t^* > 1$ such that, for all $t>t^*$, \[ \varphi(t+\alpha) \leq M\varphi(t).\]
\end{lemma}
\begin{proof} Given $\alpha > 0$ and $M>1$ choose any $t^*>1$ such that \[ t^* > \frac{\alpha}{M^{1/(1 + \Phi(1))} - 1}.\] After rearranging and taking a logarithm, we get \[ \log M >(1 + \Phi(1))\log\left( 1 + \frac{\alpha}{t^*}\right).\] For $t > t^*$,
\begin{align*}0 &< \left( \Phi(t) - \Phi(t + \alpha) \right) \log(t+ \alpha)\\ &= \left( 1 + \Phi(t) \right)\log(t+ \alpha) - \left(1 + \Phi(t+ \alpha ) \right) \log(t + \alpha) \\ &=\left( 1 + \Phi(t) \right)\left( \log\left(1+\frac{\alpha}{t}\right) + \log t \right) - \left(1 + \Phi(t+ \alpha) \right) \log(t + \alpha)\\ &< \left( 1 + \Phi(1) \right)\log\left(1+\frac{\alpha}{t} \right) + (1 + \Phi(t) ) \log t  - \left(1 + \Phi(t+ \alpha) \right) \log(t + \alpha)\\ &< \log M + \left(1 + \Phi(t)\right) \log t - \left(1 + \Phi(t+ \alpha) \right) \log(t + \alpha) \\&= \log (M\varphi(t)) - \log(\varphi(t+ \alpha)) .\end{align*} Therefore \[ \log(M\varphi(t)) > \log (\varphi(t+\alpha)) \] for all $t > t^*$ and the result follows.
\end{proof}

We next formally introduce and define the recurrence relations between $(r_j)_{j=0}^\infty$ and $(R_j)_{j=0}^\infty$ that lead to the counterexample functions with the desired rate of growth that we are after.

\begin{defn}\label{rR}
We define the following recursive relation between $r_j$ and $R_j$ in terms of $\varphi$. Let 
\begin{equation*}
\log r_{j+1} \defeq \varphi(R_j) - 1 \text{ and } \log R_{j+1} \defeq \varphi(R_j) + 9R_j. 
\end{equation*}
\end{defn}
The reason why $-1$ has been introduced to the definition of $\log r_{j+1}$, previously absent in the initial suggestions in ~\eqref{phiandpsi} and~\eqref{phipsitaketwo}, is to ensure that $r_{j+1} < \exp(\varphi(R_j))$, which will be used when trying to satisfy the conditions of Definition~\ref{cxpleconds} in calculations.

From Section~\ref{recrels}, it might be expected that, in a similar manner to~\cite[Proposition~8.1]{rrrs}, this might provide us with a counterexample tract datum $\xi \in \Xi_\mathcal{C}$ such that the order of growth of the  corresponding $F$ can be quickly determined as well. This is unfortunately more than we can hope for with the growth estimates available without some further work to ensure that we continue to satisfy the conditions of Definition~\ref{cxpleconds}. \\

We have not made any reference to the actual size of the $\eps_j$ in the paper so far but this is exactly where they become relevant to our situation. By recalling the inequality in Theorem~\ref{growth}, it can be seen that by choosing a suitably small enough $\eps_j$ we can ensure that $\abs{F(z)} > R_{j+1}$ for all $z \in W_j$ (Figure~\ref{overshoottract}). Similarly if we let $1/\eps_j$ remain uniformly bounded away from zero for all $j\geq 0$ in a way such that $\abs{F(z)} \leq r_{j+1}$ for all $z \in W_j$ and $j\geq 0$. This situation is illustrated in Figure~\ref{undershoottract}. We give a formal statement and proof in Theorem~\ref{range}.

We therefore endeavour to achieve the following for all $j\geq 0$:
\begin{equation}\label{target}
\begin{split}
\re F^{-1}(\exp(\varphi(R_j))) &= R_j - 2 - 3\nu_0 =\tau_j \text{ and } \\ \im F^{-1}(\exp(\varphi(R_j))) &\in (-\pi/3, \pi/3).
\end{split}
\end{equation}
We provide visual interpretation of this in Figure~\ref{shootthemonkey}.

\begin{figure}[htbp]
\centering
\def\svgwidth{1\textwidth}
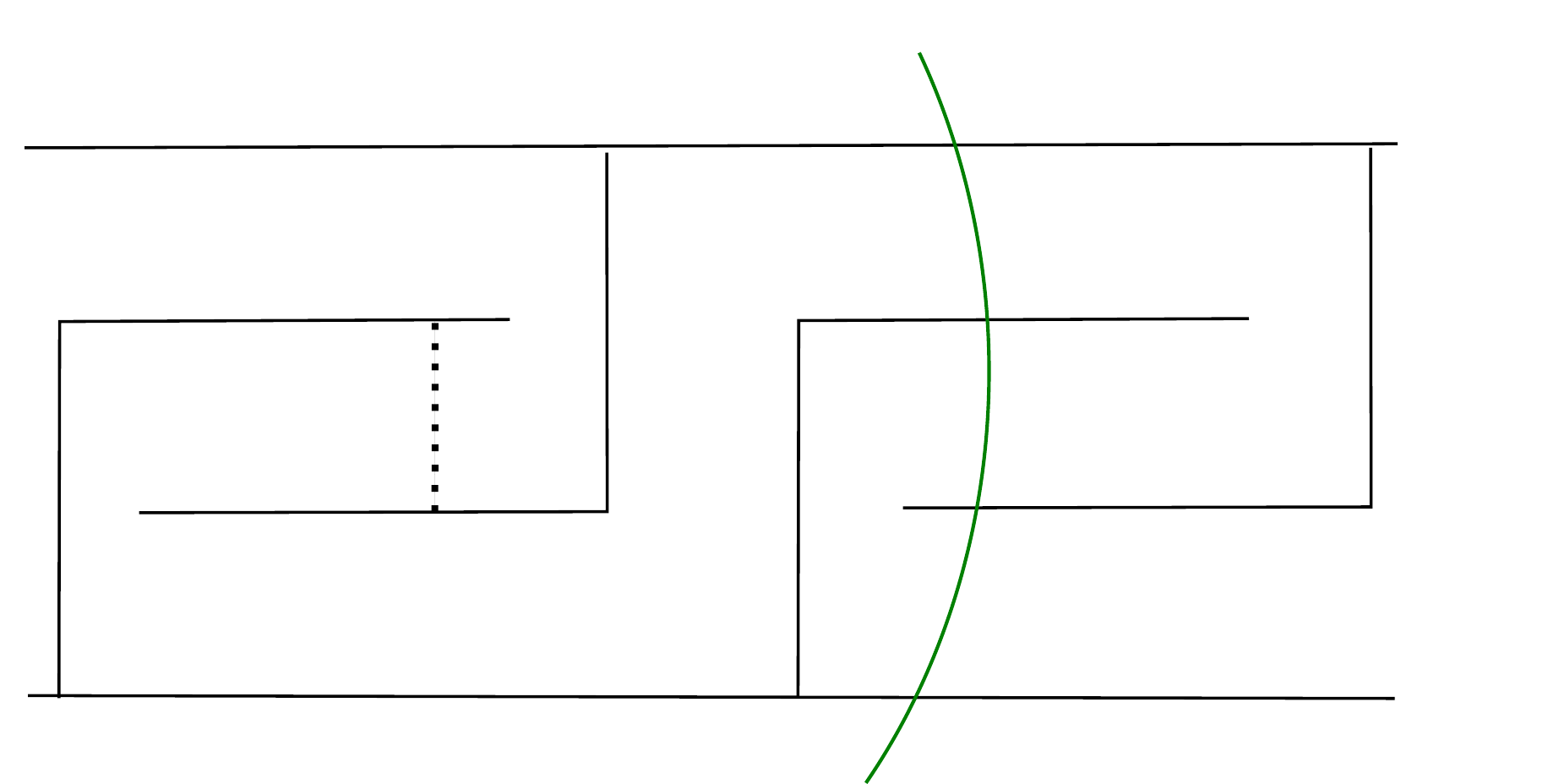
\caption{Visualisation of~\eqref{target}, described by Definition~\ref{philock} and achieved in Theorem~\ref{thm:preshoot}.}
\label{shootthemonkey}
\end{figure}

Should this be satisfied, it implies the vertical geodesic $\Gamma_{\exp(\varphi(R_j))}$ (see Definition~\ref{vertgeo}) is contained in $W_j^+$. This allows us to fulfil condition $(b)$  in Definition~\ref{cxpleconds}, by taking $C_j = \Gamma_{\exp(\varphi(R_j))}$ and $\rho_{j+1} = \exp(\varphi(R_j))$.

It can be found from a short calculation that the recurrence relation in Definition~\ref{rR} implies the following relation:
 \begin{equation}\label{r--R}
 R_j = r_j \exp(9 \varphi^{-1}(\log(r_j) + 1 ) + 1).
 \end{equation}
Recall
\[ e^{\alpha}=t^{\alpha/\log t}.\] We can then say, after a short calculation, that~\eqref{r--R} has the equivalent formulation 
\begin{equation}\label{logr--R}
\log R_j = \left( 1 + \frac{9\varphi^{-1}(\log (r_j) + 1)}{\log(r_j) + 1}\right)\left(\log (r_j) + 1\right),
\end{equation}
noting that $\varphi$ is eventually invertible by Proposition~\ref{phiprops}. \\

\noindent In light of the discussion in Section~\ref{recrels}, we make the following definition for $\psi$.

\begin{defn}\label{psi} Let $t_0 >0$ be such that $\varphi$ is strictly increasing on $[t_0, \infty)$. Let $\Psi \colon [\exp(t_0), \infty) \to (0,\infty)$ be defined as \begin{equation}\label{eqn:Psi} \Psi(t) \defeq \widehat{\Phi}_{10}(\log t) = \frac{10\varphi^{-1}(\log t)}{\log t}\end{equation} and \begin{equation}\label{eqn:psi} \psi(t) \defeq t^{1 + \Psi(t)}.\end{equation}
\end{defn}
The reason for the discrepancies between the definition taken and what was suggested by~\eqref{logr--R} ($10$ instead of $9$ in the numerator, $\log t$ instead of $\log (t) +1$ throughout) is simply one of preference and makes no material difference to the subsequent work. We return to consider how $\varphi$ and $\psi$ interact with each other when determining the order of growth of $F$ in Section~\ref{sec:phivspsi}.

\section{Gate selection: conditions for counterexamples}\label{sec:gateselection}
We will now turn our focus to proving that there is a sequence of gate openings, $(\eps_j)_{j=0}^\infty$, such that this new counterexample condition can indeed be satisfied. 
\begin{defn}\label{philock}
	Let $\xi\in\Xi$, and let $T=T^{\xi}$ and $F=F^{\xi}$ be the corresponding elements of
$\mathcal{T}^{\Xi}$ and $\mathcal{H}^{\Xi}$. We say the tract datum $\xi$ satisfies the \emph{counterexample gate condition with respect to $\varphi$} if the sequences
$(r_j)_{j=0}^{\infty}$ and $(R_j)_{j=0}^{\infty}$ satisfy the recurrence relations in Definition~\ref{rR}, and furthermore 
 \begin{align*}
    \re F^{-1}(\exp(\varphi(R_j))) &= R_j - 2 - 3\nu_0 \text{ and } \\ 
    \im F^{-1}(\exp(\varphi(R_j))) &\in (-\pi/3, \pi/3) 
 \end{align*}
  for all $j\geq 0$.
\end{defn}

The reason for choosing $R_j$ as in Definition~\ref{rR} and for the definition of the counterexample gate condition is that for a tract datum satisfying these, it will indeed give rise to a counterexample. We are able to deduce the existence of a tract datum satisfying these properties in Theorem~\ref{epseq}.
\begin{theorem}\label{thm:preshoot} Any tract datum that satisfies the counterexample gate condition with respect to $\varphi$ is a counterexample tract datum in the
 sense of Definition~\ref{cxpleconds}. 
\end{theorem}
\begin{proof}
Let $\xi$ be a tract datum satisfying the counterexample gate condition with respect to $\varphi$. 
We define the sequences  $(w_j)_{j=0}^\infty$ and  $(\dot{w}_j)_{j=0}^\infty$ by \begin{equation}  w_j \defeq F^{-1}(\exp(\varphi(R_j))) \text{ and } \dot{w}_j \defeq w_j - 2\pi i/3 \end{equation} and 
claim that the conditions of Definition~\ref{cxpleconds} are satisfied when we take \begin{equation} \rho_{j+1} \defeq \lvert F(w_j)\rvert \text{ and } \dot{\rho}_{j+1} \defeq \lvert F(\dot{w}_j)\rvert.\end{equation}
We recall that $W_j$ is the wiggling section with real parts between $r_j$ and $R_j$ and imaginary parts between $-\pi$ and $\pi/3$. $W_j^+$ forms the upper half and $W_j^-$ forms the lower half. See Definition~\ref{defn:W} and Figure~\ref{wj}.
\begin{claim}[Claim 1] For all $j\geq 0$, the vertical geodesics $C_{j} \defeq \Gamma_{\rho_{j+1}}$ and $\dot{C}_j\defeq \Gamma_{\dot{\rho}_{j+1}}$ are contained in
			     $W_j^+$ and $W_j^-$, respectively, and both have real parts strictly between $R_j - 2 - 4 \nu_0 $ and $R_j - 2 - 2\nu_0 $.
\end{claim}
\begin{subproof} This follows immediately from the construction, the hypotheses and the use of Proposition~\ref{nu0} and Definition~\ref{philock}.
\end{subproof}
In the next three claims we show that
\[ r_{j+1} + R_j + 1 + 2\pi < \lvert F(w_j) \rvert < \frac{\lvert F(\dot{w}_j)\rvert}{2} < \lvert F(\dot{w}_j)\rvert < R_{j+1} - 2 - 4 \nu_0\] for all $ j \geq 0$.

\begin{claim}[Claim 2]  $ r_{j+1} + R_j + 1 + 2\pi < \lvert F(w_j) \rvert $ for all $j\geq 0$.
\end{claim}
\begin{subproof} Recall $R_j + 60 < r_{j+1}$. Therefore \[r_{j+1} + R_j +1 + 2\pi < 2r_{j+1}.\] Then, by Definition~\ref{rR}, \begin{align*}2r_{j+1} < e r_{j+1} =\exp(\log(r_{j+1}) + 1) = \exp (\varphi(R_j))=  \lvert F(w_j) \rvert.    \end{align*} The claim follows from this.
\end{subproof}

\begin{claim}[Claim 3] $ \lvert F(w_j) \rvert < \lvert F(\dot{w}_j)\rvert/2$ for all $j \geq 0$.
\end{claim}
\begin{subproof}Consider any curve that connects $C_j$ and $\dot{C}_j$ and remains in $W_j$. Then the distance of any point on the curve from the boundary of the tract $\partial T$ is at most $\pi/3$ and the Euclidean length of such a curve is greater than $2(R_j - r_j - 4\nu_0 - 3).$ Using the standard estimate~\eqref{eqn:standardestimate} we find \[ \log \lvert F(\dot{w}_j) \rvert - \log \lvert F(w_j) \rvert \geq \frac{3}{\pi}(R_j - r_j - 4\nu_0 - 3) > \frac{6}{\pi} > \log 2.\]The final inequality is deduced by Definition~\ref{defn:rRspacing}. The claim follows immediately.
\end{subproof}

\begin{claim}[Claim 4] $ \lvert F(\dot{w}_j)\rvert < R_{j+1} - 2 - 4 \nu_0 $ for all $j \geq 0$.
\end{claim}
\begin{subproof}Considering the distance between $C_j$ and $\dot{C}_j$ once more, we provide an upper bound in the following way. Consider the path $\alpha$ from the proof of Theorem~\ref{growth} again, specifically the segment travelling between $C_j$ and $\dot{C}_j$. The distance from the boundary is at least $1/2$ and the Euclidean length can be bounded above by $2(R_j - r_j) + 2\pi/3$. Making use of the standard estimate once more we can write \[ \log \lvert F(\dot{w}_j) \rvert - \log \lvert F(w_j) \rvert < 8(R_j - r_j) + \frac{8\pi}{3} < 8R_j. \] The final inequality follows since $r_0 > 6$. Hence
\begin{align*} \log R_{j+1} - &\log \lvert F(\dot{w}_j) \rvert \\&= \log R_{j+1} - \log \lvert F(w_j) \rvert  - (\log \lvert F(\dot{w}_j) \rvert - \log  \lvert F(w_j) \rvert) \\ &= 9R_j - (\log  \lvert F(\dot{w}_j) \rvert - \log \lvert F(w_j) \rvert ) > R_j. \end{align*} This proves the claim.
\end{subproof}

The proof of the theorem is now complete.
\end{proof}

\section{Gate selection: A shooting problem}\label{sec:shoot}

In this section we are now going to show that we can in fact deduce the existence of a tract datum that satisfies the counterexample gate condition.
 We continue to fix a function $\Phi$ as in Standing Assumption~\ref{standingassumption}. With the 
 sequences $(r_j)_{j=1}^{\infty}$ and $(R_j)_{j=0}^{\infty}$ determined by $\Phi$ and $r_0$ via Definition~\ref{rR}, the only components of a tract datum $\xi \in \Xi$ that can now vary are $r_0$ and the sequence $(\eps_j)_{j=0}^\infty$. Theorem~\ref{thm:preshoot} has shown us that if we are able to find suitable values for these that further satisfy the counterexample gate condition with respect to $\varphi$, then $\xi$ is in fact a counterexample tract datum and we write $\xi = \xi(\Phi, r_0, (\eps_j)_{j=0}^\infty)$ for a tract datum defined in such a way. The task at hand is to now show that it is indeed possible, for sufficiently large $r_0$, to find such a tract datum. This is what we are able to do in Theorem~\ref{epseq}.
	
%	The way we proceed is to show that for each `gate opening', $\eps_j$, there is a range of values, $[a_j, b_j]$, such that if $\eps_j = a_j$, $\lvert F(z) \rvert > R_{j+1}$ for all $z \in W_j$ and if $\eps_j = b_j$, $\lvert F(z) \rvert < r_{j+1}$ for all $z \in W_j$. This shows us that whatever value $\varepsilon_j$ takes in $[a_j,b_j]$, whenever $z \in W_j$, $\lvert F(z)\rvert $ either `overshoots' or `undershoots' its intended target of $\exp(\varphi(R_j))$ according to this range of possible $\eps_j$ values. We make use of the fact that our vertical geodesics have bounded diameter to ensure we can deduce the existence of some $\zeta_j \in W_j$ such that $F(\zeta_j) = \exp(\varphi(R_j)).$ We then invoke a corollary of the Poincar\'e--Miranda Theorem, Corollary~\ref{surj}. One can think of the Poincar\'e--Miranda Theorem as a multi-dimensional version of the Intermediate Value Theorem to show that we can achieve this for all target values $(\exp(\varphi(R_j)))_{j=1}^\infty$ simultaneously. Once this is achieved it is simply a matter of verifying that we satisfy the conditions laid out in Definition~\ref{cxpleconds}.

\begin{theorem}\label{epseq} Let $\Phi$ and $\varphi$ be as in Standing Assumption~\ref{standingassumption}. Then 
there exists a $\rho_0>0$ such that if $r_0>\rho_0$ then there exists a sequence $(\varepsilon_j)_{j=0}^\infty$ such that the resulting tract datum $\xi=\xi(\Phi,r_0, (\eps_j)_{j=0}^\infty)$ satisfying Definition~\ref{rR}, also satisfies the counterexample gate condition with respect to $\varphi$. 
In particular, $\xi\in\Xi_\mathcal{C}$. \end{theorem}
Before proving this we will first prove some more minor results and define the range of values that $\varepsilon_j$ can take for each $j \geq 0$.

\begin{defn}\label{ajbj}
Let $(r_j)_{j=0}^\infty$ and $(R_j)_{j=0}^\infty$ be a pair of sequences defined as in Definition~\ref{rR}. Define the following sequences $(a_j)_{j=0}^\infty$ and $(b_j)_{j=0}^\infty$ by the following relations
\begin{equation*} 
 a_j \defeq \frac{1}{\exp(2C\varphi(R_j))} \quad \text{ and } \quad b_j \defeq \frac{1}{\exp(\varphi(R_j)/2C))}. 
 \end{equation*}
$C$ is the absolute constant from Theorem \ref{growth}.
\end{defn}
\begin{lemma}\label{sumlem} 
For any $M>0$, there exists $\rho(M) > 0$ with the following property. Let $\xi$ be a tract datum such that 
  $r_0>\rho(M)$, Definition~\ref{rR} holds and $\varepsilon_j\in [a_j,b_j]$ for all $j$. Then 
\[ \sum_{k=0}^{j - 1} M \varphi(R_k) < R_j \enspace \text{ for all } \enspace j \geq 0 .\]
\end{lemma}
\begin{proof}
From the deductions following \eqref{phi} we know that for any $A>1$ there is some $\rho_1 = \rho_1(A) > 0$ such that $t <  \varphi(t) < t^A$ for all $t > \rho_1(A)$.

From the way we defined $(R_j)_{j=0}^\infty$ in Definition~\ref{rR} we know that if $R_0 > \rho_1$ then $R_k < \varphi(R_k) < \log R_j$ for all $0 \leq k \leq j-1.$ Using Definition~\ref{rR} we can write \[ \frac{R_{k+1}}{R_k}> \frac{\exp(\varphi(R_k) + 9R_k)}{R_k} > \exp(8R_k).\]  We can also write \[ \frac{\log R_{k+1}}{R_k} = \frac{\varphi(R_k)+9R_k}{R_k} > \frac{\varphi(R_k)}{R_k} \to \infty \] as $k \to \infty$ by our assumption. Therefore there is some $\rho_2 = \rho_2(A, M)> 0$ such that; if $R_j > \rho_2$ then \[ \log R_j > R_{j-1} > R_k M^{1/A}\] for all $0 \leq k \leq j-1$. There is also a $\rho_3= \rho_3(A) > 0$ such that if $R_j > \rho_3$ then $R_j > (\log R_j)^{A+1}.$ So 
\[\sum_{k=0}^{j-1} M\varphi(R_k) < \sum_{k=0}^{j-1} MR_k^A < j(\log R_j)^A < (\log R_j)^{A+1} < R_j.\]

If we set $\rho_4(A,M) = \max \lbrace \rho_1, \rho_2, \rho_3 \rbrace$ then the previous inequalities hold whenever $R_0 > \rho_4$. The result follows.
\end{proof}
We recall that $W_j$ from Definition~\ref{defn:W} and Figure~\ref{wj} as the wiggling section between $r_j$ and $R_j$ with the lower two-thirds defined by , $W_j^+$ and $W_j^-$. We are now able to show that on the range of values $[a_j, b_j]$, $F(z)$ has the potential to ``undershoot'' or ``overshoot'' $W_{j+1}$ when we take $z \in W_j$. We refer back to Figure~\ref{undershoottract} and Figure~\ref{overshoottract}.
\begin{prop}\label{range} There exists $\rho_0 > 0$ with the following property. Let $\xi$ be a tract datum such that 
  $r_0>\rho_0$, Definition~\ref{rR} holds and $\varepsilon_j\in [a_j,b_j]$ for all $j$. Then, for $j \geq 0$,
\begin{itemize} 
\item if $\varepsilon_j = a_j$ then $\abs{F(z)} > R_{j+1}$ for $z \in W_j$,
\item if $\varepsilon_j = b_j$ then $\abs{F(z)} < r_{j+1}$ for $z \in W_j$.
\end{itemize}
\end{prop}
\begin{proof}
First let $\varepsilon_j = a_j$. We know that \[ \log \left( \frac{1}{a_j} \right) = 2C\varphi(R_j).\] Given $z \in W_j$, according to our growth estimate in Theorem~\ref{growth} and definition of $a_j$ in Definition~\ref{ajbj}, 

\begin{align*} \log \abs{F(z)} &\geq \frac{1}{C} \left( R_j + \sum_{k=0}^j \log \left( \frac{1}{\varepsilon_k} \right)\right) \\&=\frac{1}{C} \left( R_j + \sum_{k=0}^{j-1} \log \left( \frac{1}{\varepsilon_k} \right) + 2C\varphi(R_j) \right) \\ &\geq 2\varphi(R_j)    > \varphi(R_j) + 9R_j = R_{j+1}
\end{align*}
The last inequality holds if $r_0$ is taken large enough.

Now let $ \varepsilon_j = b_j$. We know that \[ \log \left(\frac{1}{b_j}\right)= \frac{\varphi(R_j)}{2C}.\]  Since $1/\eps_k \leq 1/a_k$ for all $k \leq j-1$, we can write \[ \log\left( \frac{1}{\eps_k}\right) \leq \log \left( \frac{1}{\eps_k}\right)  =  2C\varphi(R_k)\] for each $k \leq j-1$. We first start with the inequality from Theorem~\ref{growth} again.

\begin{align*} \log \abs{F(z)} &\leq  C\left( R_j + \sum_{k=0}^j \log \left( \frac{1}{\varepsilon_k} \right)\right) \\&=C\left( R_j + \sum_{k=0}^{j-1} \log \left( \frac{1}{\varepsilon_k} \right) + \frac{\varphi(R_j)}{2C} \right) \\ &\leq C\left( R_j + \sum_{k=0}^{j-1} 2C\varphi(R_k) + \frac{\varphi(R_j)}{2C} \right) \\ &= C R_j + \sum_{k=0}^{j-1} 2C^2\varphi(R_k)  + \frac{\varphi(R_j)}{2} \\ & \leq (C+1)R_j + \frac{\varphi(R_j)}{2} \\ & <  \frac{2}{3}\varphi(R_j) < \varphi(R_j) <\log r_{j+1} < r_{j+1}.
\end{align*}
An application of Lemma~\ref{sumlem} was used to deduce the penultimate line (where we let $M=2C^2$) and the inequalities in the final line follow for suitably large enough $r_0$ that also satisfies Theorem~\ref{growth}.
\end{proof}
We make the following update to our standing assumption.
\begin{standingassumption}\label{updatedassumption}
	We continue to assume all instances of $\Phi$ and $\varphi$ satisfy Standing Assumption~\ref{standingassumption}. From now on, we will also fix a value $\rho_0>0$ and an $r_0 > \rho_0$ such that Lemma~\ref{sumlem} and Proposition~\ref{range} are satisfied. 
\end{standingassumption} 
 We will prove Theorem~\ref{epseq} for these values of $r_0$ and $\rho_0$. Observe that this means that the values of $r_j$ and $R_j$ are fixed for the remainder of the paper by Definition~\ref{rR}. From this point onwards the sequence $(\eps_j)_{j=0}^\infty$ fully determines our tracts. We therefore introduce the following subclass of tract data.
\begin{defn}\label{XIPHI}
Let $\, \Xi^\Phi$ be the subclass of $\, \Xi$ (Definition~\ref{defn:rRspacing}) which contains precisely the tract data where we fix an $r_0 > \rho_0$ in order to satisfy Proposition~\ref{range} and the sequences $(r_j)_{j=0}^\infty$ and $(R_j)_{j=0}^\infty$ are determined by Definition~\ref{rR}. We denote a tract datum in $\Xi^\Phi$ by $\xi^\Phi =\xi^\Phi((\eps_j)_{j=0}^\infty)$, uniquely determined by $(\eps_j)_{j=0}^\infty$. $\mathcal{T}^{\Xi^{\Phi}}$ and $\mathcal{H}^{\Xi^\Phi}$ are the associated subclasses of tracts and conformal isomorphisms (Definition ~\ref{defn:rRspacing}).
\end{defn}

\subsection{Defining signed distance}

We will first define the ``open tract'' in which every gate takes on the largest value it can, that is $(\eps_j)_{j=0}^\infty = (b_j)_{j=0}^\infty$. For our sequences $(r_j)_{j=0}^\infty$ and $(R_j)_{j=0}^\infty$, which all $\xi^\Phi$ have in common, we define the following tract

\begin{defn}[Open tract]\label{defn:openT} Let $\xi_0^\Phi \defeq \xi^\Phi((b_j)_{j=0}^\infty) \in \Xi^\Phi$. We say that the corresponding tract $T_0^\Phi \defeq T^{\xi_0^\Phi}$ is the \emph{open tract} of $\Xi^\Phi$.
\end{defn}
 We define the following subset of $T_0^\Phi$ for each $j \geq 0$. \begin{align*} \mathcal{Y}_j \defeq \lbrace z \in T_0^\Phi \colon R_j - 2 - 4\nu_0 \leq \re z &  \leq R_j - 2 - 2\nu_0 \\&\text{ and } - \pi/3 < \im z <  \pi/3 \rbrace. \end{align*}
Observe that, for each $j$, $T_0^\Phi \setminus \mathcal{Y}_j$ comprises a bounded component, which we denote by $\mathcal{X}_j$, and an unbounded component denoted by $\mathcal{Z}_j$.

\begin{figure}[htbp]
\centering
\def\svgwidth{1\textwidth}
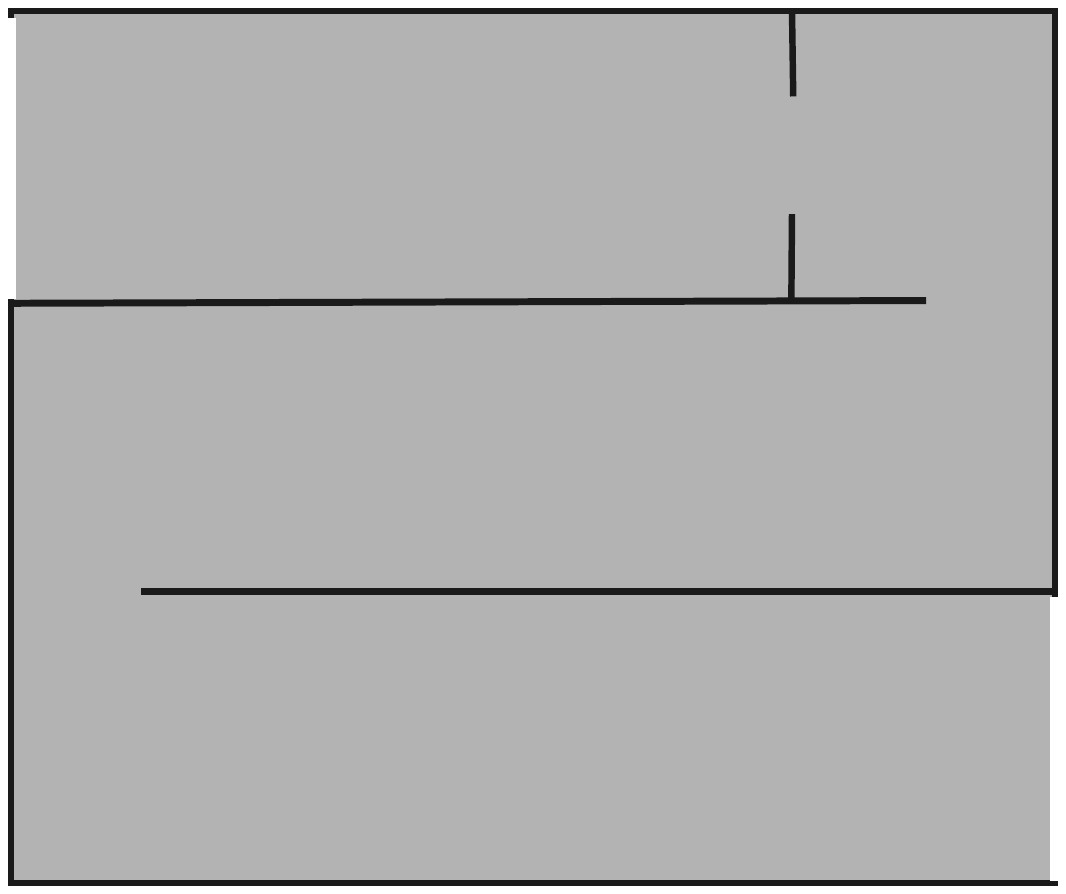
\caption{Regions $\mathcal{X}_j$, $\mathcal{Y}_j$, and $\mathcal{Z}_j$.}
\label{XYZ}
\end{figure}

We now define a notion of `signed distance' on $T$. For $z \in T$ and $j \geq 0$ let $\delta \colon T_0^\Phi \times \mathbb{N} \to [-1,1]$ be defined by

\[ \delta(z,j) \defeq 
\begin{dcases*} 
-1 & \text{ if   $z \in \mathcal{X}_j$,} \\
\frac{R_j - 2 - 3\nu_0 - \re z}{\nu_0}   &\text{ if $z \in \mathcal{Y}_j$},  \\
1 &\text{ if $z \in \mathcal{Z}_j$.}
\end{dcases*} \]

We further define \begin{equation}\label{deltaj}
\delta_j(\xi) \defeq \delta( \iota_{\Phi}((F^{\xi^\Phi})^{-1}(\exp(\varphi(R_j)))), j).
\end{equation}
Where $\iota_{\Phi}\colon T^{\xi^\Phi} \hookrightarrow T_0^\Phi$ is the inclusion map for any $T^{\xi^\Phi} \in \mathcal{T}^{\Xi^\Phi}$ being included into the open tract $T_0^\Phi$ and $F^{\xi^\Phi} \colon T^{\xi^\Phi} \to \HH$ is the corresponding element of $\mathcal{H}^{\Xi^\Phi}$.
In the following proposition and lemma we will be referring to the notion of Carath\'{e}odory kernel convergence and related results as discussed in Section~\ref{sec:kernel}.
\begin{prop}\label{Finvconv} If a sequence of tracts $T_n \in \mathcal{T}^{\Xi^\Phi}$ converges to $T$ with respect to $5$ in the sense of Carath\'eodory kernel convergence then the corresponding sequence of conformal isomorphisms $F^{-1}_n \colon \HH \to T_n$ converge locally uniformly to $F^{-1} \colon \HH \to T$.
\end{prop}
\begin{proof}This is given by \cite[Proposition 8.2]{tania}. We are allowed to use this since our tracts $\mathcal{T}^{\Xi^\Phi}$ are associated with conformal isomorphisms in the class $\mathcal{H}^{\Xi^\Phi} \subset \mathcal{H}_{\nu_0}$ (Proposition~\ref{nu0}) in which our vertical geodesics (Definition~\ref{vertgeo}) are bounded by a universal constant $\nu_0$.
\end{proof} 

Given that the tracts in $\mathcal{T}^{\Xi^\Phi}$ all share the same turning points, that is, have $(r_j)_{j=0}^\infty$ and $(R_j)_{j=0}^\infty$ in common, they can only differ in the places where the epsilon gates are. Since $\eps_j \in [a_j, b_j]$ for all $j\geq 0$ we are therefore only concerned with the following collection of line segments \[ \bigcup_{\ell=0}^\infty \lbrace \tau_j + \pi(2 + t)i/3 \colon t \in [a_j, b_j] \rbrace  
    \cup\lbrace \tau_j + \pi(2 - t)i/3 \colon t \in [a_j, b_j] \rbrace\] which is a collection of pairs of line segments of added length $\pi(b_j - a_j)/3$ at each gate. We make the following definition which will be used in the next lemma. \begin{defn} For $j\geq 0$ let $\eta_j(s)\colon [a_j,b_j] \to T_0^{\xi^\Phi}$ be defined in the following way \[ \eta_j(s) \defeq \lbrace \tau_j + \pi(2 + t)i/3 \colon t \in [s, b_j] \rbrace \cup\lbrace \tau_j + \pi(2 - t)i/3 \colon t \in [s, b_j] \rbrace. \] \end{defn} Note that $\eta_j$ has the following monotonicity property, if $y<x$ then $\eta_j(x) \subset \eta_j(y)$. We illustrate the set as defined in Figure~\ref{etasegment}.
\begin{figure}[htbp]
\centering
\def\svgwidth{1\textwidth}
%% Creator: Inkscape 1.3 (0e150ed6c4, 2023-07-21), www.inkscape.org
%% Accompanies image file 'etasegment.pdf' (pdf, eps, ps)
\begingroup%
  \makeatletter%
  \providecommand\color[2][]{%
    \errmessage{(Inkscape) Color is used for the text in Inkscape, but the package 'color.sty' is not loaded}%
    \renewcommand\color[2][]{}%
  }%
  \providecommand\transparent[1]{%
    \errmessage{(Inkscape) Transparency is used (non-zero) for the text in Inkscape, but the package 'transparent.sty' is not loaded}%
    \renewcommand\transparent[1]{}%
  }%
  \providecommand\rotatebox[2]{#2}%
  \newcommand*\fsize{\dimexpr\f@size pt\relax}%
  \newcommand*\lineheight[1]{\fontsize{\fsize}{#1\fsize}\selectfont}%
  \ifx\svgwidth\undefined%
    \setlength{\unitlength}{829.32391309bp}%
    \ifx\svgscale\undefined%
      \relax%
    \else%
      \setlength{\unitlength}{\unitlength * \real{\svgscale}}%
    \fi%
  \else%
    \setlength{\unitlength}{\svgwidth}%
  \fi%
  \global\let\svgwidth\undefined%
  \global\let\svgscale\undefined%
  \makeatother%
  \begin{picture}(1,0.24514209)%
    \lineheight{1}%
    \setlength\tabcolsep{0pt}%
    \put(0,0){\includegraphics[width=\unitlength,page=1]{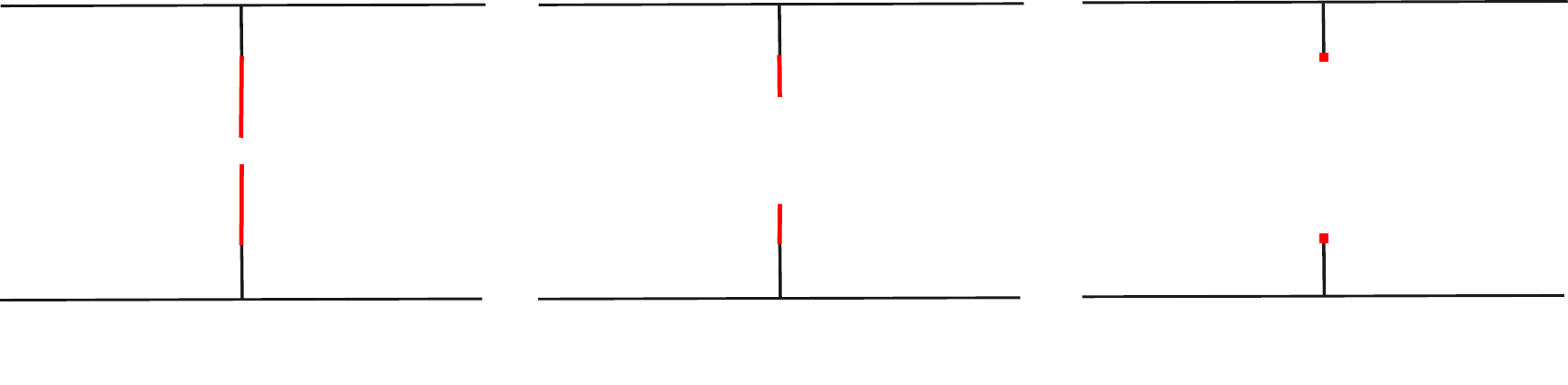}}%
    \put(0.09735436,0.00489275){\color[rgb]{0,0,0}\makebox(0,0)[lt]{\lineheight{1.25}\smash{\begin{tabular}[t]{l}\large $\eta_j(a_j)$\end{tabular}}}}%
    \put(0.35506902,0.00630485){\color[rgb]{0,0,0}\makebox(0,0)[lt]{\lineheight{1.25}\smash{\begin{tabular}[t]{l}\large $\eta_j(s), a_j < s < b_j$\end{tabular}}}}%
    \put(0.78859454,0.00736396){\color[rgb]{0,0,0}\makebox(0,0)[lt]{\lineheight{1.25}\smash{\begin{tabular}[t]{l}\large $\eta_j(b_j)$\end{tabular}}}}%
  \end{picture}%
\endgroup%

\caption{Line segments (in red) representing $\eta_j$ for three different parameters.}
\label{etasegment}
\end{figure}
Now we can solve our ``shooting problem'' and thereby complete the proof of Theorem~\ref{epseq}. We use an approach similar to \cite[Lemma 7.3 \& Theorem 7.4]{lasse13}. In the following lemma will make use of the notation and results in Section~\ref{sec:pmthm}. We assume $0 \in \mathbb{N}$ and we will take $\prod_{j=0}^\infty [a_j,b_j]$ and $\Lambda^\mathbb{N} \defeq [-1,1]^\mathbb{N}$ and their product topologies. These spaces are compact by Tychonoff's Theorem~\cite[~5]{kelley}.

In order to use the results of Section~\ref{sec:pmthm}, we need to have continuous maps $p \colon \Lambda^k \to \Lambda^k$ requiring us to normalise our sequences of $\eps_j \in [a_j,b_j]$ which we do in the following, with a slight abuse of notation.

\begin{defn}\label{epst}
 For $j\geq 0$, define \begin{align*}\eps_j \colon [-1,1] &\to [a_j,b_j];\\ t &\mapsto \frac{1}{2}( a_j(1 - t) + b_j(1 + t)). \end{align*} which gives a homeomorphism between $[-1,1]$ and $[a_j, b_j]$ for each \\$j \geq 0$. If $\underline{t} = (t_j)_{j=0}^\infty \in \Lambda^\mathbb{N}$ we write $(\eps_j(\underline{t}))_{j=0}^\infty = (\eps_0(t_0), \eps_1(t_1),\ldots)$ which is also a homeomorphism with respect to the product topology.
\end{defn} 
\begin{defn}\label{piextend}
 Let $N$ be an integer, $N\geq 1$, and let $\underline{t}^N = (t_j^N)_{j=0}^{N-1} = (t_0^N, t_1^N,\ldots , t_{N-1}^N) \in \Lambda^N$. We define the following map  \begin{align*} \pi^N \colon \Lambda^N &\to \Lambda^\mathbb{N}; \\ \underline{t}^N &\mapsto (t_0^N, t_1^N, \ldots, t_{N-1}^N, -1 , -1,\ldots) \end{align*} and further define \begin{align*} \pi \colon \bigcup_{N=1}^\infty \Lambda^N &\to \Lambda^\mathbb{N}; \\ \underline{t}^k&\mapsto \pi^k(\underline{t}^k)\end{align*} where $k\geq 1$. We will write the following for brevity in the proof of the lemma that follows,
\[(\eps_j^N)_{j=0}^\infty = (\eps_j(\underline{t}^N))_{j=0}^\infty \defeq  (\eps_j(\pi(t^N)))_{j=0}^\infty.\]
\end{defn}

\begin{lemma}\label{shoot} Given every $\Phi$ satisfying Standing Assumption~\ref{standingassumption} and $r_0$ large enough to satisfy Proposition~\ref{range}, there exists a sequence \[(\varepsilon_j)_{j=0}^{\infty} \in \prod_{j=0}^\infty [a_j,b_j] \] such that for the resulting tract datum $\xi^\Phi((\eps_j)_{j=0}^\infty)\in \Xi^\Phi$, $\delta_{j}(\xi)=0$ for all $j\geq 0$. That is, $F^{\xi^\Phi}(R_j - 2 -3\nu_0) = \exp(\varphi(R_j))$ for all $j\geq 0$.
\end{lemma}
\begin{proof} A sequence $(t_j)_{j=0}^\infty \in [-1,1]^\mathbb{N}$ will undergo the following mapping route via maps $h_0$, $h_1$, and $h_2$:
\begin{align*}  (t_j)_{j=0}^\infty &\xmapsto{h_0} (\eps_j(t_j))_{j=0}^\infty = (\eps_j)_{j=0}^\infty  \xmapsto{h_1} T^{\xi^\Phi((\eps_j)_{j=0}^\infty)} = T^{\xi^\Phi} \xmapsto{h_2} \left( \delta_j\left(\xi^\Phi\right)\right)_{j=0}^\infty.
\end{align*}
To use Corollary~\ref{surj} we need to ensure we have continuity at each step in order to define a map \[ h \defeq h_2 \circ h_1 \circ h_0 \colon [-1,1]^\mathbb{N} \to [-1,1]^\mathbb{N}.\] We first use Corollary~\ref{surj} in the case of finite-dimensional cubes but we will justify how we can take a limit of finite-dimensional cases to arrive at the countably infinite case satisfying the statement in our lemma.
\begin{claim}[Claim 1] $\delta_j$ depends continuously on $\xi^\Phi$, that is, on $(\eps_j)_{j=0}^\infty$.
\end{claim}
\begin{subproof}
 The first map\begin{align*}   h_0 \colon [-1,1]^\mathbb{N} &\to \prod_{j=0}^\infty [a_j, b_j]; \\ (t_j)_{j=0}^\infty &\mapsto (\eps_j(t_j))_{j=0}^\infty,\end{align*} is the map in Definition~\ref{epst} and a homeomorphism. 

We are now required to show that \begin{align*}  h_1 \colon \prod_{j=0}^\infty [a_j,b_j] &\to \mathcal{T}^{\Xi^\Phi}; \\ (\eps_j)_{j=0}^\infty &\mapsto T^{\xi^\Phi(\eps_j(t_j))_{j=0}^\infty}\end{align*} is continuous.

Let  \[ \left( (\eps_{j,w})_{j=0}^\infty\right)_{w=0}^\infty = \left( (\eps_{j,0})_{j=0}^\infty, (\eps_{j,1})_{j=0}^\infty, \ldots \right) \] be a sequence in $ \left(\prod_{j=0}^\infty [a_j,b_j] \right)^\mathbb{N}$ converging to $(\eps_j)_{j=0}^\infty$ and let \[ \xi^{\Phi_{w}} \defeq \xi^\Phi( (\eps_{j,w})_{j=0}^\infty)\] be the corresponding tract datum with associated sequence $(\xi^\Phi_{w})_{w=0}^\infty$, converging to \\ $\xi^\Phi \defeq \xi^\Phi( (\eps_{j})_{j=0}^\infty)$.
Set \[ U_w \defeq T^{\xi^\Phi_w} \text{ and } U \defeq T^{\xi^\Phi}.\]
By taking the constant sequence $(u_w)_{w=0}^\infty \defeq (5)_{w=0}^\infty$ we will now show that \[ (U_w, 5)_{w=0}^\infty \to (U,5)\] with respect to Carath\'eodory kernel convergence, Definition~\ref{carattop}.

We automatically have $(u_w)_{w=0}^\infty \to 5$, thus satisfying condition $1$ of Definition~\ref{carattop}.

For a set $K \subset U$ we write  \begin{align*} X(K) \defeq \lbrace j \geq 0 \colon \text{ there exists } z \in K \text{ such that } \re z =\tau_j \text{ and } \\ \pi(2-b_j)/3 \leq \im z \leq \pi(2+b_j)/3 \rbrace \end{align*} If we suppose $K \subset U$ is compact, there will be a finite number of gates that $K$ passes through, that is, $X(K)$ is a finite set for compact $K$. If $X = \emptyset$ then $K \subset U_w$ for all $w  \geq 0$. Suppose $X$ is nonempty and let $p \in X$. The line segments of interest at the $p$-th gate are $\eta_p(a_p).$ and we want to find values of $s$ such that $\eta_p(s)$ has an empty intersection with $K$. Since $K$ is a subset of $U$, \[ K \cap \eta_p(\eps_p) = \emptyset \] for all $p \in X$. Recalling $(\eps_{p,w})_{w=0}^\infty \to \eps_p$ in our notation here.

If \begin{align*}  K \cap \eta_p(a_p) = \emptyset \text{ then }  K \cap \eta_p(\eps_{p,w}) = \emptyset \text{ for all } w \geq 0. \end{align*} Now suppose \[ K \cap \eta_p(a_p) \neq \emptyset .\] Since $K$ is compact there exists \[ s_{p} \defeq \max_{[a_p,b_p]} \lbrace s \colon K \cap \eta_p(s) \neq \emptyset \rbrace \] such that \[K \cap \eta_p(s) = \emptyset \text{ for all } s > s_p\]and it must follow that $\eps_p > s_p$. Let $\sigma_p \defeq \abs{\eps_p - s_p} > 0$. Since $(\eps_{p,w})_{w=0}^\infty \to \eps_p$ there exists $w_p^* \geq 0$ such that for all $w > w_p^*$, $\abs{\eps_{p,w} - \eps_p} < \sigma_p/3$. Therefore, for all $w > w_p^*$, \[K \cap \eta_p(\eps_{p,w}) \subset K \cap \eta_p(\eps_p - \sigma_p/3) \subset K \cap \eta_p(s_p + \sigma_p/3) = \emptyset.\] We can do this for each $p \in X$ so we then take \[w^*(K) \defeq \max_{p \in X(K)}\lbrace w_p^*\rbrace.\] By letting $m= m(K) =w^*(K)$ we satisfy Condition $2$ of Definition~\ref{carattop}.
 
To show the third condition holds, let $V$ be an open connected set containing $5$ and suppose $V \subset U_w$ for infinitely many $w$ Our aim is to show that $V \subset U$. Let $(w_k)_{k=0}^\infty$ be the infinite sequence such that $V \subset U_{w_k}$ and $w_k < w_{k+1}$ for all $k\geq 0$. We make a similar argument to that of the previous condition.

The places where $V$ may fail to be a subset of $U$ are, once again, at the epsilon gates. We therefore consider $X(V)$, using the same notation from the previous part. If $X(V) = \emptyset$ when we know that since $V$ is connected, it is contained in the first part of the tract containing $5$ and does not pass through the first epsilon gate at $\tau_0$. In this case we are done given that this section is the same for all tracts in $\mathcal{T}^{\Xi^\Phi}$. \\We know that $V \cap \eta_p(b_p) = \emptyset$ for all $p$ since every $U_w$ must be contained in $T_0^\Phi$, the open tract from Definition~\ref{defn:openT}. Now suppose $X(V)$ is nonempty and let $p\in X(V).$ If $V \cap \eta_p(a_p) = \emptyset$ we know that $V \cap \eta_p(\eps_{p,w_k}) = \emptyset$ for all $k \geq 0$. Now suppose $V\cap \eta_p(a_p) \neq \emptyset.$ This means the following infimum must exist \[s_p \defeq \inf_{[a_p, b_p]}\lbrace s \in [a_p,b_p] \colon V \cap \eta_p(s) = \emptyset \rbrace.\] We proceed to complete the proof by contradiction by assuming $s_p > \eps_p$, which implies $V \cap \eta_p(\eps_p) \neq \emptyset$. Set $\sigma_p \defeq \abs{s_p - \eps_p}.$

Since subsequences of a convergent sequence are also convergent and share the same limit, $(\eps_{p,w_k})_{k=0}^\infty \to \eps_p$. There then exists $k^* \geq 0 $ such that for all $k \geq k^*$, $\abs{\eps_p - \eps_{p,w_k}} < \sigma_p/2$, meaning $\eps_{p,w_k} < s_p - \sigma_p/2 < s_p$ for all $k \geq k^*$. We know $V \cap \eta_p(s_p - \sigma_p/2) \neq \emptyset$ by our the construction of $s_p$ and we know $V \cap \eta_p(s_p - \sigma_p/2) \subset V \cap \eta_p(\eps_{p, w_k})$ therefore $V \cap \eta_p(\eps_{p, w_k}) \neq \emptyset$ for all $k\geq k^*$ and so there can only be finitely many values of $w_k$ for which $V \subset U_{w_k}$, contradicting our assumption. Therefore, $\eps_j > s_p$ and so $V \cap \eta_p(\eps_j) \subset V \cap \eta_p(s_p - \sigma_p/2) = \emptyset$ which means $V \subset U$. This shows we satisfy the third condition of Definition~\ref{carattop}. This means $(U_w, 5)_{w=0}^\infty \to (U,5)$ in the sense of Carath\'eodory kernel convergence. What we have shown is that for any sequence \begin{align*} \text{ if} \quad ((\eps_{j,w})_{j=0}^\infty)_{w=0}^\infty &\to (\eps_j)_{j=0}^\infty, \\ \text{ then } (h_1(\eps_{j,w})_{j=0}^\infty)_{w=0}^\infty &\to h_1((\eps_j)_{j=0}^\infty). \end{align*} We therefore have continuity for $h_1$.

We are now required to show $h_2$ is continuous. The map $h_2$, can be written as the composition of the following maps maps, $g_0$, $g_1$, and $g_2$.
\begin{align*} T^{\xi^\Phi} \xmapsto{g_0} \left( F^{\xi^\Phi}\right)^{-1} \xmapsto{g_1} \left( \left( F^{\xi^\Phi}\right)^{-1} (\exp(\varphi(R_j)))\right)_{j=0}^\infty \xmapsto{g_2} \left( \delta_j\left(\xi^\Phi\right)\right)_{j=0}^\infty.
\end{align*}

Continuity of $g_0$ is given since our sequence of tracts $(U_w, 5)_{w=0}^\infty$ were shown to converge in the sense of Carath\'eodory kernel convergence to $(U, 5)$ and Proposition~\ref{Finvconv}  tells us that the corresponding sequence of maps $\left((F^{\xi_w^\Phi})^{-1} \right)_{w=0}^\infty$ converges locally uniformly to $\left( F^{\xi^\Phi}\right)^{-1}$.

Continuity of $g_1$ follows immediately from this (continuity in each element of the sequence) and the  continuity of $g_2$ is clear due to the construction of $\delta_j$.

The claim follows.
\end{subproof}

By recalling Proposition~\ref{range} we are able to find a range of values that make $\delta_j$ vary from $-1$ to $1$, we will now show how to apply Corollary~\ref{surj} to achieve our result.

\begin{claim}[Claim 2] For $N\geq 1$ there exists $\underline{t}^N \in \Lambda^N$ such that the corresponding tract datum $\xi^N \defeq\xi^\Phi\left((\eps_j^N)_{j=0}^\infty\right)$ (recalling $(\eps_j^N)_{j=0}^\infty$ from Definition~\ref{piextend}) satisfies 
   $\delta_j(\xi^N)=0$ for $0\leq j < N$. 
   \end{claim}
  \begin{subproof}
 For $N=1$, this is an application of the intermediate value theorem: $\delta_0$ is a continuous function of the single 
 variable $t_0^1$. Since $\delta_0 = -1$
   when $t_0^1=-1$ ($\eps_0(-1) = a_0$) and $\delta_0 = 1$ when $t_0^1=1$ ($\eps_0(1) = b_0$). 
   There must exist some value of $t_0^1\in[-1,1]$, equally, a value $\eps_0(t_0^1) \in [a_0,b_0]$, such that $\delta_0$ takes the value $0$. 

In the case $N>1$, we recall Definition~\ref{ncube} where $\Lambda_k^-$ and $\Lambda_k^+$ refer to the to the $k$-th opposing faces of $\Lambda^N$. Let $\underline{t}^N \defeq  (t_0^N, t_1^N,\ldots,t_{N-1}^N) \in \Lambda^N$. Following Definition~\ref{piextend}, $\pi^N(\underline{t}^N) \defeq (t_0^N, t_1^N,\ldots,t_{N-1}^N, -1, -1, -1, \ldots)$, which leads to the corresponding tract datum $\xi^N \defeq \xi^\Phi((\eps_j^N)_{j=0}^\infty)$, and define
  \[ p^N \colon \Lambda^{N}\to \Lambda^{N}; \quad \underline{t}^N\mapsto  (\delta_0(\xi^N),\dots,\delta_{N-1}(\xi^N)). \]
By Proposition~\ref{range} it follows that if $t_j^N=1$ for all $j<N$ then $\eps_j^N=b_j$ for all $j<N$ and therefore $\delta_j=1$, so $p_N(\Lambda_{j}^+) \subset \Lambda_{j}^+$ for all $0\leq j<N$ and $p_N(\Lambda_{j}^-) \subset \Lambda_{j}^-$  for all $j\leq N$ by similar reasoning. 
    By Corollary~\ref{surj}, $p_N$ is then surjective on $\Lambda^N$. Therefore there exists a $\underline{t}_*^N\in \Lambda^N$ such that $p_N(\underline{t}_*^N) = 0$. By taking the sequence $(\eps_j(\pi^N(\underline{t}_*^N)))_{j=0}^\infty$ and the corresponding tract datum, $\xi_*^N \defeq \xi^\Phi((\eps_j(\pi^N(\underline{t}_*^N)))_{j=0}^\infty )$ we find that $\delta_j(\xi_*^N)=0$ for all $j < N$, proving the claim.
 \end{subproof}
 Now we prove the lemma as stated. Given $N\geq 1$, we can find a $\underline{t}^N$ such that $\pi^N(\underline{t}^N)$ satisfies the previous claim. Therefore, there is a sequence of $(\pi(\underline{t}^N))_{N=1}^\infty$ such that $\delta_j(\xi^N)=0$ for all $j<N$. Given that we are in a complete metric space, there must exist an element $\underline{\hat{t}} \in \Lambda^\mathbb{N}$ and a subsequence of $(\pi(\underline{t}^N))_{N=1}^\infty$, $(\pi(\underline{t}^{N_k}))_{k=0}^\infty$, that converges to $\underline{\hat{t}}.$ This corresponds to a convergent sequence $((\eps_j^{N_k})_{j=0}^\infty)_{k=0}^\infty \to (\eps_j)_{j=0}^\infty$. The further corresponding sequence of tract data must then converge $\xi^N\to \xi((\eps_j)_{j=0}^\infty)=\xi$. From the first claim, we have continuous dependence of $\delta_j$ on $\xi$ and so we have $\delta_j(\xi)=0$ for all $j\geq 0$ as claimed.
\end{proof}

\begin{proof}[Proof of Theorem~\ref{epseq}]
 Lemma~\ref{shoot} and the definition of $\delta_j$ have just shown that it is possible to find a sequence of $(\eps_j)_{j=0}^\infty$, hence a tract datum $\xi^\Phi$, such that $\delta_j(\xi^\Phi)=0$ for all $j\geq0$. From the definition of $\delta_j$ in~\eqref{deltaj}, this means $(F^{\xi^\Phi})^{-1}(\exp(\varphi(R_j))) = z_j \in \lbrace R_j -2 - 2\nu_0 + ti \colon -\pi/3 < t < \pi/3 \rbrace$, or, equivalently, $F^{\xi^\Phi}(z_j) = \exp(\varphi(R_j))$. This means that the situation illustrated by Figure~\ref{shootthemonkey} is achieved for all $j\geq0$ simultaneously.
 \end{proof}

\section{Estimating the growth of the counterexample: $\varphi$ and $\psi$}\label{sec:phivspsi}

\subsection{Estimating the growth of the model function}\label{sec:est}

From now on, we will only consider tract data for which the values of $(r_j)_{j=1}^{\infty}$ and $(R_j)_{j=0}^{\infty}$ 
 are given by the values above in~\ref{r--R}, for the functions $\varphi$ and $\Phi$ from the following Standing Assumption that holds throughout the rest of the paper.

We continue to fix $\Phi$ and $\varphi$ satisfying Standing Assumption~\ref{standingassumption}, as well as a tract datum $\xi^\Phi \in \Xi^\Phi$ satisfying the counterexample gate condition, whose existence is assured by Theorem~\ref{epseq}. 
 Our goal now is to estimate the growth of the function $F=F^{\xi}$. 
 
As discussed earlier at the end of Section~\ref{sec:growth}, it is intuitively clear that $\log \re F(z)$ is largest in comparison to $\log \re z$ when 
 $z$ is in the bottom two thirds of a wiggle $W_{j}$ with a small real part, i.e.\ $\re z$ is close to $r_{j+1}$. 
 By Theorem~\ref{growth} and the definition of the counterexample gate condition, for $z\in W_{j}$ the value of
 $\log \lvert F(z)\rvert$ is comparable to 
 $\varphi(R_{j})$. So in order to estimate the growth of $F$ in terms of $\re z$, we must estimate the size of $R_j$ in terms of $r_j$.  Recall the following from Definition~\ref{psi}:
 
 \[ \Psi(t) \defeq \frac{10\varphi^{-1}(\log t)}{\log t} \ \text{ and }  \ \psi(t) \defeq t^{1 + \Psi(t)}.\]
 
$\psi$ was originally introduced in~\eqref{phiandpsi} as the function that maps $r_j \mapsto R_j$ in a tract. In comparison with Definition~\ref{rR}, we use $\Psi = \widehat{\Phi}_{10}$ rather than $\widehat{\Phi}_9$ in order to make calculations and estimates neater.

We now prove a result about $\psi$ and estimate the order of growth of $F$.

\begin{prop}\label{prop:Psi} $R_j \leq \psi(r_j)$ for all sufficiently large $j$. 
\end{prop}
\begin{proof}
We must show that 
   \[ \log \frac{R_j}{r_j} \leq \Psi(r_j)\log r_j \]
  for sufficiently large $j$. 
By~\eqref{logr--R}, we have
\begin{align*}
  R_j = e r_j \exp(9\varphi^{-1}(\log(r_j)+1)) =
    r_j \exp(9\varphi^{-1}(\log(r_j)+1)+1). 
\end{align*}
So
 \begin{align*}
   \log\frac{R_j}{r_j} = 9 \varphi^{-1}(\log(r_j)+1)+1 \leq 10 \varphi^{-1}(\log(r_j)) = \Psi(r_j)\log r_j,
 \end{align*}
 where the final inequality holds for sufficiently large $r_j$ by Lemma~\ref{phishift}.
\end{proof} 

We can use this fact to estimate the growth of $F$ as follows. 
\begin{theorem}\label{idealmodel} Suppose that $\Phi$ and $\varphi$ satisfy Standing Assumption~\ref{standingassumption}.
%, and
%suppose furthermore that $\lim_{t \to \infty} \frac{\Psi(t)}{\Phi(t)} =0$, where $\Psi$ is defined by~\eqref{eqn:Psi}.
Let $\xi^\Phi$ satisfy the counterexample gate condition in Definition~\ref{philock}.

Then the function $F=F^{\xi}$ satisfies 
 \[ \log \re F(z) = O( \varphi(\psi(\re z)))\]
  as $\re z\to\infty$, where $\psi$ is defined by~\eqref{eqn:psi}. 
\end{theorem}
\begin{proof} 
%If $\varphi$ satisfies the standing assumption then so does $\varphi/2$. With this we can find a $\rho_0>0$ such that if $r_0 > \rho_0$ there also exists a set of data $\xi = (\varphi/2, r_0, (\eps_j)_{j=0}^\infty)$ that satisfies both the conditions of Theorem~\ref{epseq} and Proposition~\ref{upperboundgrowth}, which we will apply taking $\gamma = 2$. 
%
Let $j\geq 0$, and set $w_j\defeq F^{-1}(\exp(\varphi(R_j)))$. Recall that $w_j\in W_j$ by definition of the counterexample gate condition of Definition~\ref{philock}. Let  $z\in W_j$. By Theorem~\ref{growth} and Proposition~\ref{prop:Psi}, we have 
  \begin{align*} \log \re F(z) &\leq \log \lvert F(z) \rvert \leq
      C^2 \log \lvert F(w_j)\rvert = C^2\varphi(R_j) \\ &\leq C^2 \varphi(\psi(r_j)) \leq C^2\varphi(\psi(\re z)). \end{align*}
      
 It remains to show that this order of growth is not exceeded in any other part of the tract $T$ of $F$. Recalling Theorem~\ref{growth} we can see for $z \in U_{j+1}$, that is, the part of the tract after exiting $W_j$ and before the next gate so that we don't include $\log(1/\eps_{j+1})$ in our estimate. For us, $\re z < \tau_{j+1}-\nu_0$. Applying Theorem~\ref{growth} we see
 \begin{align*} \log \lvert F(z)\rvert &\leq C\left(\re z + \sum_{k=0}^{j}\log\left(\frac{1}{\eps_k} \right) \right) \\&\leq C \left( \re z + \varphi(R_j) \right) \leq C\varphi(\re z) \\ &= O\left(\varphi(\re z)\right) = O(\varphi (\psi(\re z) ) ).\end{align*} 
 
 For $z$ in the section of the tract not covered by Theorem~\ref{growth}, where \\ $\tau_{j+1} -\nu_0 < \re z < \tau_{j+1} +1$ and $\pi/3 < \im z < \pi$,  we note \begin{align*} \log \lvert F(z)\rvert &\leq C\left(R_{j+1} + \sum_{k=0}^{j+1}\log\left(\frac{1}{\eps_k} \right) \right) \\& \leq C\left(R_{j+1} + \sum_{k=0}^{j}2C\varphi(R_k) + 2C\varphi(R_{j+1}) \right) \\ &\leq C \left( R_{j+1} + (2C+1)\varphi(R_{j+1}) \right)  \\&= O(\varphi(R_{j+1})) = O(\varphi(\tau_{j+1} -\nu_0))= O(\varphi(\psi(\tau_{j+1} -\nu_0))).\end{align*} Where $\eps_{j+1} = a_{j+1}$ was used to provide an upper bound on $\log (1/\eps_{j+1}) $ again as per Proposition~\ref{range}. Lemma~\ref{phishift} in the penultimate equality. The result follows since $\log \re F(z)$ is defined for all $z \in T$ and never exceeds $\log \lvert F(z) \rvert$.
\end{proof}

Our goal now is to estimate $\varphi\circ \psi$ and thereby the order of growth of $F$, in a more manageable form. 

%An immediate corollary to this hypothesis is the following property of $\varphi(t)=t^{1 + \Phi(t)}$.
%\begin{cor}For any $\alpha>1$ and for any $M>1$ there exists some $t^*>1$ such that for all $t>t^*$, $\varphi(t/M) \leq \alpha \varphi(t)$.
%\end{cor}
%\begin{proof}
%
%\end{proof}

\begin{prop} For a given $M > 1$, there exists $\mu > 1$ such that for all $t > t_0$, \[ \varphi(Mt) < \mu\varphi(t). \]
\end{prop}
\begin{proof}
By letting $\mu = M^{1 + \varphi(t_0)}$ we can see that

\begin{align*} \varphi(Mt) = M^{1 + \varphi(Mt)}t^{1 + \varphi(Mt)} &\leq M^{1 + \varphi(t_0)}t^{1 + \varphi(Mt)} \\&\leq M^{1 + \varphi(t_0)}t^{1 + \varphi(t)}  = \mu \varphi(t).\end{align*}
\end{proof}

Now we can estimate $\varphi\circ\psi$ as follows, under the assumption that $\psi$ grows more slowly than $\varphi$. 

\begin{prop}\label{prop:phipsiestimate}
 Suppose that $\Phi$ is as in Standing Assumption~\ref{updatedassumption}. Let $\Psi$  and $\psi$ be defined by~\eqref{eqn:Psi} and \eqref{eqn:psi},  and suppose $\Psi(t)/\Phi(t)\to 0$ as $t\to\infty$. Then for every $\gamma>1$, 
   \[ \varphi(\psi(t)) = O\left( t^{1 + \gamma \Phi(t)}\right). \] 
\end{prop}
\begin{proof}
 We have 
   \begin{equation}\label{eqn:madexp} \varphi(\psi(t)) = t^{1 + \Psi(t) + \Phi(\psi(t)) + \Phi(\psi(t))\Psi(t)}
   	\end{equation}
   by definition. 
 Let $\gamma >1$ and write $\gamma = 1 + \eps$. By~\ref{Philimits}, Standing Assumption~\eqref{standingassumption}, and our assumption on $\Psi/\Phi$, we can write the following for sufficiently large $t$ 
 \begin{align*}
 	\Psi(t) &\leq (\eps/3)\Phi(t), \\
 	\Phi(\psi(t)) &\leq (1 + \eps/3)\Phi(t).\\
 \end{align*}
 Using these we can show that, for large enough $t$, the exponent in~\ref{eqn:madexp} can be bounded above by
 \begin{align*}
 1	+ \Psi(t) + \Phi(\psi(t)) + \Phi(\psi(t))\Psi(t) &\leq 1 + \Phi(t) \left( 1 + \frac{2\eps}{3} + \frac{\eps}{3}\left( 1 + \frac{\eps}{3}\right)\Phi(t)\right) \\
 													&\leq 1 + \eps\Phi(t).
 \end{align*} The claim follows. 
\end{proof}

This new limit, $\Psi(t)/\Phi(t) \to 0$ as $t \to \infty$, is one of interest as $\Psi$ was defined in terms of $\Psi$ (implicitly) in Definition~\ref{psi}. An obvious question to ask is: What happens if $\Psi(t)/\Phi(t) \to \infty$ as $t \to \infty$ instead? Suppose we have $\Phi$ satisfying the same standing assumption, that $\Psi$ is defined as in Definition~\ref{psi}, and that there is some $f \in \classB$ that satisfies $\log \log \abs{f(z)} = O(\varphi(\abs{z}))$.  It is suspected, from similar heuristic reasoning as seen in Section~\ref{heuristic} that, if $\Psi(t)/\Phi(t) \to \infty$, we must satisfy the strong Eremenko conjecture. It is hoped that an argument showing that a corresponding logarithmic transform $F$ exhibits a \emph{head start condition}~\cite[Section 4, Section 5]{rrrs} which is a key feature in the proof of~\cite[Theorem~1.2]{rrrs}. \\

Finally, we can write down the following estimate on $\Psi$, which is 
 a consequence of Proposition~\ref{prop:phiinv}. 
\begin{prop}\label{prop:PsiestimatefromPhi}
  Suppose that $\Phi$ is as in Standing Assumption~\ref{updatedassumption}. Let $\Psi$ be defined as in~\eqref{eqn:Psi}. Then for every $M>1$, 
    \[ \Psi(t) = O((\log t)^{-\Phi(\log t)/M}). \]
\end{prop} 
\begin{proof}
 This follows directly from Proposition~\ref{prop:phiinv} and the definition of $\Psi$ in~\eqref{eqn:Psi}.
\end{proof} 

\section{Proof of the main theorem}\label{secn:theoremproof}
In this section, we first use the estimates that we have produced to prove a version of Theorem~\ref{paracxple} 
in logarithmic coordinates. Then we use methods of \cite{bishb}, following the approach of \cite[Proposition 11.1]{tania}, to show that from our tracts in the previous sections, we can deduce the existence of a transcendental entire function $g$ in class $\classB$ that is a counterexample to the strong Eremenko Conjecture and has the same order of growth as our constructed $F$. Here, we are going to be deducing the existence of functions with orders of growth 
\begin{equation}\label{thetafunorder}
\log \log \abs{f(z)} = O\left( \left(\log \abs{z} \right)^{1 + \Theta(\log \abs{z})}\right)
\end{equation} 
where $\Theta$ will satisfy the same assumptions as $\Phi$, the Standing Assumption~\ref{standingassumption}. We then build a model according to our construction thus far with a corresponding $\Phi$, related to the desired $\Theta$, so that we have $F \in \Blog^p$ such that $\log \re F(z) = O\left( (\re z)^{1 + \Phi(\re z)}\right)$. Once we apply our chosen approximation method of Bishop's quasiconformal folding (see~\cite{bishb}, discussed in Section~\ref{sec:bishopmodel}) this will provide a function $f$ in class $\classB$ satisfying~\eqref{thetafunorder}. This proves Theorem~\ref{paracxple}. Finally, we establish that, for $\alpha>0$, $\Theta(t) = 1/(\log \log t)^\alpha$ satisfies the hypotheses of the theorem.

\section{Building a logarithmic model}

\begin{theorem}\label{thm:logmodel}
 Suppose that $\Theta\colon [t_0,\infty)\to (0,\infty)$ satisfies the hypotheses of Theorem~\ref{paracxple}. Then there exist
   $\delta>0$ and a tract datum $\xi\in \Xi_{\mathcal{C}}$ such that the function $F=F^{\xi}$ satisfies
   \[ \log \re F(z) = O( (\re z)^{1+(1-\delta)\Theta(\re z)}). \] 
\end{theorem}
\begin{proof} Define
 $\Phi\colon [t_0,\infty) \to (0,\infty)$ by 
  $\Phi(t) \defeq (1-\eps)\cdot \Theta(t)$, where $0<\eps<1$. We claim that 
  $\Phi$ can be extended to $\Phi\colon [1,\infty)\to (0,\infty)$ satisfying Standing Assumption~\ref{standingassumption} and Proposition~\ref{Philimits}. Indeed, clearly $\Phi$ satisfies these properties for $t\geq t_0$, and it is not difficult to 
  extend $\Phi$ as required. 
  
  Now let $\Psi$ be defined as in~\eqref{eqn:Psi}. We claim that the hypotheses on $\Theta$ imply that 
    \[ \frac{\Psi(t)}{\Phi(t)} \to 0  \text{ as } t \to \infty\]
    as long as $\eps$ is chosen sufficiently small and $M$ is sufficiently close to $1$. 
    Indeed, this follows from~\eqref{eqn:deryck} and Proposition~\ref{prop:PsiestimatefromPhi}.
  
   Now we can apply Theorem~\ref{idealmodel} to $\Phi$ to obtain the desired $\xi$ and $F$, which satisfies
      \begin{align*} \log \re F(z) &= O(\psi(\varphi(\re z))) \\&= O( (\re z)^{1 + \gamma \Phi(\re z)}) = O( (\re z)^{1+ (1-\delta)\Theta(\re z)} \end{align*} 
      by Proposition~\ref{prop:phipsiestimate}, where $\gamma < 1/(1-\eps)$ and $\delta$ is sufficiently small. 
\end{proof}

\section{Proof of the main theorem}\label{sec:modelbuild}
%
%We will define the set $V_\rho \defeq F^{-1}(\re z \leq \rho) \subset T$ and let $\mathcal{D}$ be the union of $D \defeq T\setminus V_1$ and its set of $2\pi i \mathbb{Z}$-translates.

\begin{proof}[Proof of Theorem 1.1]
Let $\xi$ be the tract datum from Theorem~\ref{thm:logmodel}, set $G\defeq F^{\xi}$, and let $T$ be the domain of
definition of $G$. Define 
 \[ \Omega = \exp(T) \quad \text{ and } \quad g\colon w \mapsto \exp(G(\log(w))). \]
 Then $(\Omega, G)$ is a disjoint-type model. Let $f$ be the entire function given by Theorem~\ref{Iconj}, and let $p$, $K$ and 
 $U$ be as in that theorem. 
 
 \begin{claim} $I(f)$ contains no curve to $\infty$.
\end{claim}
\begin{subproof} It follows from the definition that $J(g) = \exp( J(\tilde{G}))$, where $\tilde{G}$ is obtained from $G$
as in Proposition~\ref{prop:Blogp}, and similarly $I(g) = \exp( I(\tilde{G}))$. We know that $J(\tilde{G})\supset I(\tilde{G})$ 
contains no curve to infinity, because $\xi$ is a tract datum in $\Xi_{\mathcal{C}}$, the subclass of tract data that provide counterexamples,  and by Theorem~\ref{nocurve}. 
 We have $p(I(g)) = I(f)$ and hence $I(f)$ and $I(g)$ are homeomorphic. The claim follows. 
\end{subproof}

It remains to estimate the growth of $f$. Let $\zeta \in \C$ where $\lvert f(\zeta)\rvert$ is sufficiently large to ensure
 that $z\in p(U)$. Let $\omega = p^{-1}(\zeta)\in U$, and choose $w\in T$ with $\exp(w)=\omega$. Then, 
 by Theorem~\ref{thm:logmodel}, 
 \begin{align*}
\log \log \abs{f(\zeta)}  &= \log \log \abs{p(g(\omega))} \\ 
					  &\leq \log \log \abs{g(\omega)}^K \\
					  &= O ( \log \log \abs{g(\omega)}) \\
					  &= O( \log \re G(w)) \\ 
					  &= O( (\re w)^{1 + (1-\delta)\cdot \Theta(\re w)} ). \end{align*}
			
 Since $\lvert \zeta\rvert^{1/K} \leq \lvert \omega\rvert \leq \lvert \zeta\rvert^K$, we have 
   $(\log \lvert \zeta\rvert)/K \leq \re w \leq K\cdot \log \lvert \zeta\rvert$. Now 
     $\Theta(t)/ \Theta( Kt )\to 1$ as $t\to\infty$, by~\ref{Philimits}.
     So 
       \[ (1-\delta)\cdot \Theta( \re w) \leq (1-\delta)\cdot \Theta(\log \lvert \zeta\rvert/K) \leq \Theta(\log \lvert \zeta\rvert) \] 
    when $\lvert\zeta\rvert$ is sufficiently large. Hence 
\begin{align*}
 \log \log \abs{f(\zeta)}					  &=
      O( (\log \lvert \zeta\rvert )^{1 + (1-\delta)\cdot \Theta(\re w)} ) \\
         &= O( (\log \lvert \zeta\rvert)^{1 + \Theta(\log\lvert \zeta\rvert)}), \end{align*}
  as claimed. 
\end{proof}

\section{Examples}\label{loglogcheck}
As per the claim in the introduction, we now verify that the work undertaken does indeed hold for at least one elementary choice of $\Theta$.
\begin{prop}
Let $\alpha > 0$. There exists a $t_0 >  e$ so that the function $\Theta \colon [t_0, \infty) \to (0,\infty)$ defined by \[\Theta(t) \defeq \frac{1}{\left(\log\log t\right)^\alpha}\] and the associated $\theta(t) \defeq t^{1 + \Theta(t)}$ satisfy the properties of Theorem~\ref{paracxple}. That is,
\begin{enumerate}[label=(\alph*)] 
\item $\Theta(t) \to 0$ as $t \to \infty;$
\item $t \mapsto \Theta(t)\log t$ is strictly increasing;
\item $\Theta(t)\log t \to \infty$ as $t \to \infty;$
\item $\Theta(t^2)/\Theta(t) \to 1$ as $t \to \infty;$
%\item $\frac{\Psi(t)}{\Theta(t)} \to 0$ as $t \to \infty$ where $\Psi(t) \defeq \frac{9\theta^{-1}(\log t)}{\log t}$
\item  for every $0<\beta< 1$,  $(\log t)^{\beta \Theta(\log t)}\Theta(t) \to \infty$ as $t\to \infty.$
\end{enumerate}
\end{prop}
\begin{proof}
The first and third properties follow from the standard properties of the logarithm.

Note that \[ \Theta'(t) = -\frac{\alpha}{t \log t (\log\log t)^{1 + \alpha}}. \]

We calculate the following
\begin{align*}
\left( \Theta(t) \log t \right)' &= \frac{\Theta(t)}{t} + \Theta'(t)\log t  \\ &= \frac{1}{t(\log\log t)^\alpha} - \frac{\alpha}{t (\log\log t)^{1 + \alpha}} \\ &= \frac{1}{t(\log\log t)^\alpha}\left( 1 - \frac{\alpha}{\log\log t}\right).
\end{align*} If we define $t_0 \defeq \exp(\exp(\alpha))$, then the derivative calculated above is found to be positive for $t>t_0$ which gives the second property.

For the next property we check
\begin{align*} \frac{\Theta(t^2)}{\Theta(t)} &= \left(\frac{\log\log(t)}{\log\log (t^2)}\right)^\alpha \\ &= \left(\frac{\log\log t }{\log\log t + \log 2}\right)^\alpha \\ &= \left(1 - \frac{\log 2}{\log\log t + \log2}\right)^\alpha.
\end{align*} This tends to $1$ as $t \to \infty$.

Finally, let $0<\beta<1$ and note that \begin{align*} (\log t)^{\beta \Theta(\log t)}\Theta(t) &= \frac{(\log t)^{\beta\Theta(\log t)}}{(\log\log t)^\alpha}.
\end{align*} If we set $x = \log t$ this simplifies to \begin{align*} \frac{x^{\beta\Theta(x)}}{(\log x)^\alpha} &= \exp\left(\frac{\beta\log x}{(\log\log x)^\alpha}- \alpha \log\log x \right).\end{align*} Since $\log\log x = o((\log x)^K)$ for any $K>0$, \[ \frac{\beta\log x}{(\log\log x)^\alpha} - \alpha \log\log x  \to \infty \]  as $x \to \infty$.  With this the final property holds, proving the proposition.
\end{proof}

\bibliography{slowgrow} 
\bibliographystyle{alpha}

\end{document}